\documentclass[11pt]{article}

\usepackage{amssymb,amsmath,amscd,amsthm,xy,enumitem,mathrsfs,pstricks,float,graphicx,mathabx}
\xyoption{all}

\def\tauetacond#1{$(\cO_{#1})$}   
 
\newtheorem{thm}{Theorem}[section]
\newtheorem{prp}[thm]{Proposition}
\newtheorem{lmm}[thm]{Lemma}

\newtheorem{crl}[thm]{Corollary}

\theoremstyle{definition}

\theoremstyle{remark}
\newtheorem{rmk}[thm]{Remark}

\numberwithin{equation}{section}

\def\lra{\longrightarrow}
\def\Lra{\Longrightarrow}

\def\BE#1{\begin{equation}\label{#1}}
\def\EE{\end{equation}}
\def\lr#1{\langle#1\rangle}

\def\blr#1{\big\langle#1\big\rangle}
\def\ti#1{\tilde{#1}}
\def\wt#1{\widetilde{#1}}
\def\ov#1{\overline{#1}}
\def\eref#1{(\ref{#1})}
\def\tn#1{\textnormal{#1}}
\def\sf#1{\textsf{#1}}

\def\sm#1{\begin{small}#1\end{small}}
\def\mr#1{\mathring{#1}}
\def\wch#1{\widecheck{#1}}

\def\De{\Delta}
\def\Ga{\Gamma}
\def\La{\Lambda}
\def\Om{\Omega}
\def\Si{\Sigma}
\def\Th{\Theta}

\def\al{\alpha}
\def\be{\beta}
\def\de{\delta}

\def\ga{\gamma}

\def\ka{\kappa}
\def\la{\lambda}

\def\om{\omega}

\def\th{\theta}
\def\ve{\varepsilon}
\def\ups{\upsilon}

\def\vp{\varpi}

\def\cB{\mathcal B}

\def\C{\mathbb C}
\def\cC{\mathcal C}

\def\bI{\mathbb I}
\def\fI{\mathfrak i}
\def\cJ{\mathcal J}

\def\cK{\mathcal K}

\def\cM{\mathcal M}

\def\fM{\mathfrak M}

\def\cN{\mathcal N}
\def\cO{\mathcal O}
\def\Q{\mathbb Q}
\def\P{\mathbb P}
\def\cP{\mathcal P}
\def\R{\mathbb{R}}
\def\fR{\mathfrak R}

\def\cU{\mathcal U}

\def\Y{\mathbf Y}
\def\Z{\mathbb{Z}}

\def\a{\mathbf a}

\def\fd{\mathfrak d}
\def\bh{\mathbf h}

\def\fs{\mathfrak s}

\def\y{\mathbf y}

\def\tnd{\textnormal{d}}

\def\eff{\tn{eff}}

\def\ev{\tn{ev}}

\def\id{\textnormal{id}}
\def\Im{\tn{Im}}

\def\ind{\textnormal{ind}}

\def\PD{\tn{PD}}
\def\PSL{\tn{PSL}}

\def\top{\textnormal{top}}
\def\vir{\tn{vir}}

\def\0{\mathbf 0}
\def\1{\mathbf 1}

\def\dbar{\bar\partial}
\def\prt{\partial}
\def\eset{\emptyset}
\def\i{\infty}

\def\bu{\bullet}

\begin{document}

\title{Enumeration of Real Curves in $\C\P^{2n-1}$ and\\
a WDVV Relation for Real Gromov-Witten Invariants}
\author{Penka Georgieva and 
Aleksey Zinger\thanks{Partially supported by  NSF grant DMS 0846978}}
\date{\today}
\maketitle

\begin{abstract}
\noindent
We establish a homology relation for the Deligne-Mumford moduli spaces of real curves
which lifts to a WDVV-type relation for
a class of real Gromov-Witten invariants of real symplectic manifolds;
we also obtain a vanishing theorem for these invariants.
For many real symplectic manifolds, these results reduce all genus~0 real invariants 
with conjugate pairs of constraints to genus~0 invariants with a single conjugate 
pair of constraints.
In particular, we give a complete recursion for counts of real rational curves 
in odd-dimensional projective spaces with conjugate pairs of constraints and 
specify all cases when they are nonzero and thus provide non-trivial lower bounds
in high-dimensional real algebraic geometry.
We also show that the real invariants of the three-dimensional projective space
with conjugate point constraints are congruent to their complex analogues modulo~4.
\end{abstract}

\tableofcontents

\section{Introduction}
\label{intro_sec}

\noindent
The classical problem of enumerating (complex) rational curves in a complex projective space~$\P^n$ 
is solved in \cite{KM,RT} using the WDVV relation of Gromov-Witten theory.
Over the past decade, significant progress has been made in real enumerative geometry 
and real Gromov-Witten theory.
Invariant signed counts of real rational curves with point constraints
in real surfaces and in many real threefolds are defined in~\cite{Wel1}
and~\cite{Wel2}, respectively.
An approach to interpreting these counts in the style of Gromov-Witten theory,
i.e.~as counts of parametrizations of such curves, is presented in \cite{Cho,Sol}. 
Signed counts of real curves with conjugate pairs of arbitrary (not necessarily point)
constraints in arbitrary dimensions are defined in~\cite{Ge2} and extended
to more general settings in~\cite{Teh}.
Two different WDVV-type relations for the real Gromov-Witten invariants 
of real surfaces as defined in \cite{Cho,Sol}, along with the ideas behind them, 
are stated in~\cite{Sol2};
they yield complete recursions for counts of real rational curves in~$\P^2$ 
as defined in~\cite{Wel1}.
Other recursions for counts of real curves in some real surfaces have 
since been established by completely different methods in \cite{IKS09,ABL,IKS13a,IKS13b}.\\

\noindent
In this paper, we establish a homology relation between geometric classes
on the Deligne-Mumford moduli space $\R\ov\cM_{0,3}$ of real genus~0 curves
with 3~conjugate pairs of marked points    
and use it to obtain a WDVV-type relation for the real Gromov-Witten invariants 
of \cite{Ge2,Teh}; see Proposition~\ref{DMrelR_prp} and Theorem~\ref{main2_thm}.
This relation yields a complete recursion for counts of real rational curves with 
conjugate pairs of arbitrary constraints in~$\P^{2n-1}$;
see Theorem~\ref{main_thm} and Corollary~\ref{main0_crl}.
It is sufficiently simple to characterize the cases when these invariants are nonzero
and thus the existence of real rational curves passing through specified types of constraints
is guaranteed; see Corollary~\ref{oddnums_crl}.
We also show that the real genus~0 Gromov-Witten invariants of~$\P^3$
with conjugate pairs of point constraints are congruent to their complex analogues
modulo~4, as expected for the real curve counts of~\cite{Wel1,Wel2},
and that this congruence does not persist in higher dimensions or with other types
of constraints; see Corollary~\ref{main_crl} and the paragraph right after~it.\\

\noindent
Each odd-dimensional projective space~$\P^{2n-1}$ has two standard anti-holomorphic involutions
(automorphisms of order~2):
\begin{alignat}{2}
\label{taudfn_e}
\tau_{2n}\!:\P^{2n-1}&\lra\P^{2n-1}, &\quad
[z_1,\ldots,z_{2n}]&\lra [\bar{z}_2,\bar{z}_1,\ldots,\bar{z}_{2n},\bar{z}_{2n-1}],\\
\label{etadfn_e}
\eta_{2n}\!:\P^{2n-1}&\lra\P^{2n-1}, &\quad
[z_1,\ldots,z_{2n}]&\lra [-\bar{z}_2,\bar{z}_1,\ldots,-\bar{z}_{2n},\bar{z}_{2n-1}].
\end{alignat}
The fixed locus of the first involution is~$\R\P^{2n-1}$, while 
the fixed locus of the second involution is empty.
Let 
$$\tau\!=\!\tau_2,\,\eta\!=\!\eta_2: \P^1\lra\P^1\,.$$
For $\phi\!=\!\tau_{2n},\eta_{2n}$ and $c\!=\!\tau,\eta$, 
a map $u\!:\P^1\!\lra\!\P^{2n-1}$ is \sf{$(\phi,c)$-real} if
$u\!\circ\!c\!=\!\phi\!\circ\!u$.
For $k\!\in\!\Z^{\ge0}$, a \sf{$k$-marked $(\phi,c)$-real map} is a tuple
$$\big(u,(z_1^+,z_1^-),\ldots,(z_k^+,z_k^-)\big),$$
where $z_1^+,z_1^-,\ldots,z_k^+,z_k^-\!\in\!\P^1$ are distinct points
with $z_i^+\!=\!c(z_i^-)$  and $u$ is a $(\phi,c)$-real map.
Such a tuple is \sf{$c$-equivalent} to another $k$-marked $(\phi,c)$-real map
$$\big(u',(z_1'^+,z_1'^-),\ldots,(z_k'^+,z_k'^-)\big)$$
if there exists a biholomorphic map $h\!:\P^1\!\lra\!\P^1$ such that 
$$h\!\circ\!c=c\!\circ\!h, \qquad u'\!=\!u\!\circ\!h, \quad\hbox{and}\quad
z_i^{\pm}\!=\!h(z_i'^{\pm})~~\forall~i\!=\!1,\ldots,k.$$

\vspace{.2in}

\noindent
If in addition $d\!\in\!\Z^+$, denote~by 
$$\fM_{0,k}(\P^{2n-1},d)^{\phi,c}\subset\ov\fM_{0,k}(\P^{2n-1},d)^{\phi,c} $$ 
the moduli space of $c$-equivalence classes of
$k$-marked degree~$d$  holomorphic $(\phi,c)$-real maps
and its natural compactification consisting of stable real maps from nodal domains.
As in \cite[Section~3]{Teh}, let
\BE{gluedsp_e}\ov\fM_{0,k}(\P^{2n-1},d)^{\phi} \equiv
\ov\fM_{0,k}(\P^{2n-1},d)^{\phi,\tau}\cup\ov\fM_{0,k}(\P^{2n-1},d)^{\phi,\eta}\EE
be the space obtained by identifying the two moduli spaces on the right-hand side
 along their common boundary.
The glued space has no codimension~1 boundary.\\

\noindent
By \cite[Lemma~1.9]{Teh} and its~proof,
\begin{gather*}
\ov\fM_{0,k}(\P^{2n-1},d)^{\tau_{2n},\eta}=\eset~~\forall\,d\!\not\in\!2\Z\,,\\
\ov\fM_{0,k}(\P^{2n-1},d)^{\eta_{2n},\eta}=\eset~~\forall\,d\!\in\!2\Z\,, \quad
\ov\fM_{0,k}(\P^{2n-1},d)^{\eta_{2n},\tau}=\eset~~\forall\,d\in\Z\,.
\end{gather*}
By \cite[Theorem~6.5]{Ge2}, $\ov\fM_{0,k}(\P^{2n-1},d)^{\tau_{2n},\tau}$
is orientable for every $d\!\in\!\Z$.
By \cite[Section~5.2]{Teh}, the~spaces
$$\ov\fM_{0,k}(\P^{2n-1},d)^{\tau_{2n},\eta}~~\hbox{with}~d\!\in\!2\Z
\quad\hbox{and}\quad
\ov\fM_{0,k}(\P^{2n-1},d)^{\eta_{2n},\eta}~~\hbox{with}~d\!\not\in\!2\Z$$
are orientable as well.
If $\phi\!=\!\tau_{2n}$ and $d\!\in\!2\Z$,
the orientations on the two moduli spaces on the right-hand side of~\eref{gluedsp_e}
can be chosen so that they extend across the common boundary; see \cite[Proposition~5.5]{Teh}.
In the remaining three cases, at most one of the spaces on the right-hand side of~\eref{gluedsp_e}
is not empty. 
Thus, the glued moduli space~\eref{gluedsp_e} is orientable and carries a fundamental class.\\

\noindent
The glued compactified moduli spaces come with natural evaluation maps
$$\ev_i\!: \ov\fM_{0,k}(\P^{2n-1},d)^{\phi}\lra\P^{2n-1},
\qquad 
\big[u,(z_1^+,z_1^-),\ldots,(z_k^+,z_k^-)\big]\lra u(z_i^+).$$
For $c_1,\ldots,c_k\!\in\!\Z^+$, we define
\BE{realNdfn_e} 
\lr{c_1,\ldots,c_k}_d^{\phi}=\int_{\ov\fM_{0,k}(\P^{2n-1},d)^{\phi}}
\ev_1^*H^{c_1}\,\ldots\,\ev_k^*H^{c_k}\in\Z\,,\EE
where $H\!\in\!H^2(\P^{2n-1})$ is the hyperplane class.
For dimensional reasons,
\BE{codimcond_e} \lr{c_1,\ldots,c_k}_d^{\phi}\neq0
\qquad\Lra\qquad c_1+\ldots+c_k=n(d\!+\!1) -2+k\,. \EE
Similarly to \cite[Lemma~10.1]{RT}, the numbers~\eref{realNdfn_e} are enumerative counts
of real curves in~$\P^{2n-1}$, i.e.~of curves preserved by~$\phi$,
but now with some sign.
They are invariant under the permutations of the insertions and satisfy 
the usual divisor relation,
\BE{RdivRel_e}\lr{c_1,\ldots,c_k,1}_d^{\phi}=d\,\lr{c_1,\ldots,c_k}_d^{\phi}\,.\EE
The latter holds because the fiber of the forgetful morphism
$$\ov\fM_{0,k+1}(\P^{2n-1},d)^{\phi}\lra\ov\fM_{0,k}(\P^{2n-1},d)^{\phi}$$
is oriented by $z_{k+1}^+$
 and every degree~$d$ curve in~$\P^{2n-1}$ meets 
a generic hyperplane in $d$~points.\\

\noindent
By \cite[Theorem~1.10]{Teh} and \cite[Remark~1.11]{Teh}, the numbers~\eref{realNdfn_e} with 
$\phi\!=\!\tau_{2n},\eta_{2n}$ vanish if either $d$ or any $c_i$ is even;
see also Corollary~\ref{inveqCI_crl}\ref{CIvan_it}.
By \cite[Remark~1.11]{Teh} and Corollary~\ref{inveqCI_crl}\ref{CIred_it},
\BE{taueta_e}\lr{c_1,\ldots,c_k}_d^{\tau_{2n}}=\pm \lr{c_1,\ldots,c_k}_d^{\eta_{2n}}\,;\EE
the sign depends on the orientations of 
$\ov\fM_{0,k}(\P^{2n-1},d)^{\tau_{2n}}$ and $\ov\fM_{0,k}(\P^{2n-1},d)^{\eta_{2n}}$.
Systems of such orientations, compatible with the recursion of Theorem~\ref{main_thm}
for~$\P^{2n-1}$ and the WDVV-type relation of Theorem~\ref{main2_thm} for more general 
real symplectic manifolds, are described in Section~\ref{mainthm_sec}.
They ensure a fixed sign in~\eref{taueta_e} and can be specified by 
choosing the sign of the $d\!=\!1$ numbers in~\eref{realNdfn_e}. 

\begin{rmk}
The orientations for the $\tau_{4n}$ and~$\eta_{4n}$
moduli spaces are determined by a spin structure on~$\R\P^{4n-1}$
and a real square root of the canonical line bundle~$\cK_{\P^{4n-1}}$
of~$\P^{4n-1}$, respectively.
On the other hand, $\R\P^{4n+1}$ does not admit a spin structure, while 
$\cK_{\P^{4n+1}}$ does not admit a real square root.
A relatively spin structure on~$\R\P^{4n+1}$
does not provide a system of orientations compatible with 
the recursion of Theorem~\ref{main_thm},
because such a system is not compatible with smoothing a conjugate pair of nodes,
as needed for the statement of Lemma~\ref{Rgluing0_lmm};
see Remark~\ref{relspin_rmk} for more details.
\end{rmk}

\noindent
For any $d,c_1,\ldots,c_k\!\in\!\Z^+$, let
$$\blr{c_1,\ldots,c_k}_d^{\P^{2n-1}}=\int_{\ov\fM_{0,k}(\P^{2n-1},d)}
\ev_1^*H^{c_1}\,\ldots\,\ev_k^*H^{c_k}\in\Z^{\ge0}\,,$$
where $\ov\fM_{0,k}(\P^{2n-1},d)$ is the usual moduli space of stable (complex) 
$k$-marked genus~0 degree~$d$ holomorphic maps to~$\P^{2n-1}$,
denote the (complex) genus~0 Gromov-Witten invariants of~$\P^{2n-1}$;
they are computed in \cite[Theorem~10.4]{RT}.
Finally, if $c_1,\ldots,c_k\!\in\!\Z$ and $I\!\subset\!\{1,\ldots,k\}$,
let $c_I$ denote a tuple with the entries $c_i$ with $i\!\in\!I$,
in some order.

\begin{thm}\label{main_thm}
Let $\phi\!=\!\tau_{2n},\eta_{2n}$ and $d,k,n,c,c_1,\ldots,c_k\!\in\!\Z^+$.
If $k\!\ge\!2$ and $c_1,\ldots,c_k\!\not\in\!2\Z$,
\begin{equation*}\begin{split}
\blr{c_1,c_2\!+\!2c,c_3,\ldots,c_k}_d^{\phi}
-\blr{c_1\!+\!2c,c_2,c_3,\ldots,c_k}_d^{\phi}
=\sum_{\begin{subarray}{c}2d_1+d_2=d\\ d_1,d_2\ge1 \end{subarray}}
\sum_{I\sqcup J=\{3,\ldots,k\}}
\sum_{\begin{subarray}{c}2i+j=2n-1\\ i,j\ge1 \end{subarray}}\!\!\!\!\!\!
2^{|I|}\Bigg(\qquad&\\
\blr{2c,c_1,c_I,2i}_{d_1}^{\P^{2n-1}}\!\blr{c_2,c_J,j}_{d_2}^{\phi}
-\blr{2c,c_2,c_I,2i}_{d_1}^{\P^{2n-1}}\!\blr{c_1,c_J,j}_{d_2}^{\phi}&\Bigg).
\end{split}\end{equation*}
\end{thm}

\begin{crl}\label{main0_crl}
Let $\phi\!=\!\tau_{2n},\eta_{2n}$ and $d,k,n,c_1,\ldots,c_k\!\in\!\Z^+$.
If $d\!\in\!2\Z$ or $c_i\!\in\!2\Z$ for some~$i$,
$$\blr{c_1,c_2,\ldots,c_k}_d^{\phi}=0.$$
If $k\!\ge\!2$ and $c_1,\ldots,c_k\!\not\in\!2\Z$,
\begin{equation*}\begin{split}
\blr{c_1,c_2,c_3,\ldots,c_k}_d^{\phi}
=d\blr{c_1\!+\!c_2\!-\!1,c_3,\ldots,c_k}_d^{\phi}
+\sum_{\begin{subarray}{c}2d_1+d_2=d\\ d_1,d_2\ge1 \end{subarray}}
\sum_{I\sqcup J=\{3,\ldots,k\}}
\sum_{\begin{subarray}{c}2i+j=2n-1\\ i,j\ge1 \end{subarray}}\!\!\!\!\!\!
2^{|I|}\Bigg(\qquad&\\
d_2\blr{c_1\!-\!1,c_2,c_I,2i}_{d_1}^{\P^{2n-1}}\!\blr{c_J,j}_{d_2}^{\phi}
-d_1\blr{c_1\!-\!1,c_I,2i}_{d_1}^{\P^{2n-1}}\!\blr{c_2,c_J,j}_{d_2}^{\phi}&\Bigg).
\end{split}\end{equation*}
\end{crl} 

\begin{crl}\label{oddnums_crl}
Let $\phi\!=\!\tau_{2n},\eta_{2n}$ and $d,k,n,c_1,\ldots,c_k\!\in\!\Z^+$ with 
$$c_1\!+\!\ldots\!+\!c_k= n(d\!+\!1)-2+k \qquad\hbox{and}\qquad
c_1,\ldots,c_k\le2n\!-\!1.$$
\begin{enumerate}[label=(\arabic*),leftmargin=*]
\item If $d,c_1,\ldots,c_k$ are odd, then
so is $\lr{c_1,c_2,\ldots,c_k}_d^{\phi}$.
\item\label{vancrl_it} 
The signed number $\lr{c_1,c_2,\ldots,c_k}_d^{\phi}$ of degree~$d$ real curves 
in~$\P^{2n-1}$ passing through general complex linear subspaces of 
codimensions~$c_1,\ldots,c_k$ is zero if and only if either $d\!\in\!2\Z$
or $c_i\!\in\!2\Z$ for some~$i$.\\
\end{enumerate}
\end{crl} 

\noindent
The formula of Theorem~\ref{main_thm}, which is a special case of Theorem~\ref{main2_thm}, 
can be seen as a real version of \cite[Theorem~1]{LP}.
Along with~\eref{RdivRel_e}, it immediately implies the recursion of Corollary~\ref{main0_crl}.
The vanishing statement in this corrollary is the $\ell\!=\!0$ case of  
Corollary~\ref{inveqCI_crl}\ref{CIvan_it}.
Corollary~\ref{main0_crl}  reduces all numbers $\lr{c_1,c_2,\ldots,c_k}_d^{\phi}$, 
with $\phi\!=\!\tau_{2n},\eta_{2n}$, to the single number
$\blr{2n\!-\!1}_1^{\phi}$,
i.e.~the number of $\phi$-real lines through a non-real point in~$\P^{2n-1}$.
The absolute value of this number is of course~1,
and we can choose a system of orientations so that $\blr{2n\!-\!1}_1^{\phi}\!=\!1$.
Taking $d\!=\!1$ in Corollary~\ref{main0_crl}, we~obtain
$$\blr{c_1,\ldots,c_k}_1^{\phi}=\blr{2n\!-\!1}^{\phi}_1=1$$
whenever $c_1,\ldots,c_k\!\in\!\Z^+$ are odd and $c_1\!+\!\ldots\!+\!c_k\!=\!2n\!-\!2\!+\!k$;
this conclusion agrees with \cite[Example~6.3]{Teh}.
Some other numbers obtained from Corollary~\ref{main0_crl} are shown in 
Tables~\ref{P3nums_tbl} and~\ref{Pnnums_tbl}; 
the degree~3 and~5 numbers in the former agree with~\cite{Teh}.\\

\noindent
Corollary~\ref{oddnums_crl} is deduced from Corollary~\ref{main0_crl} 
in Section~\ref{quant_sec}.
It can be equivalently viewed as a statement about the parity of 
the usual counts of genus~0 curves in~$\P^{2n-1}$ with certain types of constraints
(they must come in pairs of the same codimension).
The standard WDVV recursion for counts of complex curves is not closed under 
the relevant restriction on the constraints.
We do not see how to recover Corollary~\ref{oddnums_crl} directly from~it. \\

\noindent
In the case $\P^{2n-1}\!=\!\P^3$, the only interesting non-real constraints for 
the real genus~0 counts are points.
Let
\BE{signcorr_e}N_d^{\R}
=(-1)^{\frac{d-1}{2}}\blr{\underset{d}{\underbrace{3,\ldots,3}}}_d^{\tau_{2n}}
=(-1)^{\frac{d-1}{2}}\blr{\underset{d}{\underbrace{3,\ldots,3}}}_d^{\eta_{2n}}\EE
be the number of degree~$d$ real rational curves through $d$ non-real points in~$\P^3$ 
counted with sign;
by Corollary~\ref{oddnums_crl}\ref{vancrl_it}, $N_d^{\R}$ is well-defined even 
if $d\!\in\!2\Z$.
Denote~by
$$N_d^{\C}=\blr{\underset{2d}{\underbrace{3,\ldots,3}}}_d^{\P^{2n-1}}
\qquad\hbox{and}\qquad
\wt{N}_d^{\C}=\blr{2,2,\underset{2d-1}{\underbrace{3,\ldots,3}}}_d^{\P^{2n-1}}$$
the number of degree~$d$ (complex) rational curves through $2d$ points in~$\P^3$
and the number of degree~$d$ rational curves 
through 2 lines and $2d\!-\!1$ points in~$\P^3$,
respectively.
The next corollary is also obtained in Section~\ref{quant_sec}.

\begin{crl}\label{main_crl}
If $d\!\in\!\Z^+$ and $d\!\ge\!2$, then
$$N_d^{\R}= \sum_{\begin{subarray}{c}2d_1+d_2=d\\ d_1,d_2\ge1 \end{subarray}}
\!\!\!(-4)^{d_1-1}d_2\binom{d\!-\!2}{d_2\!-\!1}\wt{N}_{d_1}^{\C}N_{d_2}^{\R}\,,
\qquad
N_d^{\R}\cong_4  N_d^{\C}\cong_4  
\begin{cases}
1,&\hbox{if}~d\!\in\!\Z^+\!-\!2\Z;\\
0,&\hbox{if}~d\!\in\!2\Z^+;
\end{cases}$$
where $\cong_4$ denotes the congruence modulo~4.
\end{crl} 

\noindent
The procedures of \cite{Wel2}, \cite{Cho,Sol}, and \cite{Ge2,Teh} for determining 
the sign of each real curve passing through a specified real collection of 
constraints in~$\P^3$ are very different and depend on some global choices.
The latter affect the signs of all curves of a fixed degree in the same way, and 
so the real counts in each degree are determined up to an overall sign
by all three procedures. 
In the case of conjugate pairs of point constraints and odd-degree curves
(the intersection of the three settings),
the three procedures yield the same count, up to a sign in each degree.\\

\noindent
The second statement of Corollary~\ref{main_crl} establishes a special case of Mikhalkin's congruence,
a conjectural relation between real and complex counts of rational curves. 
Its analogues 
for counts of real rational curves with real point constraints in real del Pezzo surfaces 
as defined in~\cite{Wel1} are proved in~\cite{IKS13a,IKS13b}.
By \cite[Proposition~3]{BM} and \cite[Theorem~2]{BM},
the analogue of this statement for real point constraints in~$\P^3$ holds
with the sign modification in~\eref{signcorr_e}.
This suggests that it would be natural to modify the signs of~\cite{Wel2} as in~\eref{signcorr_e}.
By \cite[Theorem~2]{BM}, such a modification would also ensure 
the positivity of counts of rational curves with real point constraints
(but not with conjugate pairs of point constraints, as Table~\ref{P3nums_tbl} shows).
On the other hand, the second statement of Corollary~\ref{main_crl} does not extend
to more general constraints in~$\P^3$ (it fails for $d\!=\!1$ with two conjugate pairs 
of line constraints) or to~$\P^{2n-1}$ with $n\!\ge\!3$ 
(according to Table~\ref{Pnnums_tbl}).\\

\noindent
The numbers~\eref{realNdfn_e} count real curves passing through specified constraints
with signs and thus provide lower bounds for the actual numbers of such curves.
There are indications that these bounds are often sharp.
For example, for $d,m\!\in\!\Z^+$ with $d$ odd and $m\!=\!1$ if $d\!\ge\!5$,
there are configurations of $d\!-\!m$ conjugate pairs of points and
$2m$ conjugate pairs of lines in~$\P^3$
so that there are no real degree~$d$ curves passing through them;
see \cite[Examples~12,17,18]{Kollar13}.
In light of the recursion of Corollary~\ref{main_crl} and~\eref{RTrec_e},
\cite[Proposition~3]{Kollar13}, which relates the numbers $N_d^{\R}$ to counts
of real curves in~$\P^1\!\times\!\P^1$, may be opening a way for a combinatorial
proof that the numbers~$N_d^{\R}$ provide sharp lower bounds for $d\!\not\in\!2\Z$
(if this is indeed the case).\\

\noindent
The basic case (smallest~$k$) of the 
analogue of Theorem~\ref{main_thm} in complex Gromov-Witten
theory is equivalent to the associativity of the quantum product 
on the cohomology of the manifold; see \cite[Theorem~8.1]{RT}.
The basic case of 
Theorem~\ref{main2_thm} is similarly equivalent to a property of the quantum
product of a real symplectic manifold; see Section~\ref{quant_sec}.\\

\noindent
Theorem~\ref{main_thm} is a special case of Theorem~\ref{main2_thm},
which provides a WDVV-type relation for real Gromov-Witten invariants 
of real symplectic manifolds.
In the next two paragraphs, we outline the two proofs of Theorem~\ref{main2_thm}
appearing in this paper.
While the first approach requires some preparation, 
it is more natural from the point of view of real Gromov-Witten theory.
In~\cite{GZ5}, we describe a third proof of Theorem~\ref{main_thm},
which can be extended to some other cases of Theorem~\ref{main2_thm}.\\

\noindent
The WDVV~relation for complex Gromov-Witten theory obtained in \cite{KM,RT}
is a fairly direct consequence of a $\C$-codimension~1 relation on 
the Deligne-Mumford moduli space~$\ov\cM_{0,4}$ of complex genus~0 curves with 4~marked points.
According to this relation, the homology classes represented by two different nodal curves, 
e.g.~$[1,0]$ and~$[1,1]$ in Figure~\ref{m04_fig}, are the same. 
Thus, topologically defined counts of morphisms from these two types of domains 
into an almost Kahler manifold are the same.
As this relation simply states that two points in $\ov\cM_{0,4}$ represent the same
homology class, it is an immediate consequence of the connectedness of~$\ov\cM_{0,4}$.
The WDVV-type relation of Theorem~\ref{main2_thm} is a fairly direct consequence 
of an  $\R$-codimension~2 relation on the three-dimensional 
Deligne-Mumford moduli space~$\R\ov\cM_{0,3}$ 
of real genus~0 curves with 3~conjugate pairs of marked points which we establish in Section~3
through a detailed topological description of~$\R\ov\cM_{0,3}$;
see Proposition~\ref{DMrelR_prp} and its proof.
According to this relation, the (relative) homology classes represented by 
two different, two-nodal degenerations of real curves are the same.
Thus, topologically defined counts of morphisms from these two types of domains 
into a real almost Kahler manifold are the same; see Corollary~\ref{DMrelR_crl}.
This relation, for both curves and maps, is illustrated in Figure~\ref{RM_fig},
where the vertical line represents the irreducible component of the curve
preserved by the involution and the two horizontal lines represent 
the components interchanged by the involution.
In a sense, the situation with our recursion is analogous
to the situation with the $\C$-codimension~2 recursion of \cite[Lemma~1.1]{Getz}
on~$\ov\cM_{1,4}$,
which had to be discovered and established before it could be applied  
to complex genus~1 Gromov-Witten invariants.\\

\begin{figure}
\begin{pspicture} (-1.1,-1.5)(10,2.7)
\psset{unit=.4cm}
\psline[linewidth=.03](2,-1)(33,-1)
\rput(30,-2){$\ov\cM_{0,4}\!\approx\!\P^1$}
\rput(30,4){${\cal U}$}\rput(30.5,2){$\pi$}
\psline[linewidth=.03]{->}(30,3)(30,0)
\pscircle*(5,-1){.2}\pscircle*(15,-1){.2}\pscircle*(25,-1){.2}
\rput(5,-2){$[0,1]$}\rput(15,-2){$[1,0]$}\rput(25,-2){$[1,1]$}
\psline[linewidth=.02](3,6)(7,0)\psline[linewidth=.03](3,0)(7,6)
\pscircle*(5,3){.14}\pscircle*(4.5,2.25){.17}\pscircle*(5.5,3.75){.17}
\pscircle*(5.25,2.625){.17}\pscircle*(4.333,4){.17}
\rput(3.85,2.35){$z_0$}\rput(3.9,3.7){$z_1$}
\rput(6.2,3.6){$z_3$}\rput(6,2.7){$z_2$}
\psline[linewidth=.02](9,5)(9,1)
\psarc[linewidth=.02](8,5){1}{0}{90}\psarc[linewidth=.02](10,1){1}{180}{270}
\psline[linewidth=.02](11,5)(11,1)
\psarc[linewidth=.02](10,5){1}{0}{90}\psarc[linewidth=.02](12,1){1}{180}{270}
\psline[linewidth=.02](20,5)(20,1)
\psarc[linewidth=.02](19,5){1}{0}{90}\psarc[linewidth=.02](21,1){1}{180}{270}
\psline[linewidth=.02](22,5)(22,1)
\psarc[linewidth=.02](21,5){1}{0}{90}\psarc[linewidth=.02](23,1){1}{180}{270}
\psline[linewidth=.02](13,1.5)(18,4)
\psline[linewidth=.02](13,4.29)(18,3.17)
\pscircle*(16.9,3.45){.14}\pscircle*(14.5,2.25){.17}\pscircle*(15.5,3.75){.17}
\pscircle*(15.25,2.625){.17}\pscircle*(14.333,4){.17}
\rput(13.8,2.4){$z_0$}\rput(13.85,3.65){$z_1$}
\rput(16.2,3.9){$z_3$}\rput(15.9,2.4){$z_2$}
\psline[linewidth=.02](24.233,5.05)(24.75,-.275)
\psline[linewidth=.02](25.7,4.65)(24.55,-.525)
\pscircle*(24.7,.15){.14}\pscircle*(24.5,2.25){.17}\pscircle*(25.5,3.75){.17}
\pscircle*(25.25,2.625){.17}\pscircle*(24.333,4){.17}
\rput(23.8,2.4){$z_0$}\rput(23.85,3.65){$z_1$}
\rput(26.2,3.9){$z_3$}\rput(26,2.6){$z_2$}
\end{pspicture}
\caption{The universal curve $\cU\!\lra\!\ov\cM_{0,4}$.}
\label{m04_fig}
\end{figure}

\begin{figure}
\begin{pspicture}(-.5,-.2)(10,2.4)
\psset{unit=.4cm}
\psline[linewidth=.05](3,5)(3,0)
\psline[linewidth=.02](2.5,4.5)(6,4.5)\psline[linewidth=.02](2.5,.5)(6,.5)
\pscircle*(4,4.5){.15}\pscircle*(5,4.5){.15}
\pscircle*(3,3.5){.15}\pscircle*(3,1.5){.1}
\rput(4,5.1){\sm{1}}\rput(5,5.1){\sm{0}}
\rput(2.4,3.5){\sm{2}}\rput(2.4,1.5){\sm{$\bar2$}}
\pscircle*(4,.5){.1}\pscircle*(5,.5){.1}
\rput(4,-.2){\sm{$\bar1$}}\rput(5,-.2){\sm{$\bar0$}}
\psline[linewidth=.05](11,5)(11,0)
\psline[linewidth=.02](10.5,4.5)(14,4.5)\psline[linewidth=.02](10.5,.5)(14,.5)
\pscircle*(12,4.5){.15}\pscircle*(13,4.5){.15}
\pscircle*(11,3.5){.15}\pscircle*(11,1.5){.1}
\rput(12,5.1){\sm{$\bar1$}}\rput(13,5.1){\sm{0}}
\rput(10.4,3.5){\sm{2}}\rput(10.4,1.5){\sm{$\bar2$}}
\pscircle*(12,.5){.1}\pscircle*(13,.5){.1}
\rput(12,-.2){\sm{$1$}}\rput(13,-.2){\sm{$\bar0$}}
\rput(8,2.5){\begin{Large}$+$\end{Large}}
\rput(17,2.5){\begin{Large}$=$\end{Large}}
\psline[linewidth=.05](22,5)(22,0)
\psline[linewidth=.02](21.5,4.5)(25,4.5)\psline[linewidth=.02](21.5,.5)(25,.5)
\pscircle*(23,4.5){.15}\pscircle*(24,4.5){.15}
\pscircle*(22,3.5){.15}\pscircle*(22,1.5){.1}
\rput(23,5.1){\sm{2}}\rput(24,5.1){\sm{0}}
\rput(21.4,3.5){\sm{1}}\rput(21.4,1.5){\sm{$\bar1$}}
\pscircle*(23,.5){.1}\pscircle*(24,.5){.1}
\rput(23,-.2){\sm{$\bar2$}}\rput(24,-.2){\sm{$\bar0$}}
\psline[linewidth=.05](30,5)(30,0)
\psline[linewidth=.02](29.5,4.5)(33,4.5)\psline[linewidth=.02](29.5,.5)(33,.5)
\pscircle*(31,4.5){.15}\pscircle*(32,4.5){.15}
\pscircle*(30,3.5){.15}\pscircle*(30,1.5){.1}
\rput(31,5.1){\sm{$\bar2$}}\rput(32,5.1){\sm{0}}
\rput(29.4,3.5){\sm{1}}\rput(29.4,1.5){\sm{$\bar1$}}
\pscircle*(31,.5){.1}\pscircle*(32,.5){.1}
\rput(31,-.2){\sm{$2$}}\rput(32,-.2){\sm{$\bar0$}}
\rput(27,2.5){\begin{Large}$+$\end{Large}}
\end{pspicture}
\caption{A relation in $H_1(\R\ov\cM_{0,3})$; the dots labeled $i$ and $\bar{i}$ indicate 
the marked points $z_i^+$ and $z_i^-$, respectively.}
\label{RM_fig}
\end{figure}
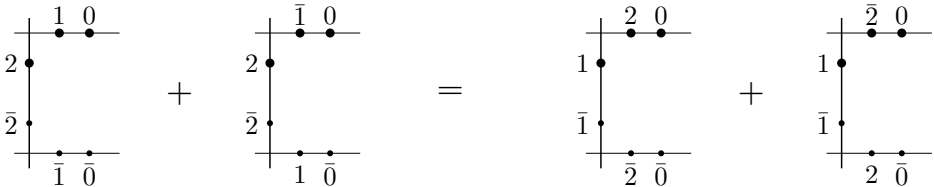

\noindent 
In Section~\ref{mainpf_sec}, we give an alternative proof of Theorem~\ref{main2_thm},
which bypasses Proposition~\ref{DMrelR_prp}.
We pull back the usual relation on~$\ov\cM_{0,4}$
by the forgetful morphism~$f_{012\bar{0}}$ which keeps the marked points
$z_0^+,z_1^+,z_2^+,z_0^-$; see~\eref{forgmap0_e} and~\eref{tiNdfn0_e}.
In the proof of \cite[Theorem~10.4]{RT}, a nodal element of~$\ov\cM_{0,4}$  
is a regular value of  a similar map and
all of its preimages are of the same type and contribute $+1$ each to the relevant count;
the situation with the proof of Corollary~\ref{DMrelR_crl} from
Proposition~\ref{DMrelR_prp} is analogous.
In contrast, a nodal element of~$\ov\cM_{0,4}$  is not a regular value of~$f_{012\bar{0}}$
and its preimages can be of four types, as indicated in Figures~\ref{LHS0_fig} and~\ref{RHS0_fig}; 
they are  morphisms from either a three-component domain or from a two-component domain.
The contribution of each three-component morphism  to the relevant count~\eref{tiNdfn0_e}
is no longer necessarily~$+1$; see Lemma~\ref{sign0_lmm}.
The stratum of two-component morphisms is not even 0-dimensional, but 
we show through a topological analysis that it does not contribute
to the count; see Lemma~\ref{sign2_lmm}.\\

\noindent
In real Gromov-Witten theory, signs of various contributions are generally 
a delicate issue. It shows up explicitly in the above description 
of the second approach, but is hidden in the first approach.
The analysis of signs for both approaches is carried out in Section~\ref{orient_sec},
where different orientations of moduli spaces of constrained real morphisms are compared.
This allows us to establish Propositions~\ref{KunnethSplit_prp} and~\ref{twistGW_prp}, 
which are used in the proofs of Theorem~\ref{main2_thm} in 
Sections~\ref{Rmainpf_sec} and~\ref{mainpf_sec}, 
as well as Theorem~\ref{evenvan_thm}, 
which provides vanishing results for 
real Gromov-Witten invariants of real symplectic manifolds, 
including in positive genera.\\

\noindent
We would like to thank J.~Morgan for his help in precisely identifying
$\ov\cM_{0,3}^{\R}$ in Remark~\ref{DMstr_rmk} and
E.~Ionel, J.~Koll\'ar, M.~Liu, N.~Sheridan, J.~Solomon, M.~Tehrani, 
and G.~Tian for related discussions.
We are also grateful to the referees for comments on previous versions of this paper
which led to significant improvements in the exposition.

\section{Main theorems and corollaries}
\label{mainthm_sec}

\noindent
The formula of Theorem~\ref{main_thm} is fundamentally a relation between 
real genus~0 GW-invariants;
it is a special case of the relation of Theorem~\ref{main2_thm}
for real symplectic manifolds. 
The latter implies that the real invariants of at least some real symplectic manifolds
are essentially independent of the involution~$\phi$; see Corollary~\ref{inveq_crl2}.
Likewise, the vanishing of the numbers $\lr{c_1,\ldots,c_k}_d^{\phi}$ with 
$c_i\!\in\!2\Z$ for some~$i$, established
in \cite[Section~A.5]{Teh} using the Equivariant Localization Theorem
\cite[(3.8)]{AB}, is a special case of the general vanishing phenomenon
for $\phi$-invariant insertions established in Theorem~\ref{evenvan_thm}~below.\\

\noindent
A \textsf{real symplectic manifold} is a triple $(X,\om,\phi)$ consisting of 
a symplectic manifold~$(X,\om)$ and an involution $\phi\!:X\!\lra\!X$
such that $\phi^*\om\!=\!-\om$.
Examples include $\P^{2n-1}$ with the standard Fubini-Study symplectic
form~$\om_{2n}$ and the involutions~\eref{taudfn_e} and~\eref{etadfn_e},
as well as  $(\P^{2n},\om_{2n+1})$ with the involution
$$\tau_{2n+1}\!:\P^{2n}\lra\P^{2n}\,, \qquad
[X_1,\ldots,X_{2n},X_{2n+1}]\lra 
[\bar{X}_2,\bar{X}_1,\ldots,\bar{X}_{2n},\bar{X}_{2n-1},\bar{X}_{2n+1}],$$
which extends~\eref{taudfn_e} to the even-dimensional projective spaces.
If 
\BE{ellbfa_e}\ell\!\ge\!0, \qquad \a\equiv(a_1,\ldots,a_\ell)\in(\Z^+)^{\ell}\,,\EE
and $X_{n;\a}\!\subset\!\P^{n-1}$ is a complete intersection of multi-degree~$\a$
preserved by~$\tau_n$,  $\tau_{n;\a}\!\equiv\!\tau_n|_{X_{n;\a}}$
is an anti-symplectic involution on $X_{n;\a}$ with respect to the symplectic form
$\om_{n;\a}\!=\!\om_n|_{X_{n;\a}}$. 
Similarly, if $X_{2n;\a}\!\subset\!\P^{2n-1}$ is preserved by~$\eta_{2n}$,
$\eta_{2n;\a}\!\equiv\!\eta_{2n}|_{X_{2n;\a}}$
is an anti-symplectic involution on $X_{2n;\a}$ with respect to the symplectic form
$\om_{2n;\a}\!=\!\om_{2n}|_{X_{2n;\a}}$.\\

\noindent
Let $(X,\om,\phi)$ be a real symplectic manifold.
The fixed locus~$X^{\phi}$ of~$\phi$ is a Lagrangian submanifold,
which may be empty.
Let
$$H_2(X)_{\phi}=\big\{\be\!\in\!H_2(X;\Z)\!:\,\phi_*\be=-\be\big\},\qquad
H^*(X)^{\phi}_{\pm}\equiv\big\{\mu\!\in\!H^*(X)\!:\,\phi^*\mu\!=\!\pm\mu\big\}.$$
Similarly to \cite[Section~1]{Ge2}, we define
\BE{fddfn_e}\fd\!: H_2(X)\lra H_2(X)_{\phi}\qquad\hbox{by}\qquad
\fd(\be)=\be-\phi_*\be\,.\EE
A \sf{real bundle pair} $(V,\ti\phi)\!\lra\!(X,\phi)$   
consists of a complex vector bundle $V\!\lra\!X$ and 
a conjugation~$\ti{\phi}$ on~$V$ lifting~$\phi$,
i.e.~an involution restricting to an anti-complex linear
homomorphism on each fiber.
The fixed locus $V^{\ti\phi}\!\lra\!X^{\phi}$ is then a maximal totally real
subbundle of~$V|_{X^{\phi}}$, i.e.
$$V|_{X^{\phi}}=V^{\ti\phi}\oplus\fI V^{\ti\phi}\,,$$
where $\fI$ is the complex structure on~$V$.
Let
$$w_2^{\ti\phi}(V)\in H^2_{\phi}(X;\Z_2)\equiv H^2_{\Z_2}(X;\Z_2)$$
denote the equivariant second Stiefel-Whitney class of $(V,\ti\phi)$;
see \cite[Section~2]{GZ1}.\\

\noindent
Let $\cJ_{\om}^{\phi}$ be the space of 
$\om$-compatible almost complex structures~$J$ on~$X$ such that $\phi^*J\!=\!-J$.
For $J\!\in\!\cJ_{\om}^{\phi}$, $c\!=\!\tau,\eta$, and $\be\!\in\!H_2(X)_{\phi}$,  
denote~by 
\BE{fMbeX_e}\fM_{0,k}(X,\be)^{\phi,c}\subset \ov\fM_{0,k}(X,\be)^{\phi,c}\EE
the moduli space of $c$-equivalence classes of 
$k$-marked $J$-holomorphic $(\phi,c)$-real maps in the homology class~$\be$ and
its natural compactification consisting of stable real maps from nodal domains.\\

\noindent
By \cite[Theorem~6.5]{Ge2}, both spaces in~\eref{fMbeX_e} with $c\!=\!\tau$ 
are orientable in the sense of Kuranishi structures 
(or for a generic~$J$ if $(X,\om)$ is strongly semi-positive)~if 
\begin{enumerate}[label=($\cO_{\tau}$),leftmargin=*]
\item $X^{\phi}$ is orientable and there exists a real bundle pair
$(E,\ti\phi)\!\lra\!(X,\phi)$ such that 
$$w_2(TX^{\phi})=w_1(E^{\ti\phi})^2 \quad\hbox{and}\quad
\frac12\lr{c_1(X),\be'}+\lr{c_1(E),\be'}\in 2\Z~~\forall~\be'\!\in\!H_2(X)_{\phi}\,.$$
\end{enumerate}
The first requirement on~$(E,\ti\phi)$ above implies that 
$TX^{\phi}\!\oplus\!2E^{\ti\phi}$  admits a spin structure.
By \cite[Theorem~1.1]{GZ1}, both spaces in~\eref{fMbeX_e} with $c\!=\!\eta$
are orientable~if
\begin{enumerate}[label=($\cO_{\eta}$),leftmargin=*]
\item $w_2^{\La_{\C}^{\top}\tnd\phi}(\La_{\C}^{\top}TX)=\ka^2$ for some
$\ka\!\in\!H_{\phi}^1(X)$.
\end{enumerate}
This condition implies that $(TX,\tnd\phi)$ admits a spin sub-structure, 
as defined above \cite[Corollary~5.10]{GZ2}.
By \cite[Corollary~2.4]{GZ1}, \tauetacond{\eta} holds if either 
$\La_{\C}^{\top}(TX,\tnd\phi)$ admits a real square root, i.e.~there is an isomorphism
\BE{SqRT_e}\La_{\C}^{\top}(TX,\tnd\phi)\approx (L,\ti\phi)^{\otimes 2}\EE
for a real line bundle pair $(L,\ti\phi)\!\lra\!(X,\phi)$, or 
$\pi_1(X)\!=\!0$ and $w_2(X)\!=\!0$.
A~fixed real square root determines a spin  sub-structure on~$(TX,\tnd\phi)$.\\

\noindent
The moduli space $\ov\fM_{0,k}(X,\be)^{\phi,c}$ with $c\!=\!\tau,\eta$
has no boundary in the sense of Kuranishi structures~if
\BE{bndcond_e}\be\not\in\Im(\fd) \qquad\hbox{or}\qquad X^{\phi}=\eset.\EE
Thus, it carries a virtual fundamental class if~\tauetacond{c}, with $c$ as above, and~\eref{bndcond_e} hold. 
Under the above assumptions, we define
\BE{phinums_e}
\blr{\mu_1,\ldots,\mu_k}_{\be}^{\phi,c}
=\int_{[\ov\fM_{0,k}(X,\be)^{\phi,c}]^{\vir}}
\ev_1^*\mu_1\,\ldots\,\ev_k^*\mu_k\in\Q\EE
for any $\mu_1,\ldots,\mu_k\!\in\!H^*(X)$.
This number depends on the chosen orientation of the moduli space.
If $c\!=\!\tau$, we orient the moduli space as
 in the proofs of \cite[Corollary~1.8]{Ge} and
\cite[Theorem~6.5]{Ge2} from any spin structure  on~$TX^{\phi}\!\oplus\!2E^{\ti\phi}$.
If $c\!=\!\eta$, we orient the moduli space via the pinching construction of 
\cite[Lemma~2.5]{Teh} from any spin sub-structure on~$(TX,\tnd\phi)$;
see \cite[Corollary~5.10]{GZ2}.
In either case, we use the {\it same} spin structure or sub-structure for all~$\be$.\\

\noindent
If \tauetacond{\tau} and \tauetacond{\eta} are satisfied, 
but not necessarily~\eref{bndcond_e}, the glued moduli space
\BE{gluedsp_e2}\ov\fM_{0,k}(X,\be)^{\phi} \equiv
\ov\fM_{0,k}(X,\be)^{\phi,\tau}\cup\ov\fM_{0,k}(X,\be)^{\phi,\eta}\EE
is orientable and has no boundary; see \cite[Theorem~1.7]{Teh}.
We then define
\BE{phinums_e2}
\blr{\mu_1,\ldots,\mu_k}_{\be}^{\phi}
=\int_{[\ov\fM_{0,k}(X,\be)^{\phi}]^{\vir}}
\ev_1^*\mu_1\,\ldots\,\ev_k^*\mu_k\in\Q\EE
for any $\mu_1,\ldots,\mu_k\!\in\!H^*(X)$.
The orientations on $\ov\fM_{0,k}(X,\be)^{\phi,\tau}$ and 
$\ov\fM_{0,k}(X,\be)^{\phi,\eta}$ constructed as in the previous paragraph
induce an orientation on $\ov\fM_{0,k}(X,\be)^{\phi}$ 
after reversing the orientation on $\ov\fM_{0,k}(X,\be)^{\phi,\eta}$ if 
the chosen spin structure on $TX^{\phi}\!\oplus\!2E^{\ti\phi}$
and  spin sub-structure on~$(TX,\tnd\phi)$ induce the same
orientation on~$X^{\phi}$; see \cite[Proposition~3.3]{Teh}.\\

\noindent
Choose bases $\{\ga_i\}_{i\le\ell}$ and $\{\ga^i\}_{i\le\ell}$ for $H^*(X)$ so that 
$$\ga^i\in H^*(X)^{\phi}_+\cup H^*(X)^{\phi}_-
\qquad\hbox{and}\qquad
\PD_{X^2}(\De_X)=\sum_{i=1}^{\ell}\ga_i\times\ga^i\in H^*(X^2),$$
where $\De_X\!\subset\!X^2$ is the diagonal.
If $\mu_1,\ldots,\mu_k\!\in\!H^*(X)$ and $I\!\subset\!\{1,\ldots,k\}$,
let $\mu_I$ denote a tuple with the entries~$\mu_i$ with $i\!\in\!I$,
in some order.
Let
\BE{GWdfn_e}\blr{\mu_1,\ldots,\mu_k}_{\be}^X=\int_{[\ov\fM_{0,k}(X,\be)]^{\vir}}
\ev_1^*\mu_1\,\ldots\,\ev_k^*\mu_k\in\Q\,,\EE
denote the (complex) genus~0 GW-invariants of~$X$.

\begin{thm}\label{main2_thm}
Let $(X,\om,\phi)$ be a compact real symplectic manifold, $k\!\in\!\Z$ with $k\!\ge\!2$,
$$\be\in H_2(X)_{\phi}\!-\!\{0\},\qquad \mu\!\in\!H^{2*}(X)_+^{\phi},
\quad\hbox{and}\quad\mu_1,\ldots,\mu_k\!\in\!H^{2*}(X)^{\phi}_-.$$
\begin{enumerate}[label=(\arabic*),leftmargin=*]

\item\label{WDVVc_it} 
If $c\!=\!\tau,\eta$ and \tauetacond{c} and \eref{bndcond_e} are satisfied, then
\begin{equation*}\begin{split}
\blr{\mu_1,\mu\mu_2,\mu_3,\ldots,\mu_k}_{\be}^{\phi,c}
-\blr{\mu\mu_1,\mu_2,\mu_3,\ldots,\mu_k}_{\be}^{\phi,c}
=\sum_{\begin{subarray}{c}\fd(\be_1)+\be_2=\be\\ \be_1,\be_2\in H_2(X)-\{0\} \end{subarray}}
\!\!\!\sum_{I\sqcup J=\{3,\ldots,k\}}\!\!\!
\sum_{\begin{subarray}{c}1\le i\le\ell\\ \ga^i\in H^{2*}(X)^{\phi}_-\end{subarray}}
\!\!\!\!\!\!2^{|I|}\Bigg(\quad&\\
\blr{\mu,\mu_1,\mu_I,\ga_i}_{\be_1}^X
\!\blr{\mu_2,\mu_J,\ga^i}_{\be_2}^{\phi,c}
\!-\blr{\mu,\mu_2,\mu_I,\ga_i}_{\be_1}^X
\!\blr{\mu_1,\mu_J,\ga^i}_{\be_2}^{\phi,c}&\Bigg).
\end{split}\end{equation*}

\item If \tauetacond{\tau} and \tauetacond{\eta} are satisfied, then
the above identity holds for the $\lr{\ldots}^{\phi}$ invariants.

\end{enumerate}
\end{thm}

\noindent
This theorem, established in Section~\ref{Rmainpf_sec}, 
concerns real genus~0 GW-invariants~\eref{phinums_e} and~\eref{phinums_e2}
with all insertions~$\mu_i$ coming from~$H^{2*}(X)^{\phi}_-$.
By the first part of Theorem~\ref{evenvan_thm} below,  
the invariants~\eref{phinums_e} and~\eref{phinums_e2} with any 
insertion~$\mu_i$ coming from~$H^{2*}(X)^{\phi}_+$ vanish.
The proof of Theorem~\ref{evenvan_thm} in Section~\ref{orient_sec}
extends the vanishing statement of \cite[Theorem~1.10]{Teh} for real genus~0
invariants with even-degree insertions to all settings
when the real GW-invariants are defined and the unmarked real moduli space is orientable.
By \cite[Theorem~1.3]{GZ3}, this is the case  
in any genus under the assumptions in~\ref{vanpos_it} of Theorem~\ref{evenvan_thm}.

\begin{thm}\label{evenvan_thm}
Let $(X,\om,\phi)$ be a compact real symplectic $2n$-manifold,
$\be\!\in\!H_2(X)_{\phi}\!-\!\{0\}$, and 
$\mu_1,\ldots,\mu_k\!\in\!H^*(X)$ with $\mu_i\!\in\!H^*(X)^{\phi}_+$ 
for some~$i$.
\begin{enumerate}[label=(\arabic*),leftmargin=*]
\item\label{van0_it} Suppose $c\!=\!\tau,\eta$, $(\cO_{\tau})$ holds if $c\!=\!\tau$, and 
$(\cO_{\eta})$ holds if $c\!=\!\eta$.
If  \eref{bndcond_e} is  satisfied, then
\hbox{$\lr{\mu_1,\ldots,\mu_k}_{\be}^{\phi,c}\!=\!0$}.
If $(\cO_{\tau})$ and $(\cO_{\eta})$ are satisfied, but not necessarily~\eref{bndcond_e},
$\lr{\mu_1,\ldots,\mu_k}_{\be}^{\phi}\!=\!0$.

\item\label{vanpos_it} If $c$ is an orientation-reversing involution on a compact orientable genus~$g$ 
surface~$\Si_g$, $n$ is odd, $\La_{\C}^{\top}(TX,\tnd\phi)$ admits a real square root,
and $X^{\phi}\!=\!\eset$,
then real genus~$g$ GW-invariants $\lr{\mu_1,\ldots,\mu_k}_{\be}^{\phi,c}$ 
of~$(X,\om,\phi)$ vanish. 
\end{enumerate}
\end{thm}

\begin{rmk}\label{RealGWsI_rmk}
Let $c_g$ be an orientation-reversing involution on~$\Si_g$ so that $\Si_g^{c_g}\!=\!\eset$.
By \cite[Corollary~1.1]{Nat}, $c_g$ is unique up to conjugation by diffeomorphisms of~$\Si_g$.
Similarly to~\eref{gluedsp_e2}, the moduli spaces $\ov\fM_{g,k}(X,\be)^{\phi,c}$
of real $J$-holomorphic maps corresponding to different topological types of 
involutions~$c$ on~$\Si_g$ can be glued together into a moduli space
$\ov\fM_{g,k}(X,\be)^{\phi}$ without boundary. 
If $X^{\phi}\!=\!\eset$, then
\BE{fMpos_e}\ov\fM_{g,k}(X,\be)^{\phi}=\ov\fM_{g,k}(X,\be)^{\phi,c_g}\EE
and the GW-invariants $\lr{\ldots}_{\be}^{\phi,c_g}$ are the same as the combined
real GW-invariants $\lr{\ldots}_{g,\be}^{\phi}$ 
expected to arise from the left-hand side of~\eref{fMpos_e}.
Since the present paper was first completed,
such invariants have been defined with the condition $X^{\phi}\!=\!\eset$ weakened
to the existence of the square root as in~\eref{SqRT_e} such that 
\hbox{$w_2(TX^{\phi})\!=\!w_1(L^{\ti\phi})^2$}; see \cite[Theorems~1.3,1.4]{RealGWsI}.
The proof of Theorem~\ref{evenvan_thm} applies verbatim to the real genus~$g$ GW-invariants
of \cite[Theorems~1.4,1.5]{RealGWsI}.
\end{rmk}

\noindent
For a strongly semi-positive real symplectic manifold $(X,\om,\phi)$,
the real genus~0 GW-invariant through constraints $\mu_1,\ldots,\mu_k$
is of the same parity as the complex GW-invariant of the same degree
through the constraints $\mu_1,\phi^*\mu_1,\ldots,\mu_k,\phi^*\mu_k$.
Thus, Theorem~\ref{evenvan_thm} implies that certain complex genus~0 GW-invariants are even. 
For example, the GW-invariants of~$\P^{2n-1}$ with even numbers of insertions of 
each codimension that include insertions of even codimensions are even. 
This is not the case for even-dimensional projective spaces
(for which the degree $4d\!+\!1$ unmarked real moduli spaces are not orientable;
see \cite[Proposition~5.1]{Sol}).
For example, the number of lines through two points in~$\P^n$ is~1
(these constraints are of even codimension if~$n$ is even).\\

\noindent
For a real symplectic manifold~$(X,\om,\phi)$, let 
$H_{\eff}(X)_{\phi}\!\subset\!H_2(X)_{\phi}$ denote the subset of nonzero classes
that can be  represented by a $J$-holomorphic map from a disjoint union of copies of~$\P^1$
for every $J\!\in\!\cJ_{\om}^{\phi}$.

\begin{crl}\label{inveq_crl}
Let $(X,\om,\phi)$ be a compact real  symplectic manifold such that 
every positive-degree element of $H^{2*}(X)^{\phi}_-$ is divisible by an element 
of~$H^2(X)^{\phi}_-$ in~$H^{2*}(X)$ and \hbox{$\be\!\in\!H_{\eff}(X)_{\phi}$}.
Then there exist linear maps
$$P_{\be';\be}\!: \bigoplus_{k=1}^{\i}H^{2*}(X)^{\otimes k} \lra H^{2*}(X)^{\phi}_-,
\quad \be'\!\in\!H_{\eff}(X)_{\phi},\,
\be\!-\!\be'\!\in\!\big(H_{\eff}(X)_{\phi}\!\cup\!\{0\}\big)\!\cap\!\Im(\fd),$$
determined by the  GW-invariants of~$(X,\om)$ and 
$\phi_*\!:H_*(X)\!\lra\!H_*(X)$ with the following properties.
\begin{enumerate}[label=(\arabic*),leftmargin=*]

\item\label{inveq_it1}  If $c\!=\!\tau,\eta$, \tauetacond{c} is satisfied, 
and either $\be\!\not\in\!\Im(\fd)$ or $X^{\phi}\!=\!\eset$,
then
\BE{inveq_e0}
\blr{\mu_1,\ldots,\mu_k}_{\!\be}^{\!\phi,c}=\!\!\!
\sum_{\begin{subarray}{c}\be'\in H_{\eff}(X)_{\phi}\\
\be-\be'\in(H_{\eff}(X)_{\phi}\cup\{0\})\cap\Im(\fd)\end{subarray}}
\hspace{-.7in}\blr{P_{\be';\be}(\mu_1,\ldots,\mu_k)}_{\!\be'}^{\!\phi,c}
\quad\forall\,\bigotimes_{i=1}^k\!\mu_i\in H^{2*}(X)^{\otimes k},\,k\in\!\Z^+.\EE

\item\label{inveq_it2} If \tauetacond{\tau} and \tauetacond{\eta} are satisfied, 
then~\eref{inveq_e0} holds with 
$\lr{\cdot}^{\phi,c}$ replaced by~$\lr{\cdot}^{\phi}$.

\end{enumerate}
\end{crl}

\noindent
Corollary~\ref{inveq_crl} is deduced from Theorems~\ref{main2_thm} and~\ref{evenvan_thm} 
in Section~\ref{quant_sec}.
It provides the strongest results for real \textsf{Fano} symplectic 
manifolds, i.e.~real symplectic manifolds~$(X,\om,\phi)$ such that 
$\lr{c_1(X),\be}\!>\!0$ for all $\be\!\in\!H_{\eff}(X)_{\phi}$.
For a real Fano symplectic  manifold $(X,\om,\phi)$ and 
$\be\!\in\!H_{\eff}(X)_{\phi}$, let
$$c_{\min}^{\phi}(\be),c_+^{\phi}(\be)\in\Z^+$$
denote the smallest and the second smallest values of the function
$$\big\{\be'\!\in\!H_{\eff}(X)_{\phi},\,\be\!-\!\be'\!\in\!\Im(\fd)\big\}
\lra\Z^+, \qquad \be'\lra\lr{c_1(X),\be'};$$
if the smallest value is achieved by two different classes~$\be'$,
then $c_{\min}^{\phi}(\be)\!\equiv\!c_+^{\phi}(\be)$.
Let $\be_{\bu}^{\phi}\!\in\!H_{\eff}(X)_{\phi}$ be such~that 
$$\blr{c_1(X),\be_{\bu}^{\phi}}=c_{\min}^{\phi}(\be).$$

\begin{crl}\label{inveq_crl2}
Let $(X,\om,\phi)$ and $\be$ be as in Corollary~\ref{inveq_crl}.
If $X$ is Fano and \hbox{$c_+^{\phi}(\be)\!>\!(\dim X)/2\!+\!1$}, then there exists a linear map
$$P_{\be}\!: \bigoplus_{k=1}^{\i}H^{2*}(X)^{\otimes k} \lra
H^{n-1+c_{\min}(\phi)}(X)^{\phi}_-$$
determined by the GW-invariants of~$(X,\om)$ and 
$\phi_*\!:H_*(X)\!\lra\!H_*(X)$ with the following properties.
\begin{enumerate}[label=(\arabic*),leftmargin=*]

\item\label{inveq2_it0} 
If $c\!=\!\tau,\eta$, \tauetacond{c} is satisfied, and either 
$\be\!\not\in\!\Im(\fd)$ or $X^{\phi}\!=\!\eset$, then
\BE{inveq2_e0}
\blr{\mu_1,\ldots,\mu_k}_{\!\be}^{\!\phi,c}=
\blr{P_{\be}(\mu_1,\ldots,\mu_k)}_{\!\be_{\bu}^{\phi}}^{\!\phi,c}
\quad\forall\,\mu_1,\ldots,\mu_k\in H^{2*}(X)^{\otimes k},\,k\in\!\Z^+.\EE

\item\label{inveq2_it1} If \tauetacond{\tau} and \tauetacond{\eta} are satisfied, 
then~\eref{inveq2_e0} holds with 
$\lr{\cdot}^{\phi,c}$ replaced by~$\lr{\cdot}^{\phi}$.

\end{enumerate} 
If $X$ is Fano and $c_{\min}^{\phi}(\be)\!>\!(\dim X)/2\!+\!1$, then the invariants 
$\lr{\cdot}_{\be}^{\phi,c}$ and $\lr{\cdot}_{\be}^{\phi}$ 
with insertions from $H^{2*}(X)$
vanish
under the assumptions in~\ref{inveq2_it0} and~\ref{inveq2_it1}, respectively.
\end{crl}

\noindent
The virtual dimensions of the moduli spaces in~\eref{gluedsp_e2} are 
\BE{virdim_e}\dim^{\vir}\ov\fM_{0,k}(X,\be)^{\phi} =\dim^{\vir}\ov\fM_{0,k}(X,\be)^{\phi,c}
=\blr{c_1(X),\be}\!+\!(n\!-\!3)+2k\,,\EE
where $n\!=\!(\dim X)/2$.
In particular, the real genus~0 one-insertion GW-invariants 
$\lr{\mu}_{\be}^{\phi,c}$ and $\lr{\mu}_{\be}^{\phi}$ (whenever they are defined)
vanish if $\lr{c_1(X),\be}\!>\!n\!+\!1$.
Thus, Corollary~\ref{inveq_crl2} is an immediate consequence of Corollary~\ref{inveq_crl}.\\

\noindent
For $n\!\in\!\Z$ and $\ell,\a$ as in~\eref{ellbfa_e}, let
$$|\a|=a_1\!+\!\ldots\!+\!a_{\ell}, \qquad
\lr{\a}_n=2n\!-\!2\!-\!|\a|\!-\!\ell.$$
If $X_{n;\a}\!\subset\!\P^{n-1}$ is a complete intersection as before, then
$$\La_{\C}^{\top}TX_{n;\a}\approx\cO_{\P^{n-1}}\big(n\!-\!|\a|\big)\big|_{X_{n;\a}}.$$
By the Lefschetz Theorem on Hyperplane Sections \cite[p156]{GH},
$\pi_1(X_{n;\a})\!=\!0$ if the complex dimension of~$X_{n;\a}$ is at least~2.
If $n\!\in\!2\Z$ and \hbox{$X_{n;\a}\!\subset\!\P^{n-1}$} is $\eta_n$-invariant, then 
$w_2(X_{n;\a})\!=\!0$ and $X_{n;\a}^{\eta_{n;\a}}\!=\!\eset$.
By \cite[Corollary~2.4]{GZ1}, 
an $\eta_n$-invariant complete intersection $X_{n;\a}\!\subset\!\P^{n-1}$
thus satisfies~\tauetacond{\eta} if its complex dimension is at least~2 
or $X_{n;\a}\!\approx\!\P^1$.
If the complex dimension of~$X_{n;\a}$ is~1 and $X_{n;\a}\!\not\approx\!\P^1$,
then the moduli spaces in~\eref{gluedsp_e2} with $X\!=\!X_{n;\a}$ are empty.
A $\tau_n$-invariant complete intersection $X_{n;\a}\!\subset\!\P^{n-1}$
satisfies~\tauetacond{\tau} if
\BE{CItau_e} n-|\a|\in2\Z \qquad\hbox{and}\qquad
a_1^2+\ldots+a_{\ell}^2- |\a|\in 4\Z\,;\EE
see the proof of \cite[Corollary~6.8]{Ge2}.
If the first condition in~\eref{CItau_e} is satisfied, then 
$(X_{n;\a},\tau_{n;\a})$ satisfies~\tauetacond{\eta}.\\

\noindent
For $d\!\in\!\Z$, let \hbox{$\lr{d}\!\subset\!H_2(X_{n;\a})$} denote the subset of classes~$\be$
whose image in~$\P^{n-1}$ is~$d$ times the homology class of a line $\P^1\!\subset\!\P^{n-1}$.
If in addition~$\phi$ is an involution on~$X_{n;\a}$, let
$$\lr{d}_{\phi}=\lr{d}\cap H_2\big(X_{n;\a}\big)_{\phi}\,.$$
If the complex dimension of~$X_{n;\a}$ is at least~3, $\lr{d}$ consists of a single element.
In all cases, $\be\!\in\!\lr{d}_{\phi}$ satisfies the first condition in~\eref{bndcond_e} if 
$d\!\not\in\!2\Z$.
We denote by $H\!\in\!H^2(X_{n;\a})$ the restriction of the hyperplane class.\\

\noindent
Suppose $X\!=\!X_{n;\a}\!\subset\!\P^{n-1}$ is a complete intersection 
of multi-degree~$\a$ invariant under 
$\phi_{\P^{n-1}}\!\equiv\!\eta_n$ or $\phi_{\P^{n-1}}\!\equiv\!\tau_n$, 
$\phi\!=\!\phi_{\P^{n-1}}|_X$, $c\!=\!\eta,\tau$, and $d\!\in\!\Z$.
We denote~by
\begin{enumerate}[label=$\bu$,leftmargin=*]

\item $\lr{\ldots}_d^{\phi}$ the sum of the numbers~\eref{phinums_e2} 
over $\be\!\in\!\lr{d}_{\phi}$ if $\phi_{\P^{n-1}}\!=\!\eta_n$
or \eref{CItau_e} is satisfied;

\item $\lr{\ldots}_d^{\phi,c}$ the sum of the numbers~\eref{phinums_e} 
over $\be\!\in\!\lr{d}_{\phi}$ if $\phi_{\P^{n-1}}\!=\!\eta_n$, or

\begin{enumerate}[leftmargin=*]

\item[($\tau_n\eta$)] 
$d\!\not\in\!2\Z$ , $c\!=\!\eta$, and the first condition in~\eref{CItau_e} is satisfied, or

\item[($\tau_n\tau$)]  
$d\!\not\in\!2\Z$ , $c\!=\!\tau$, and both conditions in~\eref{CItau_e} are satisfied.

\end{enumerate}
\end{enumerate}
The next corollary is also proved in Section~\ref{quant_sec}.

\begin{crl}\label{inveqCI_crl}
Suppose $n\!\in\!\Z^+$, $\ell\!\in\!\Z^{\ge0}$, $\a\!\in\!(\Z^+)^{\ell}$,
$X\!=\!X_{n;\a}\!\subset\!\P^{n-1}$ is a complete intersection 
of multi-degree~$\a$ invariant under $\phi_{\P^{n-1}}\!\equiv\!\eta_n$ or 
$\phi_{\P^{n-1}}\!\equiv\!\tau_n$, and $\phi\!=\!\phi_{\P^{n-1}}|_X$.
\begin{enumerate}[label=(\arabic*),leftmargin=*]

\item\label{CIvan_it}
Let $c\!=\!\eta,\tau$ and $\mu_1,\ldots,\mu_k\!\in\!H^{2*}(X)$.
If $\phi_{\P^{n-1}}\!=\!\eta_n$ and $c\!=\!\tau$, then $\lr{\mu_1,\ldots,\mu_k}_d^{\phi,c}\!=\!0$.
The same conclusion holds~if either ($\tau_n\eta$) holds or
\begin{enumerate}[label=$\bu$,leftmargin=*]

\item $\phi_{\P^{n-1}}\!=\!\eta_n$ or ($\tau_n\tau$) is satisfied and

\item $a_i\!\in\!2\Z$ for some~$i$, or $\mu_j\!\in\!H^{4*}(X)$ for some~$j$, 
or $d\!\in\!2\Z$.

\end{enumerate}
If the last bullet condition holds and either $\phi_{\P^{n-1}}\!=\!\eta_n$ or 
\eref{CItau_e} is satisfied, then \hbox{$\lr{\mu_1,\ldots,\mu_k}_d^{\phi}\!=\!0$}.

\item\label{CIred_it} 
Suppose $3|\a|\!-\!\ell\!<\!2n$ and $d\!\in\!\Z$. Then there exists a linear map
$$C_d\!:\bigoplus_{k=1}^{\i}H^{2*}(X)^{\otimes k}\lra\Z$$
determined by the  GW-invariants of~$(X,\om_n|_X)$ and 
$\phi_*\!:H_*(X)\!\lra\!H_*(X)$ such~that for all $\mu_1,\ldots,\mu_k\!\in\!H^{2*}(X)$
\begin{enumerate}[label=(2\alph*),leftmargin=*]

\item  $\lr{\mu_1,\ldots,\mu_k}_d^{\phi}=C_d(\mu_1,\ldots,\mu_k)\lr{H^{\lr{\a}_n}}_1^{\phi}$
if $\phi_{\P^{n-1}}\!=\!\eta_n$ or  \eref{CItau_e} is satisfied;

\item $\lr{\mu_1,\ldots,\mu_k}_d^{\phi,c}=C_d(\mu_1,\ldots,\mu_k)\lr{H^{\lr{\a}_n}}_1^{\phi,c}$
if $c\!=\!\eta,\tau$ and either ($\tau_n\eta$) or 
the first bullet condition in~\ref{CIvan_it} holds.

\end{enumerate}
\end{enumerate}
\end{crl}

\vspace{.1in}

\noindent
For example, the genus~0 real GW-invariants~\eref{realNdfn_e} 
of $(\P^{2n-1},\phi)$ with  $\phi\!=\!\eta_{2n},\tau_{2n}$ satisfy
$$\lr{c_1,\ldots,c_k}_d^{\phi}=C_d(c_1,\ldots,c_k)\lr{2n\!-\!1}_1^{\phi}$$
for some $C_d(c_1,\ldots,c_k)\!\in\!\Z$ independent of the choice of~$\phi$.
This implies~\eref{taueta_e}.
Corollary~\ref{inveqCI_crl}\ref{CIred_it} extends~\eref{taueta_e} to
Fano complete intersections $X_{n;\a}\!\subset\!\P^{2n-1}$ with $n\!\in\!2\Z^+$ that are preserved
by both~$\tau_n$ and~$\eta_n$.
The approach to~\eref{taueta_e} in~\cite{Teh} extends to $X_{n;\a}\!\subset\!\P^{n-1}$
without a restriction on~$\a$, but is generally limited to
complete intersections in real symplectic manifolds with large torus actions
and insertions coming from the ambient manifolds.
Corollary~\ref{inveqCI_crl}\ref{CIvan_it} extends the vanishing statement of 
\cite[Theorem~1.10]{Teh} to complete intersection using completely different reasoning.
More generally, physical considerations in~\cite{Wal} suggest that 
the real genus~0 GW-invariants vanish whenever \eref{bndcond_e} 
does not hold, but~\tauetacond{\tau} and~\tauetacond{\eta} are satisfied,
i.e.~when the gluing of the two parts of the moduli space as in~\eref{gluedsp_e2} 
is necessary and possible.\\

\noindent  
If $X^{\phi}$ is orientable or $\ov\fM_{0,k}(X,\be)^{\phi,\eta}\!\neq\!\eset$,
then $\lr{c_1(X),\be}\!\in\!2\Z$ and
\BE{dimcong_e}\dim^{\vir}\ov\fM_{0,k}(X,\be)^{\phi} =\dim^{\vir}\ov\fM_{0,k}(X,\be)^{\phi,c}
\cong n\!-\!3 \mod2;\EE
see~\eref{virdim_e}.
Thus, the real genus~0 GW-invariants~\eref{phinums_e} and~\eref{phinums_e2}
with all insertions $\mu_i\!\in\!H^{2*}(X)$ vanish if  $n\!\in\!2\Z$.
In this case, Theorem~\ref{main2_thm} and Corollary~\ref{inveq_crl} are inutile.
Theorem~\ref{main2_thm} can be extended to odd-degree
cohomology insertions at the cost of adding signs for each summand depending
on the permutation of the odd-degree insertions.
In particular, the formula of Theorem~\ref{main2_thm} is valid without any changes 
if there is only one odd-degree insertion, $\mu$ or~$\mu_i$.
Corollary~\ref{inveq_crl} extends to odd-degree insertions
as a reduction to invariants with at most one even-degree insertion 
which does not increase the number of odd insertions.\\

\noindent
Theorem~\ref{main2_thm} can be extended to the real GW-invariants with
real marked points defined in~\cite{Ge2}.
These invariants are defined by intersecting with the pull-back of 
a homology class~$\Ga$ from the corresponding Deligne-Mumford space of real curves
by the forgetful map; see \cite[Section~1]{Ge2}.
Since the proof of Theorem~\ref{main2_thm} is essentially intersection theory,
it readily fits with the definition of the invariants in~\cite{Ge2}. 
The analogue of the right-hand side of the formula in Theorem~\ref{main2_thm}
would then involve 
Kunneth-style splitting of~$\Ga$ between the real and complex GW-invariants
represented by the diagrams in Figures~\ref{RM_fig}, ~\ref{LHS0_fig}, and~\ref{RHS0_fig}
and all splittings of $\{3,\ldots,k\}$ into three subsets $I^+,I^-,J$.
It would no longer be possible to merge~$I^+$ and~$I^-$ into a single subset,
as done in the proof of Theorem~\ref{main2_thm}.
For a related reason, Theorem~\ref{evenvan_thm} does not extend to 
GW-invariants with real marked points.

\begin{rmk}\label{relspin_rmk}
The homomorphism~\eref{fddfn_e} factors through a similar doubling homomorphism
$$\fd\!: H_2(X,X^{\phi};\Z)\lra H_2(X)_{\phi};$$ 
see \cite[Section~1]{Ge2}.
Let 
$$\prt\!:H_2(X,X^{\phi};\Z)\lra H_1(X^{\phi};\Z)$$ 
denote the boundary homomorphism.
By the proof of \cite[Theorem~6.5]{Ge2} and \cite[Corollary~5.9]{GZ2},
$\ov\fM_{0,k}(X,\be)^{\phi,\tau}$ is orientable if 
\begin{enumerate}[label=($\cO_{\tau}'$),leftmargin=*]
\item $X^{\phi}$ is orientable and there exist $\vp\!\in\!H^2(X;\Z_2)$ 
and $\ka\!\in\!H^1(X^{\phi};\Z_2)$ such that 
\begin{gather}
w_2(TX^{\phi})=\ka^2+\vp|_{X^{\phi}}\qquad\hbox{and}\notag\\
\label{vpcond_e}
\frac12\lr{c_1(X),\fd(\be')}+\lr{\vp,\fd(\be')}
+\lr{\ka,\prt\be'}\in 2\Z\quad\forall~\be'\!\in\!H_2(X,X_{\phi};\Z)\,.
\end{gather}
\end{enumerate}
The two requirements on~$(\vp,\ka)$ imply that $(X,X^{\phi})$ admits 
a relatively spin sub-structure in the sense of \cite[Definition~5.5]{GZ2}
and that it can be chosen so that the $(\phi,\tau)$-moduli space is orientable,
respectively.
The relatively spin condition of \cite[Theorem~8.1.1]{FOOO} is  
the $\ka\!=\!0$ case of the first requirement;
$(\cO_{\tau})$ is effectively the $\ka\!=\!0$, $\vp\!=\!w_2(E)$ case of~$(\cO_{\tau}')$.
Theorems~\ref{main2_thm} and~\ref{evenvan_thm} and Corollaries~\ref{inveq_crl} and~\ref{inveq_crl2}
can be extended with the assumption~$(\cO_{\tau})$ relaxed to~$(\cO_{\tau}')$.
The key difference is that this would introduce the sign $(-1)^{\lr{\vp,\be_1}}$
over~$\cN_{\be_1,\be_2}$ as happens in~\cite{GZ5}.
This sign can be absorbed into the numbers~\eref{GWdfn_e} whenever~\eref{bndcond_e} is satisfied, 
but this would result in a different dependence on the complex GW-invariants for
the $\tau$- and $\eta$-invariants.
It can be absorbed into the numbers~\eref{phinums_e} with $c\!=\!\tau$,
similarly to~\eref{realNdfn_e}, if $\phi^*\vp\!=\!\vp$;
this is the case for $\tau$-relatively spin structures in the sense of 
\cite[Definition~3.11]{FOOO9}.
We avoid such a sign modification in~\eref{phinums_e} by treating $(\P^{2n-1},\tau_{2n})$
as a special case of~$(\cO_{\tau})$.\\
\end{rmk}

\noindent
Throughout this section and Sections~\ref{Rmainpf_sec}-\ref{mainpf_sec},
the moduli spaces $\ov\fM_{0,k}(X,\be)^{\phi,c}$ and 
their glued, constrained, and complex versions refer to regularizations of these spaces. 
If $(X,\om,\phi)$ is semi-positive in the sense of \cite[Definition~1.2]{RealRT},
e.g.~$\P^{2n-1}$,
the latter are obtained by choosing a generic $J\!\in\!\cJ_{\om}^{\phi}$ and 
the invariants are defined through pairing with the pseudocycles 
determined by the moduli spaces.
Both proofs of Theorem~\ref{main2_thm}, outlined at the end of Section~\ref{intro_sec},
are completely geometric in this case
and have no relation to virtual fundamental class (VFC) constructions.
The situation with the first proof in the general case is analogous to 
that with the WDVV and Getzler's relations in complex GW-theory: 
as the relation of Proposition~\ref{DMrelR_prp} is universal (induced from the moduli of domains), 
its validity is independent of the choice of VFC construction and 
depends only on properties of GW-invariants any such construction must yield to be relevant. 
The relevant properties are the $g\!=\!0$ case of Kontsevich-Manin's axioms 2.2.0 ({\it Effectivity}), 
2.2.1 ({\it $S_n$-invariance}), 2.2.2 ({\it Grading}), 2.2.4 ({\it Divisor}), and 
2.2.6 ({\it Splitting}) in~\cite{KM}
and their real analogues ({\it Splitting} at interior nodes only). 
Suitable adaptations to the real case of the usual VFC constructions of \cite{LT,FO}
are carried out  in \cite[Section~7]{Sol}, \cite[Section~7]{FOOO9}, 
and \cite[Section~2.3]{Teh}.
The invariants arising from these adaptations satisfy the real analogues of the first
three axioms above for trivial reasons.
The proofs of the last two axioms in the complex case readily extend to the real case.

\section{A homology relation for $\R\ov\cM_{0,3}$}
\label{DMrelR_sec}

\noindent
In this section, we formulate and prove a codimension~2 relation on 
the Deligne-Mumford moduli space~$\R\ov\cM_{0,3}$ of real genus~0 curves 
with 3~pairs of marked points; see Proposition~\ref{DMrelR_prp}.
Its proof involves a detailed topological description of~$\R\ov\cM_{0,3}$.\\

\noindent
For $c\!=\!\tau,\eta$ and $k\!\in\!\Z^+$, 
denote by $\ov\cM_{0,k+1}^c$ the moduli space of $c$-real rational curves 
with $k\!+\!1$ conjugate pairs of marked points.
As it is convenient to designate one of the pairs as principal,
we index the pairs by the set $\{0,1,\ldots,k\}$ and 
view~0 as the principal index.
Thus, the main stratum of~$\ov\cM_{0,k+1}^c$ is the quotient~of
$$\big\{\big((z_0^+,z_0^-),(z_1^+,z_1^-),\ldots,(z_k^+,z_k^-)\big)\!:\,
z_i^{\pm}\!\in\!\P^1,\,z^+\!=\!c(z_i^-),\,
z_i^+\!\neq\!z_j^+,z_j^-~\forall\,i\!\neq\!j,\,
z_i^+\!\neq\!z_i^-\big\}$$
by the natural action of the subgroup $\PSL_2^c\C\!\subset\!\PSL_2\C$ 
of automorphisms of~$\P^1$ commuting with~$c$.
Many notions concerning~$\ov\cM_{0,k+1}^c$ are defined below with respect to 
the index~0, which in these cases is implicitly understood.\\

\noindent
The moduli spaces $\ov\cM_{0,k+1}^{\eta}$ and $\ov\cM_{0,k+1}^{\tau}$
are $(2k\!-\!1)$-dimensional manifolds with the same boundary,
$$\prt\ov\cM_{0,k+1}^{\tau}=\prt\ov\cM_{0,k+1}^{\eta}\,.$$
The latter  consists of the curves with no irreducible component fixed by the involution;
the strata of $\ov\cM_{0,k+1}^{\tau}$ with two invariant bubbles attached 
at a real node are of codimension~1, but not a boundary for this space. 
Gluing along the common boundary, we obtain the moduli space 
$$\ov\cM_{0,k+1}^{\R}\equiv \ov\cM_{0,k+1}^{\tau}\cup \ov\cM_{0,k+1}^{\eta}\,;$$ 
it is a $(2k\!-\!1)$-dimensional manifold without boundary.
We will use $\R\ov\cM_{0,k+1}$ to refer to any one of these three moduli spaces
and $(z_i^+,z_i^-)$ to denote the $i$-th conjugate pair of marked points.\\

\noindent
If $c\!=\!\tau,\eta$, $\ov\cM_{0,2}^c$ is a compact connected one-dimensional
manifold with boundary and is therefore an interval.
It has a canonical orientation induced by requiring the boundary point corresponding 
to the two-component curve with the marked points~$z_0^+$ and~$z_1^-$ on the same component
to be the initial point of the interval. 
An explicit orientation-preserving isomorphism is given by the cross-ratio
\begin{gather}\label{Mcisom_e}
\ov\cM_{0,2}^c\lra \bI\!\equiv\![0,\infty] , \qquad
\big[(z_0^+,z_0^-),(z_1^+,z_1^-)\big]\lra 
(-1)^c\,\frac{z_1^+\!-\!z_0^+}{z_1^-\!-\!z_0^+}:
\frac{z_1^+\!-\!z_0^-}{z_1^-\!-\!z_0^-}\,,\\
\hbox{where}\qquad
(-1)^c=\begin{cases} 1,&\hbox{if}~c\!=\!\tau; \\ 
-1,&\hbox{if}~c\!=\!\eta;\end{cases} \notag
\end{gather}
with $z_0^+\!=\!0$, the above element of $\ov\cM_{0,2}^c$ is sent to $|z_1^+|^2$.
For $k\!\ge\!2$, $\ov\cM_{0,k+1}^c$ is oriented using the first element in 
each conjugate pair~$(z_i^+,z_i^-)$ with $i\!\ge\!2$ to orient the general
fiber of the forgetful morphism $\ov\cM_{0,k+1}^c\!\lra\!\ov\cM_{0,2}^c$.
Since the boundaries of $\ov\cM_{0,k+1}^{\eta}$ and $\ov\cM_{0,k+1}^{\tau}$ 
are oriented in the same way, we obtain an orientation on~$\ov\cM_{0,k+1}^{\R}$
by reversing the orientation on~$\ov\cM_{0,k+1}^{\eta}$. 
An explicit orientation-preserving isomorphism of $\ov\cM_{0,2}^{\R}$ with 
$S^1\!\equiv\!\R\!\sqcup\!\{\i\}$ 
is given by the map in~\eref{Mcisom_e} with $(-1)^c$ dropped. 
The general fibers of the forgetful morphism $\ov\cM_{0,k+1}^{\R}\!\lra\!\ov\cM_{0,2}^{\R}$
are again oriented  using the first element in  each conjugate pair~$(z_i^+,z_i^-)$ 
with $i\!\ge\!2$.

\begin{lmm}\label{DMorient_lmm}
Let $k\!\in\!\Z^+$ and  let $\R\ov\cM_{0,k+1}$ denote
$\ov\cM_{0,k+1}^{\tau}$, $\ov\cM_{0,k+1}^{\eta}$, or  $\ov\cM_{0,k+1}^{\R}$.
\begin{enumerate}[label=(\arabic*),leftmargin=*]
\item For every $i\!=\!0,1,\ldots,k$, the automorphism of $\R\ov\cM_{0,k+1}$
interchanging the marked points in the $i$-th conjugate pair 
is orientation-reversing.
\item For all $i,j\!=\!0,1,\ldots,k$, the automorphism of $\R\ov\cM_{0,k+1}$
interchanging the $i$-th and $j$-th conjugate pairs of marked points
is orientation-preserving.
\end{enumerate}
\end{lmm}

\begin{proof}
(1) For $i\!=\!0,1$, this automorphism interchanges the two boundary points 
of $\ov\cM_{0,2}^c$ with $c\!=\!\tau,\eta$.
Thus, it is orientation-reversing on the base of the forgetful morphism
\BE{RMDforg_e}\R\ov\cM_{0,k+1}\lra\R\ov\cM_{0,2}\EE 
for every $k\!\ge\!1$.
Since this automorphism takes a general fiber of~\eref{RMDforg_e} to another general fiber
in an orientation-preserving way, it is orientation-reversing on~$\R\ov\cM_{0,k+1}$.
For $i\!\ge\!2$, the automorphism of $\R\ov\cM_{0,k+1}$ interchanging the marked points 
in the $i$-th conjugate pair  takes a general fiber of~\eref{RMDforg_e}
to itself in an orientation-reversing way.
Thus, it is again orientation-reversing on~$\R\ov\cM_{0,k+1}$.\\

\noindent
(2) If $i,j\!\le\!1$ or $i,j\!\ge\!2$, the automorphism of $\R\ov\cM_{0,k+1}$
interchanging $i$-th and $j$-th conjugate pairs of marked points 
takes a general fiber of~\eref{RMDforg_e} to itself in an orientation-preserving way
and so is orientation-preserving on~$\R\ov\cM_{0,k+1}$.
Thus, it remains to consider the case $i\!=\!1$ and $j\!=\!2$.
Since the corresponding automorphism of $\R\ov\cM_{0,k+1}$, with $k\!\ge\!2$,
takes a general fiber of the forgetful morphism $\R\ov\cM_{0,k+1}\!\lra\!\R\ov\cM_{0,3}$
to itself in an orientation-preserving way,
it is sufficient to check that it is orientation-preserving for $k\!=\!2$.
The latter is the case if and only~if the forgetful morphisms
\BE{f1f2_e}
f_1,f_2\!: \R\ov\cM_{0,3}\lra \R\ov\cM_{0,2}\,,\quad
\begin{aligned}
f_1\big([(z_0^+,z_0^-),(z_1^+,z_1^-),(z_2^+,z_2^-)]\big)&=
\big[(z_0^+,z_0^-),(z_1^+,z_1^-)\big],\\
f_2\big([(z_0^+,z_0^-),(z_1^+,z_1^-),(z_2^+,z_2^-)]\big)&= 
\big[(z_0^+,z_0^-),(z_2^+,z_2^-)\big],
\end{aligned}\EE
induce the same orientation on~$\R\ov\cM_{0,3}$.\\

\noindent
It is enough to check that $f_1$ and $f_2$ induce the same orientation on the tangent space
at a three-component curve~$\cC$ with~$z_1^+$ and~$z_2^+$ on 
the same bubble component~$\cC^{\C}$, i.e.~as in the first diagram in Figure~\ref{RM_fig},
but with the label~$0$ interchanged with~$2$ and the label~$\bar{0}$ interchanged with~$\bar{2}$.
The restrictions of~$f_1$ and~$f_2$ to the space~$\Ga$ of such curves are the same
and take~$\Ga$ isomorphically onto~$\R\ov\cM_{0,2}$;
thus, $f_1$ and~$f_2$ induce the same orientations on~$T_{\cC}\Ga$.
The vertical tangent bundles of~$f_1$ and~$f_2$ along~$\cC$ are canonically isomorphic
to the normal bundle of~$\Ga$ in~$\R\ov\cM_{0,3}$.
The orientation of the fiber of the vertical tangent bundle of~$f_1$ at~$\cC$ 
given by varying~$z_2^+$
is the complex orientation of the tangent node of the real component~$\cC^{\R}$
of~$\cC$ at the node separating~$\cC^{\R}$ from~$\cC^{\C}$.
The same is the case for the orientation of the vertical tangent bundle of~$f_2$ at~$\cC$ 
given by varying~$z_1^+$.
Thus, the orientations of the normal bundle of~$\Ga$ in~$\R\ov\cM_{0,3}$ 
with respect to the orientations induced by~$f_1$ and~$f_2$ are the same.
This implies that the orientations induced by~$f_1$ and~$f_2$ 
on $\R\ov\cM_{0,3}$ are the same as well.
\end{proof}

\noindent
For $i\!=\!1,2$, let $\Ga_i\!\subset\!\R\ov\cM_{0,3}$ denote 
the closure of the subset~$\mr\Ga_i$ consisting of the three-component real curves~$(\cC,c)$
such that the marked point~$z_i^+$ lies on the same component as the marked point~$z_0^+$.
Let $\Ga_{\bar{i}}\!\subset\!\R\ov\cM_{0,3}$ denote 
the closure of the subset~$\mr\Ga_{\bar{i}}$ consisting of the three-component real curves~$(\cC,c)$
such that the marked point~$z_i^-$ 
lies on the same component as the marked point~$z_0^+$.
The stability condition implies that such a three-component curve has
\begin{enumerate}[leftmargin=23pt]
\item[($\R$)] a component~$\cC^{\R}$ preserved by~$c$ and containing the conjugate pair 
$(z_{3-i}^+,z_{3-i}^-)$ of marked points and 
a conjugate pair $(z_{\bu}^+,z_{\bu}^-)$ of nodes, and
\item[($\C$)] a pair of conjugate components, with the component~$\cC^{\C}$ containing 
the marked point~$z_0^+$ also carrying the marked point~$z_i^+$ in the case of~$\Ga_i$ and 
$z_i^-$ in the case of~$\Ga_{\bar{i}}$;
\end{enumerate}
see Figures~\ref{fig_eta}-\ref{fig_real}.
We will take~$z_{\bu}^+\!\in\!\cC^{\R}$ to be the node identified with a point 
$z^{\C}\!\in\!\cC^{\C}$.
Thus, there are canonical isomorphisms
\BE{Gasplit_e}\Ga_i\approx \R\ov\cM_{0,2}\times \ov\cM_{0,3} \qquad\hbox{and}\qquad
\Ga_{\bar{i}}\approx \R\ov\cM_{0,2}\times \ov\cM_{0,3}^-\,,\EE
where the superscript $-$ indicates that one of the marked points 
(the one corresponding to~$z_i^-$) is decorated with the minus sign.
Following the principle introduced in~\cite{Ge2}, 
we define the canonical orientation of~$\ov\cM_{0,3}^-$ to be the opposite of 
the canonical (complex) orientation of~$\ov\cM_{0,3}$
and then use~\eref{Gasplit_e} to orient~~$\Ga_i$ and~$\Ga_{\bar{i}}$.
Thus, the orientation on~$\Ga_i$ is the same as the one induced by the natural isomorphism
$\Ga_i\!\approx\!\R\ov\cM_{0,2}$, while
the orientation on~$\Ga_{\bar{i}}$ is the opposite of the one induced by the natural isomorphism
$\Ga_{\bar{i}}\!\approx\!\R\ov\cM_{0,2}$.
The canonical orientation of $\R\ov\cM_{0,2}$ is defined above, but
this choice of the orientation does not  affect the validity of 
Lemma~\ref{DMgluing_lmm} or Proposition~\ref{DMrelR_prp}.
Whenever $\R\ov\cM_{0,3}\!=\!\ov\cM_{0,3}^c$ for a specific $c\!=\!\tau,\eta,\R$,
we will write $\Ga_*^c$, where  $*\!=\!i,\bar{i}$ with $i\!=\!1,2$, for~$\Ga_*$.\\

\noindent
With $\Ga\!=\!\Ga_i,\Ga_{\bar{i}}$, $i\!=\!1,2$,  as in~\eref{Gasplit_e}, let 
$$L_{\Ga}^{\R}\lra\R\ov\cM_{0,2} \qquad\hbox{and}\qquad 
L_{\Ga}^{\C}\lra  \ov\cM_{0,3},\ov\cM_{0,3}^-$$
be the universal tangent line bundles 
at the marked points~$z_{\bu}^+$ and $z^{\C}$, respectively, and
$$L_{\Ga}=\pi_1^*L_{\Ga}^{\R}\otimes_{\C} \pi_2^*L_{\Ga}^{\C}\lra \Ga\,,$$
where $\pi_1,\pi_2$ are the component projection maps.

\begin{lmm}\label{DMgluing_lmm}
Let $\R\ov\cM_{0,3}$ denote
$\ov\cM_{0,3}^{\tau}$, $\ov\cM_{0,3}^{\eta}$, or  $\ov\cM_{0,3}^{\R}$.
\begin{enumerate}[label=(\arabic*),leftmargin=*]

\item For $i\!=\!1,2$, the automorphism of $\R\ov\cM_{0,3}$
interchanging the marked points in the $i$-th conjugate pair 
restricts to an orientation-reversing isomorphism from~$\Ga_i$
to~$\Ga_{\bar{i}}$ and canonically lifts to 
a $\C$-linear isomorphism from~$L_{\Ga_i}$ to~$L_{\Ga_{\bar i}}$.

\item The automorphism of $\R\ov\cM_{0,3}$ interchanging the 1st  and 2nd conjugate 
pairs of marked points  restricts to an orientation-preserving isomorphism from~$\Ga_1$
to~$\Ga_2$ and canonically lifts to 
a $\C$-linear isomorphism from $L_{\Ga_1}$  to $L_{\Ga_2}$.

\item For $i\!=\!1,2$,
the oriented normal bundle of $\mr\Ga\!=\!\mr\Ga_i,\mr\Ga_{\bar i}$ in $\R\ov\cM_{0,3}$ 
is isomorphic to $L_{\Ga}$ with its canonical complex orientation.
\end{enumerate}
\end{lmm}

\begin{proof}
(1,2)  It is immediate that the automorphism in~(1) interchanges $\Ga_i$ and~$\Ga_{\bar{i}}$
and  the automorphism in~(2) interchanges $\Ga_1$ and~$\Ga_2$.
These restrictions respect the component moduli spaces in~\eref{Gasplit_e}
and induce the identity on the first component (the second component is a point).
Given our choice of orientations, 
the domain and target orientations of the automorphism in~(1) are opposite,
while the domain and target orientations of the automorphism  in~(2) are the same.
This implies the first parts of the first two statements in the lemma.
Since these automorphisms respect the component moduli spaces in~\eref{Gasplit_e},
they canonically lift to all universal tangent line bundles for these moduli spaces
and thus to~$L_{\Ga_i}$. 
They act by the identity on the tangent spaces at~$z_{\bu}^+$ and~$z^{\C}$
and thus $\C$-linearly on~$L_{\Ga}$.\\

\noindent
(3) The restriction of the forgetful morphism $f_1$ in~\eref{f1f2_e}
to~$\Ga_2$ is an orientation-preserving isomorphism.
By the definition of the orientation on~$\R\ov\cM_{0,3}$,
the vertical tangent bundle along~$\mr\Ga_2$ is thus oriented by the complex
orientation of $L_{\Ga_2}^{\C}\!\approx\!L_{\Ga_2}$.
Since the vertical tangent bundle  of~$f_1$ along~$\mr\Ga_2$ is canonically isomorphic
to the normal bundle of~$\Ga_2$ in~$\R\ov\cM_{0,3}$,
this implies the last statement of the lemma in the $\mr\Ga\!=\!\mr\Ga_2$ case. 
The remaining three cases follow from this case, the first two statements of the lemma,
and the $k\!=\!2$ case of Lemma~\ref{DMorient_lmm}.
\end{proof}

\begin{prp}\label{DMrelR_prp}
Let $\R\ov\cM_{0,3}$ denote
$\ov\cM_{0,3}^{\tau}$, $\ov\cM_{0,3}^{\eta}$, or  $\ov\cM_{0,3}^{\R}$.
The submanifolds $\Ga_1,\Ga_{\bar{1}},\Ga_2,\Ga_{\bar{2}}$ of  $\R\ov\cM_{0,3}$
determine relative cycles  in $(\R\ov\cM_{0,3},\prt\R\ov\cM_{0,3})$ and
\BE{DMrelR_e}[\Ga_1]+[\Ga_{\bar{1}}]=[\Ga_2]+[\Ga_{\bar{2}}]\in 
H_1\big(\R\ov\cM_{0,3},\prt\R\ov\cM_{0,3};\Q\big).\EE
\end{prp}

\noindent
Since $\prt\Ga_i$ and $\prt\Ga_{\bar{i}}$ are contained in $\prt\R\ov\cM_{0,3}$,
only the second statement of this proposition remains to be established.
The relation~\eref{DMrelR_e} in fact holds over~$\Z$; 
though we do not need this stronger statement, we give two separate reasons 
for it in Remarks~\ref{3term_rmk} and~\ref{DMstr_rmk}.

\begin{proof}[{\bf{\emph{Proof for $\ov\cM_{0,3}^{\eta}$}}}]
The boundary of $\ov\cM_{0,3}^{\eta}$ has four components, which we denote by
$S_{12}$, $S_{1\bar2}$, $S_{\bar12}$, and $S_{\bar1\bar2}$,
which contain the two-component curves with the points $\{z_1^+,z_2^+\}$,  
$\{z_1^+,z_2^-\}$, $\{z_1^-,z_2^+\}$, and $\{z_1^-,z_2^-\}$,
respectively, on the same component as the base point~$z_0^+$;
each of them is isomorphic to~$S^2$.
The forgetful~morphism
\BE{etaforg_e}\ov\cM_{0,3}^{\eta}\lra \ov\cM_{0,2}^{\eta}\approx \bI\equiv [0,\i]\EE
is a singular fibration; see Figure~\ref{fig_eta}. 
The fiber over every interior point is a sphere with four special points corresponding 
to the strata where $z_2^+$ collides with $z_0^+$, $z_0^-$, $z_1^+$, or~$z_1^-$. 
The fiber over the boundary point $0\!\in\!\bI$ consists of the spheres~$S_{12}$ and~$S_{1\bar2}$
joined together by the interval~$\Ga_1^{\eta}$ defined above.
The fiber over the boundary point $\i\!\in\!\bI$ consists 
of the spheres~$S_{\bar12}$ and~$S_{\bar1\bar2}$
joined together by the interval~$\Ga_{\bar{1}}^{\eta}$.
The lines~$\Ga_2^{\eta}$ and~$\Ga_{\bar2}^{\eta}$ connect the boundary spheres in the two fibers:
$S_{12}$ with $S_{\bar12}$ and~$S_{1\bar2}$ with~$S_{\bar1\bar2}$, respectively.\\

\begin{figure}
\begin{center}
\leavevmode
\includegraphics[width=0.8\textwidth ]{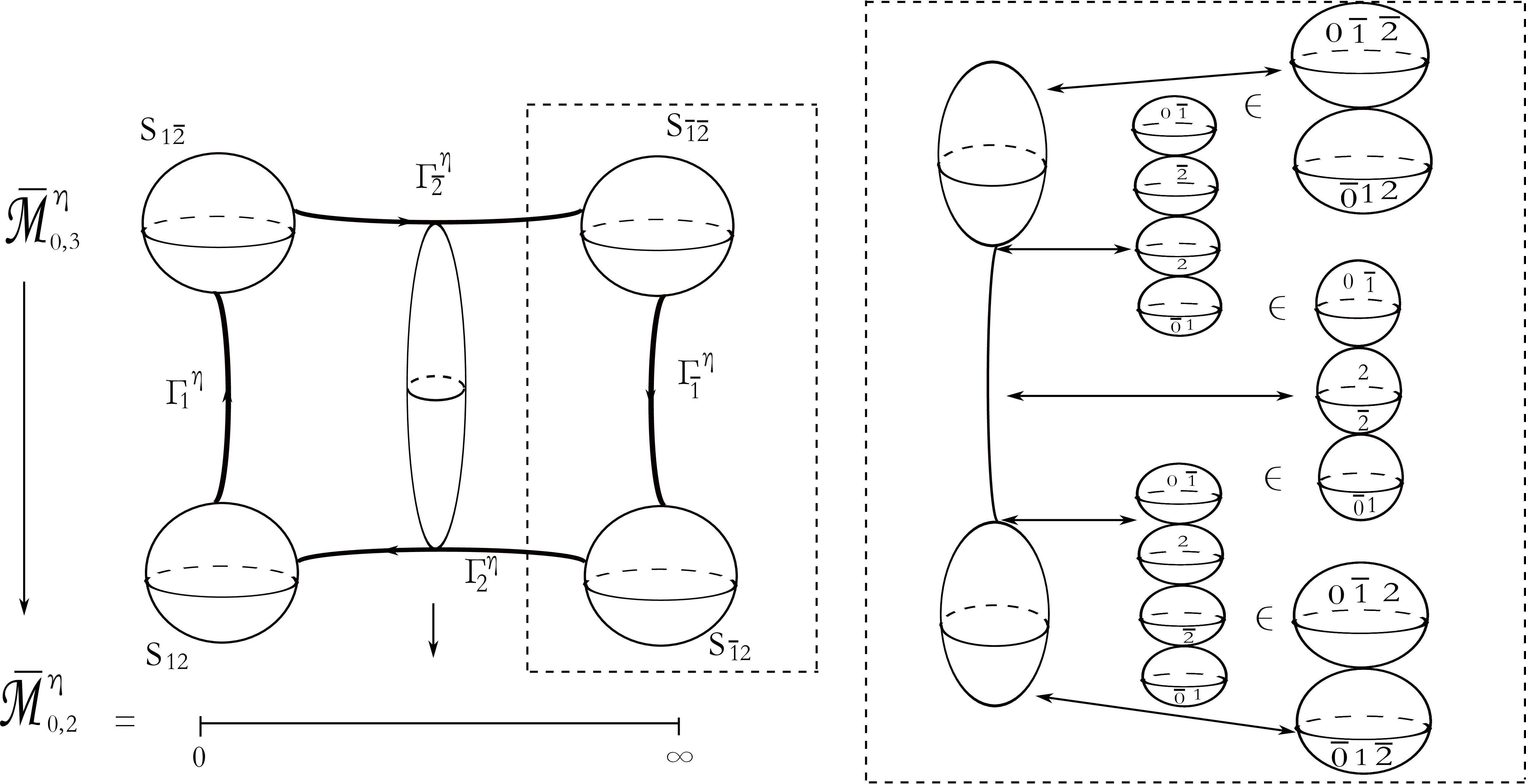}
\end{center}
\caption{The moduli space $\ov\cM_{0,3}^\eta$ as a fibration over $\ov\cM_{0,2}^\eta$;
the labels $i$ and $\bar{i}$ indicate 
the marked points $z_i^+$ and $z_i^-$, respectively.}
\label{fig_eta}
\end{figure}

\noindent
Let 
$$H_1(\ov\cM_{0,3}^{\eta})\stackrel{j_*}{\lra}
H_1(\ov\cM_{0,3}^{\eta},\prt\ov\cM_{0,3}^{\eta})
\stackrel{\prt}{\lra} H_0(\prt\ov\cM_{0,3}^{\eta})$$
denote the homomorphisms in the homology long exact sequence for the pair
$(\ov\cM_{0,3}^{\eta},\prt\ov\cM_{0,3}^{\eta})$.
With the canonical orientations on~$\Ga_{i}^{\eta}$ and~$\Ga_{\bar{i}}^{\eta}$ 
described above 
$$\prt\big[\Ga_1^{\eta}-\Ga_{\bar2}^{\eta}+\Ga_{\bar1}^{\eta}-\Ga_2^{\eta}\big]=0;$$
see Figure~\ref{fig_eta}.
Thus, $[\Ga_1^{\eta}\!-\!\Ga_{\bar2}^{\eta}\!+\!\Ga_{\bar1}^{\eta}\!-\!\Ga_2^{\eta}]$ 
is the image of an element of $H_1(\ov\cM_{0,3}^{\eta})$ under~$j_*$.
A representative~$\Ga^{\eta}$ for this class is obtained by connecting the end points 
of the line segments inside each boundary sphere.
This loop can be homotoped away from the fibers over $0,\i\!\in\!\bI$ 
by smoothing out the nodes.   
The resulting loop in~$S^2\!\times\!\R^+$ is therefore contractible and hence 
is trivial in $H_1(\ov\cM_{0,3}^{\eta})$.
This implies~\eref{DMrelR_e} in the $\eta$~case.
\end{proof}

\begin{proof}[{\bf{\emph{Proof for $\ov\cM_{0,3}^{\tau}$}}}]
In comparison with~\eref{etaforg_e}, the forgetful morphism
$$\ov\cM_{0,3}^{\tau}\lra \ov\cM_{0,2}^{\tau}\approx \bI\equiv[0,\i]$$
has an additional singular value: the point $1\!\in\!\bI$ 
corresponding to the two-component curve with~$z_0^+$ and~$z_1^+$ on separate invariant bubbles; 
see Figure~\ref{fig_tau}. 
The fiber~$F_1$ over this point consists of two copies of~$\R\P^2$ joined along
a non-contractible circle in each copy or equivalently the quotient of~$S^2$
by the action of the antipodal map on the equator only.
The complement of the common circle in one copy of~$\R\P^2$ consists of
the two-component curves, with each component fixed by the involution,
with~$z_2^+$ on the same component as~$z_0^+$;
the complement in the other copy consists of
the two-component curves, with each component fixed by the involution,
with~$z_2^+$ on the same component as~$z_1^+$.
The circle corresponds to the three-component curves
with each component fixed by the involution and $z_2^+$ on the middle component;
see Figure~\ref{fig_tau}.\\

\begin{figure}
\begin{center}
\leavevmode
\includegraphics[width=0.8\textwidth  ]{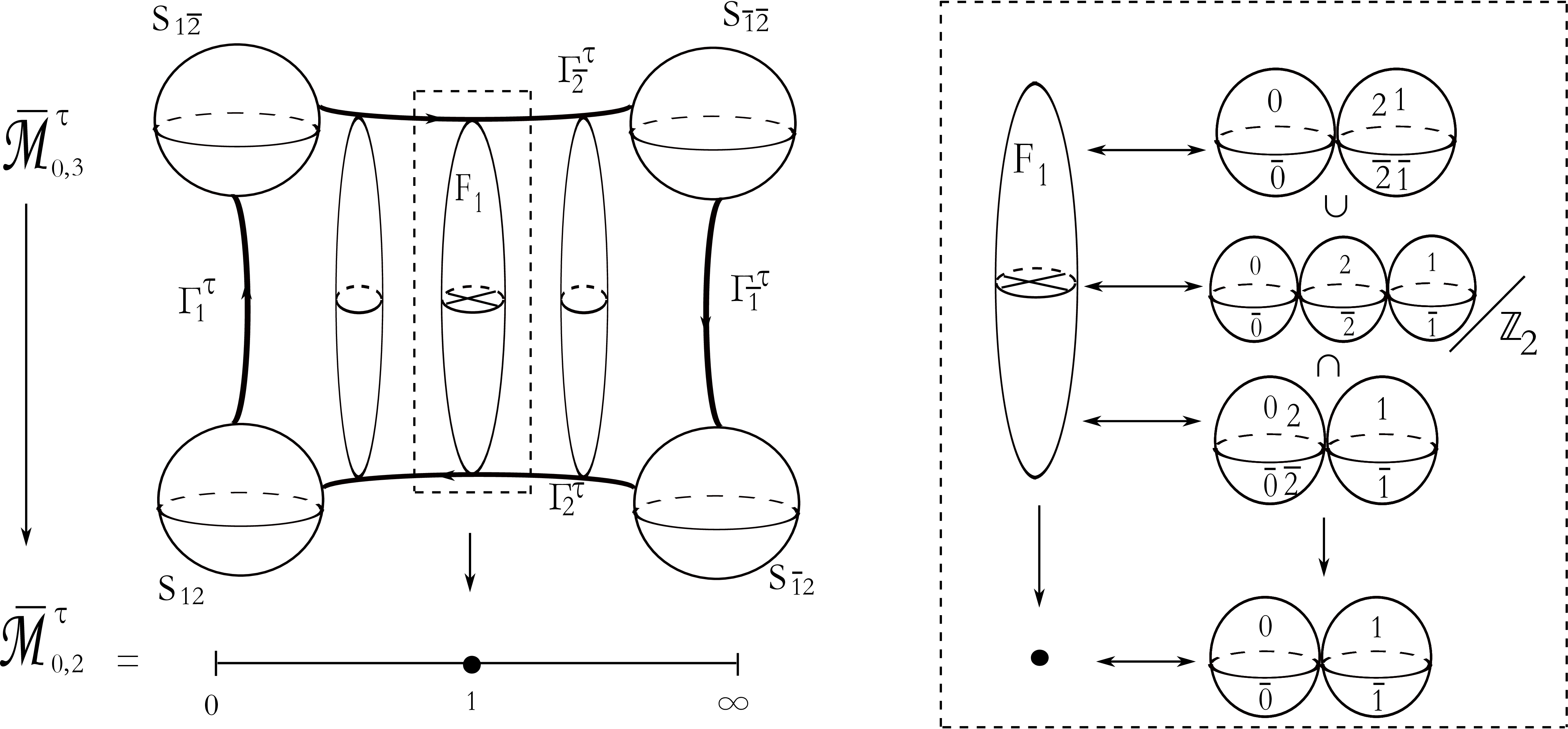}
\end{center}
\caption{The moduli space $\ov\cM_{0,3}^\tau$ as a fibration over $\ov\cM_{0,2}^\tau$;
the labels $i$ and $\bar{i}$ indicate 
the marked points $z_i^+$ and $z_i^-$, respectively.}
\label{fig_tau}
\end{figure}

\noindent
By the same reasoning as in the $\eta$ case, the class
$[\Ga_1^{\tau}\!-\!\Ga_{\bar2}^{\tau}\!+\!\Ga_{\bar1}^{\tau}\!-\!\Ga_2^{\tau}]$
is in the image of an element in $H_1(\ov\cM_{0,3}^{\tau})$
which can be represented by a loop in $\ov\cM_{0,3}^{\tau}$ away 
from the fibers over $0,\i\in\!\bI$. 
This loop can be homotoped to a loop in the special fiber~$F_1$.
Since $\pi_1(F_1)\!\approx\!\Z_2$,  it still represents the zero class 
in~$H_1(\ov\cM_{0,3}^{\tau})$ with $\Q$-coefficients.
\end{proof}

\begin{proof}[{\bf{\emph{Proof for $\ov\cM_{0,3}^{\R}$}}}]
The fibers of the forgetful morphism
\BE{DMforgR_e} \ov\cM_{0,3}^{\R}\lra \ov\cM_{0,2}^{\R}\approx S^1\EE
away from the identification points of $\ov\cM_{0,2}^{\eta}$
and $\ov\cM_{0,2}^{\tau}$
are as described in the $\eta,\tau$ cases; see Figure~\ref{fig_real}.
A fiber over either of the two identification points, $0,\i\!\in\!\bI$, 
consists of two spheres joined by a circle.
The submanifolds~$\Ga_*^{\R}$ with $*\!=\!1,\bar1,2,\bar2$ form 4 loops in~$\ov\cM_{0,3}^{\R}$:
$$\Ga_1^{\R}=\Ga_1^{\tau}-\Ga_1^\eta\, , \qquad
\Ga_{\bar1}^{\R}=\Ga_{\bar1}^{\tau}-\Ga_{\bar1}^\eta\,,\qquad
\Ga_2^{\R}=\Ga_2^{\tau}-\Ga_2^\eta\, , \qquad
\Ga_{\bar2}^{\R}=\Ga_{\bar2}^{\tau}-\Ga_{\bar2}^\eta\,.$$
Connecting the points of these loops on each of the four spheres by paths as before, 
we obtain the loops
$\Ga^{\eta}\!\subset\!\ov\cM_{0,3}^{\eta}$ and 
$\Ga^{\tau}\!\subset\!\ov\cM_{0,3}^{\tau}$ as in the $\eta,\tau$ cases above 
so~that 
$$\big[\Ga_1^{\R}-\Ga_{\bar2}^{\R}+\Ga_{\bar1}^{\R}-\Ga_2^{\R}\big]
=[\Ga^{\tau}]-[\Ga^{\eta}].$$
By the $\eta,\tau$ cases above, $[\Ga^{\eta}]$ and $[\Ga^{\tau}]$
are zero  in $H_1(\ov\cM_{0,3}^{\eta})$ and 
$H_1(\ov\cM_{0,3}^{\tau})$, respectively.
\end{proof}

\begin{figure}
\begin{center}
\leavevmode
\includegraphics[width=0.6\textwidth ]{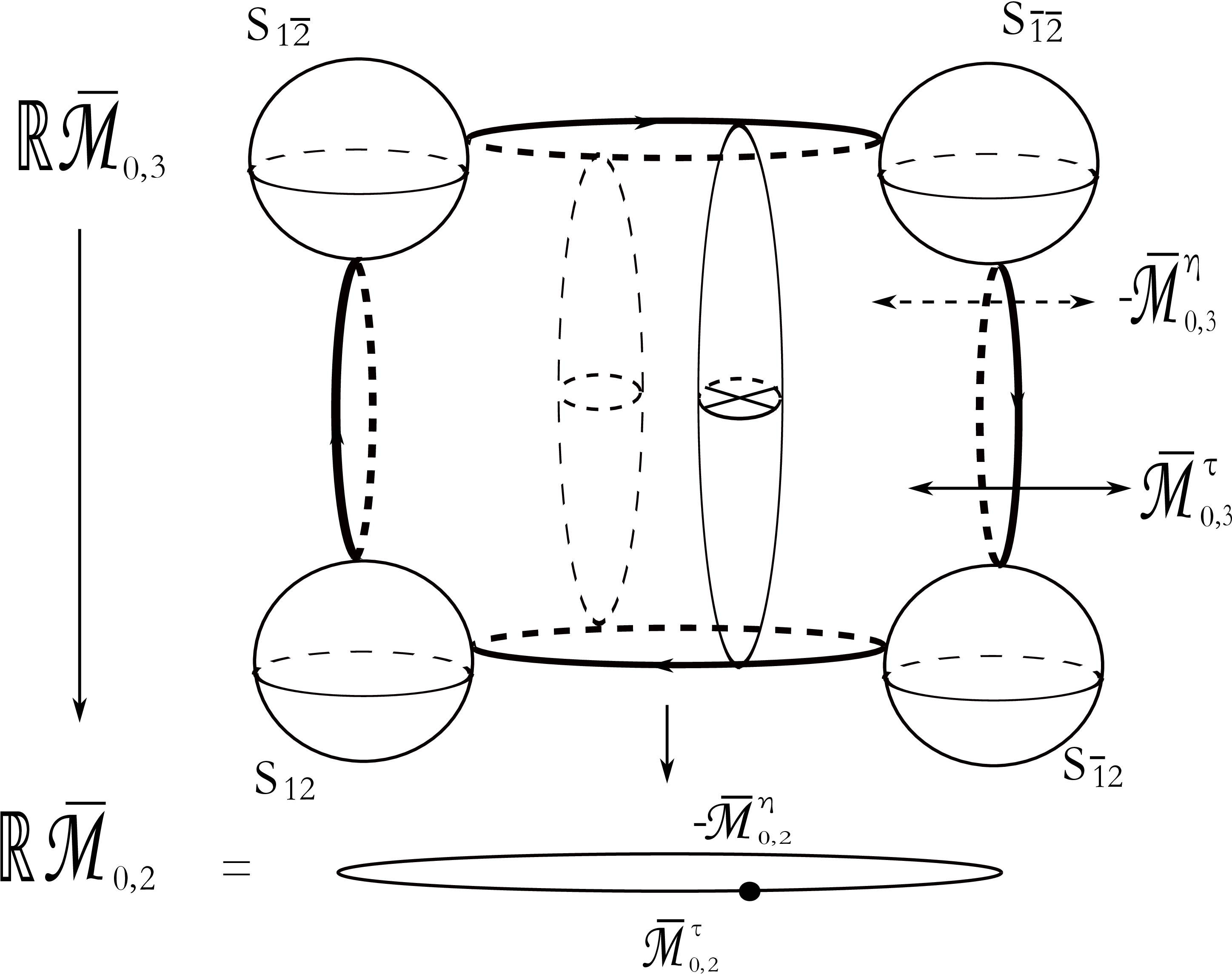}
\end{center}
\caption{The moduli space $\R\ov\cM_{0,3}$ as a fibration over $\R\ov\cM_{0,2}$;
the labels $i$ and $\bar{i}$ indicate the marked points $z_i^+$ and $z_i^-$, respectively.}
\label{fig_real}
\end{figure}

\begin{rmk}\label{3term_rmk}
The same argument can be used to obtain 3-term relations in 
$H_1(\R\ov\cM_{0,3},\prt\R\ov\cM_{0,3})$ by going diagonally in 
Figures~\ref{fig_eta}-\ref{fig_real}.
While these relations are nominally stronger than~\eref{DMrelR_e},
we do not see any applications for them at this point and 
they have a less appealing appearance than~\eref{DMrelR_e}.
On the other hand, they can be used to conclude  that~\eref{DMrelR_e}
holds over~$\Z$ as follows.
Let $\al$ denote a nontrivial loop in the fiber~$F_1$ in the proof of
the $\tau$ case of Proposition~\ref{DMrelR_prp}.
The loops formed by the upper left and lower right triangles equal
to $\ve_l\al$ and $\ve_r\al$ in homology, for some $\ve_l,\ve_r\!\in\!\{0,1\}$. 
Pulling back the loops to $\ov\fM_{0,3}^{\tau}(\P^2,1)$ by~\eref{Rforgmap_e}, 
evaluating on 3 conjugate pairs of lines over~$\Z_2$, and 
using Proposition~\ref{KunnethSplit_prp}, we find that each of the three segments 
in each of the triangles contributes $1\!\in\!\Z_2$ 
(the number of real lines through a non-real point) to the total count for the triangle.  
Thus, $\ve_l,\ve_r\!=\!1$
(the preimage of the loop~$\al$ in fact corresponds to the number of real lines through 2~real 
points in~$\P^2$).
This implies that the loop
 $\Ga_1^{\tau}\!-\!\Ga_{\bar2}^{\tau}\!+\!\Ga_{\bar1}^{\tau}\!-\!\Ga_2^{\tau}$,
which is the sum of the two triangular loops, is contractible.
Thus, \eref{DMrelR_e} holds over~$\Z$.
\end{rmk}

\begin{rmk}\label{DMstr_rmk}
A local model for~\eref{DMforgR_e} near the intersection point of an~$S^2$
and~$S^1$ in the same fiber is given~by
\BE{S2S1loc_e1}\R\!\times\!\C\lra\R, \qquad (t,z)\lra t|z|^2\,.\EE
Local models for~\eref{DMforgR_e} around $S^2\!\equiv\!\C\!\sqcup\!\{\i\}$ and 
$S^1\!\equiv\!\R\!\sqcup\!\{\i\}$ are given~by
\BE{S2S1loc_e2} S^2\!\times\!\R\lra\R, \quad (z,t)\lra \frac{2t}{1\!+\!|z|^2}\,, \qquad
S^1\!\times\!\C\lra\R, \quad (t,z)\lra \frac{2|z|^2}{t+t^{-1}}\,.\EE
The remaining singular fiber of~\eref{DMforgR_e} is obtained by blowing up a point
of another compact orientable 3-manifold~$\wch\cM_{0,3}^{\R}$.
The latter is isomorphic to the orientable ``double connect-sum" of two copies
of $S^1\!\times\!S^2$, i.e.~the manifold obtained by removing two disjoint 
three-balls from each copy of $S^1\!\times\!S^2$ and gluing the two copies together
along the common boundary so that the glued manifold is orientable.
The manifold $\wch\cM_{0,3}^{\R}$ can be obtained by contracting
the second copy of~$\R\P^2$ described in the proof of the $\tau$ case 
of Proposition~\ref{DMrelR_prp};
the loop
$\Ga_1^{\tau}\!-\!\Ga_{\bar2}^{\tau}\!+\!\Ga_{\bar1}^{\tau}\!-\!\Ga_2^{\tau}$
then arises from a contractible loop in the complement of the blowup point
in $\wch\cM_{0,3}^{\R}$ and thus is contractible in $\ov\cM_{0,3}^{\tau}$.
This implies that~\eref{DMrelR_e} holds over~$\Z$.
\end{rmk}

\section{Proof of Theorem~\ref{main2_thm}}
\label{Rmainpf_sec}

\noindent
The relation on $\R\ov\cM_{0,3}$ of Proposition~\ref{DMrelR_prp} induces
relations between counts of real maps from nodal domains into
a real symplectic manifold~$(X,\om,\phi)$; see Corollary~\ref{DMrelR_crl}.
Proposition~\ref{KunnethSplit_prp}, which is proved in Section~\ref{orient_sec},
expresses these counts in terms of real GW-invariants and 
a decorated version of complex GW-invariants
via the Kunneth splitting of the diagonal~$\De_X$ in~$X^2$.
Proposition~\ref{twistGW_prp}, which is also proved in Section~\ref{orient_sec},
relates the decorated invariants to the usual complex GW-invariants.
We conclude this section by deducing Theorem~\ref{main2_thm} 
from Corollary~\ref{DMrelR_crl} and 
Propositions~\ref{KunnethSplit_prp} and~\ref{twistGW_prp}.\\

\noindent
Let $(X,\om,\phi)$ be a compact real symplectic $2n$-manifold, 
$\be\!\in\!H_2(X)_{\phi}$, and $k\!\in\!\Z$ with $k\!\ge\!2$.
Let
\BE{Rforgmap_e}
f_{012}^c\!: \ov\fM_{0,k+1}(X,\be)^{\phi,c}\lra \ov\cM_{0,3}^c,~~~c=\tau,\eta,\qquad
f_{012}^{\R}\!: \ov\fM_{0,k+1}(X,\be)^{\phi}\lra \ov\cM_{0,3}^{\R},\EE
be the forgetful morphisms keeping the first three conjugate pairs of marked points only
(i.e.~those indexed by~0,1,2).
If $c\!=\!\tau,\eta$ and \tauetacond{c} and \eref{bndcond_e} 
in Section~\ref{mainthm_sec} are satisfied, we~set
$$\R f_{012}=f_{012}^c, \qquad 
\R\ov\fM_{k+1}(\be)=\ov\fM_{0,k+1}(X,\be)^{\phi,c}, 
\qquad \R\ov\cM_{0,3}=\ov\cM_{0,3}^c\,.$$
If \tauetacond{\tau} and \tauetacond{\eta} are satisfied, but not~\eref{bndcond_e},
we~set
$$\R f_{012}=f_{012}^{\R}, \qquad 
\R\ov\fM_{k+1}(\be)=\ov\fM_{0,k+1}(X,\be)^{\phi}, 
\qquad \R\ov\cM_{0,3}=\ov\cM_{0,3}^{\R}\,.$$
In all cases, we index the conjugate pairs of marked points of elements
of $\R\ov\fM_{k+1}(\be)$ by the set $\{0,1,\ldots,k\}$.
For any  relative cycle~$\Ga$ in $(\R\ov\cM_{0,3},\prt\R\ov\cM_{0,3})$
and $\mu_0,\ldots,\mu_k\!\in\!H^*(X)$, we define
\BE{RtiNdfn0_e}
\lr{\mu_0,\ldots,\mu_k}_{\be}^{\Ga}=
\int_{[\R\ov\fM_{k+1}(\be)]^{\vir}}\R f_{012}^*\PD([\Ga])\,
\ev_0^*\mu_0\,\ldots\,\ev_k^*\mu_k.\EE
This number counts degree~$\be$ real morphisms into $(X,\phi)$
from domains that stabilize to elements of~$\Ga$ after dropping
the conjugate pairs labeled by the set $\{3,\ldots,k\}$.
From Proposition~\ref{DMrelR_prp}, we immediately obtain the following corollary.

\begin{crl}[of Proposition~\ref{DMrelR_prp}]\label{DMrelR_crl}
Let $(X,\om,\phi)$ be a compact real symplectic manifold,
\hbox{$\be\!\in\!H_2(X)_{\phi}$}, and $\mu_0,\ldots,\mu_k\!\in\!H^*(X)$ 
for some $k\!\ge\!2$.
If $c\!=\!\tau,\eta$ and the conditions \tauetacond{c}
and~\eref{bndcond_e} in Section~\ref{mainthm_sec} are satisfied, then
\BE{DMrelRcrl_e}\lr{\mu_0,\ldots,\mu_k}_{\be}^{\Ga_1^c}+
\lr{\mu_0,\ldots,\mu_k}_{\be}^{\Ga_{\bar1}^c}
=\lr{\mu_0,\ldots,\mu_k}_{\be}^{\Ga_2^c}+
\lr{\mu_0,\ldots,\mu_k}_{\be}^{\Ga_{\bar2}^c}\,,\EE
where $\Ga_1^c,\Ga_{\bar1}^c,\Ga_2^c,\Ga_{\bar2}^c$ are the relative cycles in 
$(\R\ov\cM_{0,3}^c,\prt\R\ov\cM_{0,3}^c)$ defined in Section~\ref{DMrelR_sec}
and represented by the diagrams in Figures~\ref{fig_eta} and~\ref{fig_tau}.
If the conditions~\tauetacond{\tau} and \tauetacond{\eta} are satisfied,
but not necessarily~\eref{bndcond_e}, then~\eref{DMrelRcrl_e} holds with $c\!=\!\R$. 
\end{crl}

\noindent
We next express the numbers appearing in Corollary~\ref{DMrelR_crl} in terms 
of complex and real GW-invariants.
Let $(X,\om,\phi)$ be a compact real symplectic $2n$-manifold,
$\be\!\in\!H_2(X)_{\phi}$, and $k\!\in\!\Z^{\ge0}$.
We denote~by 
$$\ov\fM_{k+1}(\be)=\ov\fM_{0,k+1}(X,\be)$$ 
the moduli space of stable genus~0 degree~$\be$ maps with marked points indexed
by the~set $\{0,1,\ldots,k\}$. 
For any $I\!\subset\!\{1,\ldots,k\}$, let $\ov\fM_{k+1}^I(\be)$ be
the space $\ov\fM_{k+1}(\be)$ with the reverse orientation if~$|I|$ is odd
and~let
$$\ev^I\!:\ov\fM_{k+1}^I(\be)\lra X^{k+1}$$
be the modification of the total evaluation map
\BE{evdfn_e}\ev\equiv \ev_0\!\times\!\ev_1\!\times\!\ldots\!\times\!\ev_k\!: 
\ov\fM_{k+1}(\be)\lra X^{k+1}\EE
obtained by replacing~$\ev_i$ with $\phi\!\circ\!\ev_i$ whenever $i\!\in\!I$.
For any $\mu_0,\ldots,\mu_k\!\in\!H^*(X)$, define
$$\lr{\mu_0,\ldots,\mu_k}_{\be}^I
=\int_{[\ov\fM_{k+1}^I(\be)]^{\vir}}\ev^{I*}(\mu_0\!\times\!\ldots\!\times\!\mu_k).$$
This setup is motivated by the introduction of sign decorations for
disk maps in~\cite{Ge2}.
The next two propositions are established in Section~\ref{orient_sec}.
As before, if $\mu_1,\ldots,\mu_k\!\in\!H^*(X)$ and $I\!\subset\!\{1,\ldots,k\}$,
let $\mu_I$ denote a tuple with the entries~$\mu_i$ with $i\!\in\!I$,
in some order.

\begin{prp}\label{KunnethSplit_prp}
Let $(X,\om,\phi)$ be a compact real symplectic manifold,
\hbox{$\be\!\in\!H_2(X)_{\phi}$},
$\Ga_1^c,\Ga_{\bar1}^c,\Ga_2^c,\Ga_{\bar2}^c$ be the relative cycles in 
$(\R\ov\cM_{0,3},\prt\R\ov\cM_{0,3})$ defined in Section~\ref{DMrelR_sec}
and represented by the diagrams in Figure~\ref{RM_fig},
and $\{\ga_i\}_{i\le\ell}$ and $\{\ga^i\}_{i\le\ell}$ be dual bases
for~$H^*(X)$.
If $c\!=\!\tau,\eta$ and the conditions \tauetacond{c}
and~\eref{bndcond_e} in Section~\ref{mainthm_sec} are satisfied, then
\BE{KunnethSplit_e}\begin{split}
\lr{\mu_0,\ldots,\mu_k}_{\be}^{\Ga_j^c}
=&\lr{\mu_0\mu_j,\mu_{3-j},\mu_3,\ldots,\mu_k}_{\be}^{\phi,c}\\
&+\sum_{\begin{subarray}{c}\fd(\be_1)+\be_2=\be\\ \be_1,\be_2\in H_2(X)-\{0\} \end{subarray}}
\!\!\!
\sum_{\begin{subarray}{c} I^+\sqcup J\sqcup I^-=\{1,\ldots,k\}\\ 
j\in I^+,\,3-j\in J\end{subarray}}
\!\!\!
\sum_{\begin{subarray}{c}1\le i\le\ell\\ \ga^i\in H^{2*}(X)^{\phi}_-\end{subarray}}
\hspace{-.25in}
\blr{\mu_0,\mu_{I^+\sqcup I^-},\ga_i}_{\be_1}^{I^-}
\!\blr{\mu_J,\ga^i}_{\be_2}^{\phi,c}\end{split}\EE
for all $j\!=\!1,2$, $k\!\ge\!2$, and $\mu_0,\ldots,\mu_k\!\in\!H^{2*}(X)$. 
The same identity also holds with \hbox{$(\Ga_j,\mu_0\mu_j,j\!\in\!I^+)$} replaced by 
$(\Ga_{\bar{j}},-\mu_0\phi^*\mu_j,j\!\in\!I^-)$.  
If the conditions~\tauetacond{\tau} and \tauetacond{\eta} are satisfied,
but not necessarily~\eref{bndcond_e}, then the four identities hold 
with $\Ga_*^c\!=\!\Ga_*^{\R}$ and $\lr{\ldots}^{\phi,c}\!=\!\lr{\ldots}^{\phi}$.
\end{prp}

\begin{prp}\label{twistGW_prp}
Let $(X,\om,\phi)$ be a compact real symplectic $2n$-manifold,
$\be\!\in\!H_2(X)_{\phi}$, $k\!\in\!\Z^{\ge0}$, and $I\!\subset\!\{0,1,\ldots,k\}$.
For all $\mu_0,\ldots,\mu_k\!\in\!H^*(X)_-^{\phi}\!\cup\!H^*(X)_+^{\phi}$,
$$\lr{\mu_0,\ldots,\mu_k}_{\be}^I=(-1)^{\ve_I(\mu)}\lr{\mu_0,\ldots,\mu_k}_{\be}^X\,,$$
where $\ve_I(\mu)\!=\!|\{i\!\in\!I\!:\,\mu_i\!\in\!H^*(X)_+^{\phi}\}|$.
\end{prp}

\begin{proof}[{\bf\emph{Proof of Theorem~\ref{main2_thm}}}] 
We apply Corollary~\ref{DMrelR_crl} with $\mu_0\!\equiv\!\mu,\mu_1,\ldots,\mu_k$
as in the statement of Theorem~\ref{main2_thm}.
Since $\mu_i\!\in\!H^*(X)_-^{\phi}$ for all $i\!=\!1,\ldots,k$,
$$\blr{\mu_0,\mu_{I^+\sqcup I^-},\ga_i}_{\be_1}^{I^-}
= \blr{\mu_0,\mu_{I^+\sqcup I^-},\ga_i}_{\be_1}^X$$
for all decompositions $I^+\!\sqcup\!J\!\sqcup\!I^-=\{1,\ldots,k\}$
and for all four terms in~\eref{DMrelRcrl_e}; see Proposition~\ref{twistGW_prp}.
If $c\!=\!\tau,\eta$ and the conditions \tauetacond{c} and~\eref{bndcond_e} are satisfied,
Proposition~\ref{KunnethSplit_prp} thus reduces the left-hand side of~\eref{DMrelRcrl_e}
to 
$$2\Bigg(\!\!\lr{\mu_0\mu_1,\mu_2,\mu_3,\ldots,\mu_k}_{\be}^{\phi,c}\\
+\!\!\!
\sum_{\begin{subarray}{c}\fd(\be_1)+\be_2=\be\\ \be_1,\be_2\in H_2(X)-\{0\} \end{subarray}}
\!\!\!
\sum_{I\sqcup J=\{3,\ldots,k\}}
\!\!\!
\sum_{\begin{subarray}{c}1\le i\le\ell\\ \ga^i\in H^{2*}(X)^{\phi}_-\end{subarray}}
\hspace{-.24in} 2^{|I|}
\blr{\mu_0,\mu_1,\mu_I,\ga_i}_{\be_1}^X \!\blr{\mu_2,\mu_J,\ga^i}_{\be_2}^{\phi,c}\Bigg)$$
and the right-hand side of~\eref{DMrelRcrl_e} to 
$$2\Bigg(\!\!\lr{\mu_1,\mu_0\mu_2,\mu_3,\ldots,\mu_k}_{\be}^{\phi,c}\\
+\!\!\!
\sum_{\begin{subarray}{c}\fd(\be_1)+\be_2=\be\\ \be_1,\be_2\in H_2(X)-\{0\} \end{subarray}}
\!\!\!
\sum_{I\sqcup J=\{3,\ldots,k\}}
\!\!\!
\sum_{\begin{subarray}{c}1\le i\le\ell\\ \ga^i\in H^{2*}(X)^{\phi}_-\end{subarray}}
\hspace{-.24in} 2^{|I|}
\blr{\mu_0,\mu_2,\mu_I,\ga_i}_{\be_1}^X \!\blr{\mu_1,\mu_J,\ga^i}_{\be_2}^{\phi,c}\Bigg)\,.$$
Setting the two expressions equal, we obtain the formula in Theorem~\ref{main2_thm}.
If the conditions~\tauetacond{\tau} and \tauetacond{\eta} are satisfied,
but not necessarily~\eref{bndcond_e}, the same argument applies with 
$\lr{\ldots}^{\phi,c}$ replaced by~$\lr{\ldots}^{\phi}$.
\end{proof}

\section{Orientations and signs}
\label{orient_sec}

\noindent
In this section, we analyze and compare orientations of various moduli spaces
of complex and real maps.
We use these comparisons to establish  
Proposition~\ref{twistGW_prp}, Theorem~\ref{evenvan_thm},
and Proposition~\ref{KunnethSplit_prp}.

\begin{proof}[{\bf\emph{Proof of Proposition~\ref{twistGW_prp}}}]
For each cycle $h:Y\!\lra\!X$ representing the Poincare dual of an element
of~$H^*(X)_{\pm}^{\phi}$, let $\ve(h)\!=\!\pm1$, respectively. 
Define an involution $\Th^I\!:X^{k+1}\!\lra\!X^{k+1}$ by  
$$\Th^I\big(x_0,\ldots,x_k\big)\lra \big(\Th^I_0(x_0),\ldots,\Th^I_k(x_k)\big),
\qquad\hbox{where}\qquad
\Th^I_i(x)=\begin{cases}x,&\hbox{if}~i\!\not\in\!I;\\
\phi(x),&\hbox{if}~i\!\in\!I.
\end{cases}$$
We can assume that the cohomology degrees of $\mu_0,\mu_1,\ldots,\mu_k$ satisfy
$$\deg\mu_0+\ldots+\deg\mu_k =\dim^{\vir}\ov\fM_{k+1}(\be)
=2\big(\blr{c_1(X),\be}+n\!-\!2+k\big)\,, $$
where $2n\!=\!\dim X$.
Choose a generic collection of representatives \hbox{$h_i\!:Y_i\!\lra\!X$} 
for the Poincare duals of $\mu_0,\ldots,\mu_k$, respectively.
The Poincare dual of~$\phi^*\mu_i$ is then represented by the~cycle
\BE{conjmu_e0}\ov{h_i}\equiv \phi\!\circ\!h_i\!: \bar{Y}_i\equiv (-1)^nY_i\lra X,\EE
with $-Y_i$ denoting $Y_i$ with the opposite orientation.
Let
$$\lr\bh=h_0\!\times\!\ldots\!\times\!h_k\!: 
\Y\equiv Y_0\!\times\!\ldots\!\times\!Y_k\lra X^{k+1}.$$
We denote by $\Y^I$ the modification of~$\Y$
with the $i$-th factor replaced by $\ve(h_i)\bar{Y}_i$ and~by
$$\lr\bh^I\!:\Y^I\lra X^{k+1}$$
the modification of~$\lr\bh$ with the $i$-th factor map replaced by~$\ov{h_i}$
whenever $i\!\in\!I$. Thus, 
$\lr\bh^I\!=\!\Th^I\!\circ\!\lr\bh$.\\

\noindent
We set 
$$\ov\fM_{\bh}^I(\be)=
\big\{(u,\y)\!\in\!\ov\fM_{k+1}^I(\be)\!\times\!\Y^I\!:\,\ev^I(u)\!=\!\lr\bh^I(\y)\big\}\,,
\qquad
\ov\fM_{\bh}(\be)\equiv\ov\fM_{\bh}^{\eset}(\be)\,.$$
As sets, these two objects are the same.
For a generic tuple~$\bh$, the restriction of the total evaluation map~\eref{evdfn_e} 
to every stratum of $\ov\fM_{k+1}(\be)$ is transverse to $\lr\bh$ in~$X^{k+1}$
and thus $\ov\fM_{\bh}(\be)$ is a finite collection
of signed weighted points contained in the main stratum of the moduli space.
Since $\bh$ and $\bh^I$ represent the Poincare duals of 
$\mu_0\!\times\!\ldots\!\times\!\mu_k$, the signed weighted cardinalities of 
$\ov\fM_{\bh}(\be)$ and $\ov\fM_{\bh}^I(\be)$ are the numbers
$\lr{\mu_0,\ldots,\mu_k}_{\be}^X$ and $\lr{\mu_0,\ldots,\mu_k}_{\be}^I$, respectively.\\

\noindent
The sign of each element~$(u,\y)$ of $\ov\fM_{\bh}^I(\be)$ is determined by
the orientations of $\ov\fM_{k+1}^I(\be)$, $\Y^I$, and~$X^{k+1}$
via the maps~$\ev^I$ and~$\lr\bh^I$.
It is the sign of the isomorphism
\BE{normbndl_e0} \tnd\big\{\ev^I\!\times\!\lr\bh^I\big\}\!: 
T(\ov\fM_{k+1}^I(\be)\!\times\!\Y^I)|_{(u,\y)}
\lra 
\frac{T(X^{k+1}\!\times\!X^{k+1})|_{\De_{X^{k+1}}}}{T(\De_{X^{k+1}})}
\bigg|_{(\ev^I(u),\lr\bh^I(\y))} \,,\EE
where $\De_{X^{k+1}}\!\subset\!X^{k+1}\!\times\!X^{k+1}$ is the diagonal.
By the chain rule,
$$\tnd\big\{\ev^I\!\times\!\lr\bh^I\big\}  =  
\tnd\big\{\Th^I\!\times\!\Th^I\big\}\circ
\tnd\big\{\ev\!\times\!\lr\bh\big\}\,.$$
The sign of the isomorphism
$$\tnd\big\{\Th^I\!\times\!\Th^I\big\}\!:
\frac{T(X^{k+1}\!\times\!X^{k+1})|_{\De_{X^{k+1}}}}{T(\De_{X^{k+1}})}
\bigg|_{(\ev(u),\lr\bh(\y))}
\lra
\frac{T(X^{k+1}\!\times\!X^{k+1})|_{\De_{X^{k+1}}}}{T(\De_{X^{k+1}})}
\bigg|_{(\ev^I(u),\lr\bh^I(\y))}$$
is $(-1)^{n|I|}$.
The orientations of $\ov\fM_{k+1}^I(\be)$ and $\ov\fM_{k+1}(\be)$
differ by~$(-1)^{|I|}$,
while the orientations of~$\Y^I$ and~$\Y$ differ by
$$(-1)^{n|I|+|\{i\in I:\,\mu_i\in H^*(X)^{\phi}_-\}|}\,.$$
Thus, the signed weighted cardinalities of 
$\ov\fM_{\bh}^I(\be)$ and $\ov\fM_{\bh}(\be)$ differ by the sign $(-1)^{\ve_I(\mu)}$.
\end{proof}

\noindent
We next recall how the main stratum $\fM_{k+1}^{\phi,c}(\be)$ of the moduli space
$\ov\fM_{k+1}(X,\be)^{\phi,c}$ with $c\!=\!\tau,\eta$ is oriented 
if the condition~\tauetacond{c} in Section~\ref{mainthm_sec} is satisfied.
We begin with the case $k\!=\!-1$.
By definition,
$$\fM_0^{\phi,c}(\be)= \cP_0^{\phi,c}(\be)/G_c\,, \qquad g\cdot u=u\circ g,$$
where $\cP_0^{\phi,c}(\be)$ is the space of parametrized $(\phi,c)$-real 
degree~$\be$ $J$-holomorphic maps $\P^1\!\lra\!X$
 and $G_c\!\subset\!\PSL_2\C$ is the subgroup of automorphisms of~$\P^1$
commuting with~$c$.
The latter is oriented by the short exact sequence 
$$0\lra T_{\id}S^1\lra T_{\id}G_c\lra\C\lra 0,$$
where $\C\!=\!T_0\C$ corresponds to shifting the origin and
$S^1\!\subset\!G_c$ is the subgroup of standard rotations of~$\C$,
which we identify with $S^1\!\subset\!\C^*$.
The (virtual) tangent space of~$\cP_0^{\phi,c}(\be)$ at a point $u\!\in\!\cP_0^{\phi,c}(\be)$  
is the index of the linearization~$D_u^c$ of the $\bar\prt$-operator at~$u$.
If $c\!=\!\tau$, we orient this index as
in the proofs of \cite[Corollary~1.8]{Ge} and \cite[Lemma~7.3]{Ge2} from 
a fixed spin structure  on~$TX^{\phi}\!\oplus\!2E^{\ti\phi}$,
with $E$ as in~\tauetacond{\tau}.
If $c\!=\!\eta$, we orient the index via the pinching construction of 
\cite[Lemma~2.5]{Teh} from a fixed spin sub-structure on~$(TX,\tnd\phi)$;
see \cite[Corollary~5.10]{GZ2}.
The orientation of~$\fM_0^{\phi,c}(\be)$ at~$[u]$ is then specified~by
$$ \ind\, D_u^c \approx T_{[u]}\fM_0^{\phi,c}(\be) \oplus T_{\id}G_c\,.$$
The order of the factors on the right-hand side above is motivated
by the choice of the orientation on~$\ov\cM_{0,2}^c$ in Section~\ref{DMrelR_sec}.
For $k\!\ge\!0$, $\ov\fM_{k+1}(X,\be)^{\phi,c}$ is oriented using the first element in 
each conjugate pair~$(z_i,z_{\bar i})$ to orient the general
fibers of the forgetful morphism 
\BE{fMforg_e}\ov\fM_{k+1}(X,\be)^{\phi,c}\lra \ov\fM_0(X,\be)^{\phi,c}\EE
obtained by forgetting the $k$ pairs of conjugate marked points.\\

\noindent
In this paper, we use a different natural construction of orientation on
$\ov\fM_{k+1}(X,\be)^{\phi,c}$ in the stable range, i.e.~$k\!\ge\!1$;
in Lemma~\ref{MapsOrient_lmm}, we show that the two orientations coincide.
It is obtained using the forgetful morphism
$$f\!:\ov\fM_{k+1}(X,\be)^{\phi,c}\lra \ov\cM_{0,k+1}^c$$
and the orientation on $\ov\cM_{0,k+1}^c$ defined in Section~\ref{DMrelR_sec}.
For a general $[u]\!\in\!\ov\fM_{k+1}(X,\be)^{\phi,c}$ in this case,
the domain~$\Si_u$ of~$u$ with its marked points is stable and 
thus $\cC\!=\![\Si_u]$ is the image of~$[u]$ in~$\ov\cM_{0,k+1}^c$.
The (virtual) vertical tangent bundle of~$f$ at such~$[u]$ is
the index of~$D_u^c$.
The orientation of $\ov\fM_{k+1}(X,\be)^{\phi,c}$ is then specified~by
\BE{MapsOrient2_e} T_{[u]}\ov\fM_{k+1}(X,\be)^{\phi,c}\approx
\ind\,D_u^c\oplus T_{[\Si_u]}\ov\cM_{0,k+1}^c\,,\EE
with $\ind\,D_u^c$ oriented as in the previous paragraph.

\begin{lmm}\label{MapsOrient_lmm}
Let $c\!=\!\tau,\eta$, $(X,\om,\phi)$ be a compact real symplectic manifold
satisfying  the condition~\tauetacond{c} in Section~\ref{mainthm_sec},
$k\!\in\!\Z^{\ge0}$, and $\be\!\in\!H_2(X)_{\phi}$.
\begin{enumerate}[label=(\arabic*),leftmargin=*]

\item For every $i\!=\!0,1,\ldots,k$, the automorphism of $\ov\fM_{0,k+1}(\be)^{\phi,c}$
interchanging the marked points in the $i$-th conjugate pair is orientation-reversing.

\item For all $i,j\!=\!0,1,\ldots,k$, the automorphism of $\ov\fM_{0,k+1}(\be)^{\phi,c}$
interchanging the $i$-th and $j$-th conjugate pairs of marked points
is orientation-preserving.

\item If $k\!\ge\!1$, the two orientations on $\ov\fM_{0,k+1}(\be)^{\phi,c}$
described above are the same.

\end{enumerate}
If the conditions~\tauetacond{\tau} and~\tauetacond{\eta} are satisfied, 
the three statements also apply with $\ov\fM_{0,k+1}(\be)^{\phi,c}$ replaced
by $\ov\fM_{0,k+1}(\be)^{\phi}$.
\end{lmm}

\begin{proof}
(1,2) Both automorphisms take a fiber of~\eref{fMforg_e} to the same fiber.
The restriction of the automorphism in~(1) to a fiber of~\eref{fMforg_e} is orientation-reversing,
while the restriction of the automorphism in~(2) to a fiber 
of~\eref{fMforg_e} is orientation-preserving.
This implies the first two statements of the lemma.\\

\noindent
(3) Let $u$ be an element of $\ov\fM_{0,k+1}(\be)^{\phi,c}$ at a point~$u$ 
with smooth domain~$\Si_u$ and $u_0$ be its image of $[u]$ under~\eref{fMforg_e}.
The first orientation of $\ov\fM_{0,k+1}(\be)^{\phi,c}$ described above satisfies
\begin{equation*}\begin{split}
T_{[u]}\ov\fM_{0,k+1}(\be)^{\phi,c}\oplus T_{\id}G_c
&\approx T_{[u_0]}\ov\fM_{0,k+1}(\be)^{\phi,c}
\oplus \bigoplus_{i=0}^k T_{z_i^+}\Si_u \oplus T_{\id}G_c\\
&\approx T_{[u_0]}\ov\fM_{0,k+1}(\be)^{\phi,c}
\oplus T_{\id}G_c \oplus \bigoplus_{i=0}^k T_{z_i^+}\Si_u
\approx \ind\,D_{u_0}^c\oplus  \bigoplus_{i=0}^k T_{z_i^+}\Si_u\,.
\end{split}\end{equation*}
The orientation of $\ov\cM_{0,k+1}^c$ chosen in Section~\ref{DMrelR_sec}
at a smooth curve $\cC\!=\![(z_0^+,z_0^-),\ldots,(z_k^+,z_k^-)]$ is described~by
$$T_{z_0^+}\cC\oplus\ldots\oplus T_{z_k^+}\cC\approx T_{\cC}\ov\cM_{0,k+1}^c 
\oplus T_{\id}G_c\,.$$
Thus, the second orientation of $\ov\fM_{0,k+1}(\be)^{\phi,c}$ described above satisfies
\begin{equation*}\begin{split}
T_{[u]}\ov\fM_{0,k+1}(\be)^{\phi,c}\oplus T_{\id}G_c
&\approx \ind\,D_{u_0}^c\oplus T_{[\Si_0]}\ov\cM_{0,k+1}^c\oplus T_{\id}G_c
\approx \ind\,D_{u_0}^c\oplus  \bigoplus_{i=0}^k T_{z_i^+}\Si_u\,.
\end{split}\end{equation*}
Thus, the two orientations of $T_{[u]}\ov\fM_{0,k+1}(\be)^{\phi,c}$ are the same.
\end{proof}

\noindent
If $c$ is an orientation-reversing involution on a compact orientable surface~$\Si$
of genus~$g$ and $(X,\om,\phi)$ is a compact real symplectic $2n$-manifold such that
$n$ is odd,  $X^{\phi}\!=\!\eset$, and 
$\La_{\C}^{\top}(TX,\tnd\phi)$ admits a real square root,
then the moduli spaces $\ov\fM_{g,k+1}(X,\be)^{\phi,c}$ are oriented 
via the analogue of the morphism~\eref{fMforg_e}.
Thus, the first two statements of Lemma~\ref{MapsOrient_lmm} also hold
if $\ov\fM_{0,k+1}(X,\be)^{\phi,c}$ is replaced by $\ov\fM_{g,k+1}(X,\be)^{\phi,c}$.

\begin{proof}[{\bf\emph{Proof of Theorem~\ref{evenvan_thm}}}]
We denote by $\R\ov\fM_{g,k}(\be)$ the appropriate moduli space of real morphisms,
as determined by the case of Theorem~\ref{evenvan_thm} under consideration.
We can assume that the cohomology degrees of $\mu_1,\ldots,\mu_k$ satisfy
$$\deg\mu_1+\ldots+\deg\mu_k =\dim^{\vir}\R\ov\fM_{g,k}(\be)
=\blr{c_1(X),\be}+(n\!-\!3)(1\!-\!g)+2k\,. $$
Choose $h_i\!:Y_i\!\lra\!X$ as in the proof of Proposition~\ref{twistGW_prp}
and define $\lr\bh$, $\lr\bh^I$, $\R\ov\fM_{\bh}(\be)$, and $\R\ov\fM_{\bh}^I(\be)$, 
for any subset $I\!\subset\!\{1,\ldots,k\}$, as before, but starting 
with the moduli space $\R\ov\fM_{g,k}(\be)$ in the last two cases. 
By exactly the same argument as in the proof of Proposition~\ref{twistGW_prp},
the signed weighted cardinalities of  $\R\ov\fM_{\bh}^I(\be)$ and $\R\ov\fM_{\bh}(\be)$ 
differ by the sign $(-1)^{\ve_I(\mu)}$.\\

\noindent 
If $\mu_{i^*}\!\in\!H^*(X)_+^{\phi}$, we apply the above conclusion with $I\!=\!\{i^*\}$.
The signed weighted cardinalities of $\R\ov\fM_{\bh}^I(\be)$ and $\R\ov\fM_{\bh}(\be)$ 
are then opposite.
Interchanging the points in the $i^*$-th conjugate pair induces
an orientation-preserving isomorphism from $\R\ov\fM_{\bh}^I(\be)$
to $\R\ov\fM_{\bh'}(\be)$, where $\bh'$ is the tuple obtained by 
replacing~$h_{i^*}$ with
$$\ov{h_{i^*}}\!=\!\phi\!\circ\!h_{i^*}:(-1)^n Y_{i^*}\lra X\,;$$ 
this cycle represents the Poincare dual of $\phi^*\mu_{i^*}\!=\!\mu_{i^*}$.
Thus, the signed weighted cardinalities of $\R\ov\fM_{\bh}(\be)$ and
$\R\ov\fM_{\bh'}(\be)$ are opposite.
Since both of them are equal to the real invariant 
$\lr{\mu_1,\ldots,\mu_k}_{\be}^{\phi}$ in question, 
the latter vanishes.
\end{proof}

\noindent
In the remainder of this section, we establish Proposition~\ref{KunnethSplit_prp}.
The key point in its proof is that all orientations are chosen compatibly;
in particular, the oriented normal bundle of~$\Ga_*$ in $\R\ov\cM_{0,3}$ 
and the oriented normal bundle 
of its preimage in $\R\ov\fM_{k+1}(\be)$ are given by the complex line bundle
of smoothings of the node on the bubble containing~$z_0$.
We proceed with the notation and assumptions as in the statement of 
Proposition~\ref{KunnethSplit_prp}.
We will also use the same notation for the uncompactified moduli spaces
(maps only from smooth domains) as we have introduced for the compactified moduli
spaces.\\

\noindent
For $\be\!\in\!H_2(X)_{\phi}$, denote~by $\cN_{\be}\!\subset\!\R\ov\fM_{k+1}(\be)$
the sub-orbifold of maps from domains consisting of precisely three components 
with one invariant bubble 
and two conjugate bubbles with the marked point~$z_0^+$ on one of the conjugate bubbles. 
For $u\!\in\!\cN_{\be}$, denote by~$u^{\C}$ the restriction of~$u$ to the component
containing~$z_0^+$ and by $z^{\C}$ the marked point corresponding to the node on this component; 
denote by $u^{\R}$ the restriction of~$u$ to the invariant component and by $z_{\bu}^+$ 
the marked point on this component corresponding to the same node as~$z^{\C}$.
If $\be\!=\!\fd(\be_1)\!+\!\be_2$ and $\{1,\ldots,k\}\!=\!I^+\!\sqcup\!J\!\sqcup\!I^-$,
let
$$\cN_{\be_1,\be_2;I^+,J,I^-} \subset \cN_{\be}$$
be the subspace of the maps~$u$ so that 
the degrees of~$u^{\C}$ and~$u^{\R}$ are~$\be_1$ and~$\be_2$, respectively,
and the rest of the marked points carried by the component containing~$z_0^+$
are the first elements in the pairs of conjugate points indexed by~$I^+$ 
and the second elements in the pairs indexed by~$I^-$.
If 
\BE{empcond_e}(\be_1,I^+,I^-)\!=\!(0,\eset,\eset)
\qquad\hbox{or}\qquad (\be_2,J)\!=\!(0,\eset),\EE
$\cN_{\be_1,\be_2;I^+,J,I^-}\!=\!\eset$ for stability reasons.\\

\noindent
The restrictions~$u^{\C}$ and~$u^{\R}$ determine an isomorphism
\BE{cNisom_e}\cN_{\be_1,\be_2;I^+,J,I^-}\approx
\big\{(u^{\C}, u^{\R})\in
\fM_{|I^+|+|I^-|+2}^{I^-}(\be_1)\!\times\!\R\fM_{|J|+1}(\be_2)\!:
u^{\C}(z^{\C})\!=\!u^{\R}(z_{\bu}^+)\big\},\EE
with the marked points of the elements of $\fM_{|I^+|+|I^-|+2}^{I^-}(\be_1)$ 
indexed by~$0$, the elements of~$I^+\!\sqcup\!I^-$, and the superscript~$\C$;
under either of the conditions~\eref{empcond_e}, one of the moduli spaces
on the right-hand side of~\eref{cNisom_e} is empty for stability reasons.
The inverse map is obtained by identifying the marked point~$z^{\C}$ of the domain of~$u^{\C}$ 
with the marked point~$z_{\bu}^+$ of the domain of~$u^{\R}$ and the marked point~$c(z^{\C})$ 
of the map $\phi\!\circ\!u^{\C}\!\circ\!c$ with $z_{\bu}^-\!=\!c(z_{\bu}^+)$; 
the marked points of~$u^{\C}$ indexed by~$I^+$ become the first points 
in the corresponding pair of the nodal map, while
those indexed by~$I^-$ become the second.
As in Section~\ref{Rmainpf_sec}, $\fM_{|I^+|+|I^-|+2}^{I^-}(\be_1)$ is oriented by twisting 
the canonical complex orientation of $\fM_{|I^+|+|I^-|+2}(\be_1)$ by $(-1)^{|I^-|}$.
The canonical orientation of~$X$ and the chosen orientations of 
$\fM_{|I^+|+|I^-|+2}^{I^-}(\be_1)$ and $\R\fM_{|J|+1}(\be_2)$
induce an orientation on each component of~$\cN_{\be}$ via the isomorphism~\eref{empcond_e}.\\

\noindent
Let $L^{\C}\!\lra\!\fM_{|I^+|+|I^-|+2}^{I^-}(\be_1)$ and 
$L^{\R}\!\lra\!\R\fM_{|J|+1}(\be_2)$ be the universal tangent line bundles 
at the marked points~$z^{\C}$ and~$z_{\bu}^+$, respectively, and
$$L=\pi_1^*L^{\C}\otimes_{\C} \pi_2^*L^{\R}\lra \cN_{\be}\,,$$
where $\pi_1,\pi_2$ are the component projection maps.
The   line bundle $L\!\lra\!\cN_{\be}$ is the normal bundle 
of $\cN_{\be}$ in~$\R\ov\fM_{k+1}(\be)$.
There is a gluing~map
\BE{RPhidfn_e}\Phi\!:U\lra \R\ov\fM_{k+1}(\be),\EE
where $U\!\subset\!L$ is a neighborhood of the zero set in~$L$;
it is obtained via a $(\phi,c)$-equivariant version of a standard gluing
construction, such as in \cite[Section~3]{LT}.\\

\noindent
If $k\!\ge\!2$ and $|J\!\cap\!\{1,2\}|\!=\!1$,
$\cN_{\be_1,\be_2;I^+,J,I^-}$ is a topological component of the pre-image of~$\mr\Ga$
under the forgetful morphism~$\R f_{012}$ in~\eref{Rforgmap_e} for some 
$\Ga\!=\!\Ga_i,\Ga_{\bar{i}}$, with $i\!=\!1,2$.
In this case, the restriction of~$L$ to $\cN_{\be_1,\be_2;I^+,J,I^-}$
equals $\R f_{012}^*L_{\Ga}$, where $L_{\Ga}\!\lra\!\Ga$ is the complex line bundle
defined in Section~\ref{DMrelR_sec}.
The gluing map~$\Phi$ in~\eref{RPhidfn_e} can be chosen so that its restriction 
to each such component $\cN_{\be_1,\be_2;I^+,J,I^-}$ lifts 
any pre-specified gluing map on~$L_{\Ga;\bh}$ over~$\R f_{012}$.

\begin{lmm}\label{Rgluing0_lmm}
If $k\!\ge\!2$ and $|J\!\cap\!\{1,2\}|\!=\!1$, 
the restriction of the gluing map~\eref{RPhidfn_e} 
to a neighborhood of $\cN_{\be_1,\be_2;I^+,J,I^-}$ in~$L$
is orientation-preserving with respect to 
the complex orientation on~$L$ and the orientation on the base described above.
\end{lmm}

\begin{proof}
This follows readily from the definitions of the three orientations above;
we follow the second construction, which is described just before Lemma~\ref{MapsOrient_lmm}.
Let $\Ga$ be as in the preceding paragraph.
If $k\!=\!2$, $\Phi$ can be chosen so that there is a commutative diagram
$$\xymatrix{ U \ar[rr]^{\Phi} \ar[d]_{f_{012}} && \R\ov\fM_3(\be) \ar[d]^{f_{012}}\\
U_{\Ga} \ar[rr]^{\Phi_\Ga}&& \R\ov\cM_{0,3} }$$
with the bottom arrow being some gluing map on a neighborhood of $\mr\Ga$ in $L_{\Ga}$.
By Lemma~\ref{DMgluing_lmm}, $\Phi_{\Ga}$ is orientation-preserving.
Since all domains are stable in this case, the vertical tangent spaces of 
the vertical arrows in the diagram are oriented by orienting the indices of 
the linearized $\dbar$-operators; see \cite[Section~6]{GZ1}. \\

\noindent
The index for the complex moduli space has a canonical orientation; see \cite[p51]{MS}.
The indices for the two real moduli spaces are oriented from either 
the same trivialization of $TX^{\phi}\!\oplus\!2E^{\ti\phi}$ over a loop in~$X^{\phi}$
or from the same trivialization of~$(TX,\tnd\phi)$ over a $\Z_2$-invariant loop in~$X$
by pinching off the relevant vector bundle  onto a conjugate pair of sphere bubbles, as in 
the proofs of \cite[Lemma~7.3]{Ge2} and in \cite[Theorem~1.1]{GZ1};
the index over the first of these bubbles,~$B$, has a canonical complex orientation.
Thus, the index of an element of $\cN_{\be_1,\be_2;I^+,J,I^-}$ is oriented by introducing 
an extra pinching in~$B$ as compared to what is
used to orient nearby real maps from~$\P^1$.
This pinching, which is given by the inverse of~$\Phi$, induces the same canonical orientation 
over~$B$.
Thus, the orientation of the index for a map from~$\P^1$ is equivalent 
to the orientation obtained from the orientation of an element of~$\cN_{\be}$
by smoothing the node.\\

\noindent
If $k\!\ge\!3$, $\Phi$ can be chosen so that there is a commutative diagram
\BE{gluingforg_e2}\begin{split}
\xymatrix{U \ar[rr]^{\Phi} \ar[d] && \R\ov\fM_{k+1}(\be) \ar[d]\\
U' \ar[rr]^{\Phi'} && \R\ov\fM_3(\be) }\end{split}\EE
with the vertical arrows being forgetful maps again.
Since the fibers of the right arrow are oriented by the first points in each conjugate pair
and the orientation of the fibers of the left arrow is based on the number of 
conjugate pairs with the second point carried by~$u^{\C}$,
Lemma~\ref{DMorient_lmm} implies that 
$\Phi$ is again orientation-preserving between the fibers and
thus between the spaces on the first line of~\eref{gluingforg_e2}.
\end{proof}

\begin{rmk}\label{gluing_rmk}
The assumption that $k\!\ge\!2$ and $|J\!\cap\!\{1,2\}|\!=\!1$ in Lemma~\ref{Rgluing0_lmm}
is not necessary, but it simplifies the argument.
The general case is not needed for our purposes.
\end{rmk}

\begin{proof}[{\bf\emph{Proof of Proposition~\ref{KunnethSplit_prp}}}]
We can assume that the cohomology degrees of $\mu_0,\ldots,\mu_k$ satisfy
\BE{mucond_e}\deg\mu_0+\ldots+\deg\mu_k=\dim^{\vir}\R\ov\fM_{k+1}(\be)-2
=\blr{c_1(X),\be}+n\!-\!3+2k\,,\EE
where $2n\!=\!\dim X$.
Choose a generic collection of representatives \hbox{$h_i\!:Y_i\!\lra\!X$} 
for the Poincare duals of $\mu_0,\ldots,\mu_k$, respectively, and define
$$\R\ov\fM_{\bh}(\be)=
\big\{(u,\y)\!\in\!\R\ov\fM_{k+1}(\be)\!\times\!Y_0\!\times\!\ldots\!\times\!Y_k\!:
\ev_i(u)\!=\!h_i(y_i)~\forall\,i\!=\!0,\ldots,k\big\}.$$
If $\be\!=\!\fd(\be_1)\!+\!\be_2$ and $\{1,\ldots,k\}\!=\!I^+\!\sqcup\!J\!\sqcup\!I^-$,
let
\BE{cNhdfn_e}\cN_{\be_1,\be_2;I^+,J,I^-}(\bh)
=\cN_{\be_1,\be_2;I^+,J,I^-}\cap \R\ov\fM_{\bh}(\be).\EE
If the representatives $h_i$ for $\mu_i$ are generic, 
each set $\cN_{\be_1,\be_2;I^+,J,I^-}(\bh)$ is a compact zero-dimensional 
suborbifold of the oriented orbifold $\cN_{\be_1,\be_2;I^+,J,I^-}$
and thus has a well-defined signed weighted cardinality. 
The latter is computed by the usual Kunneth decomposition,
with respect to the specified orientations of 
$\fM_{|I^+|+|I^-|+2}^{I^-}(\be_1)$ and $\R\fM_{|J|+1}(\be_2)$;
this gives the last sum in~\eref{KunnethSplit_e} if $\be_1,\be_2\!\neq\!\eset$, but without 
the restriction  $\ga^i\!\in\!H^{2*}(X)^{\phi}_-$.
Since $\mu_i\!\in\!H^{2*}(X)$ for all~$i$, the complex GW-invariant in~\eref{KunnethSplit_e} 
with $\ga_i\!\in\!H^{2*-1}(X)$ vanishes for dimensional reasons;
the real GW-invariant in~\eref{KunnethSplit_e} with $\ga^i\!\in\!H^*(X)_+^{\phi}$ 
vanishes by Theorem~\ref{evenvan_thm}.
If $\be_1\!=\!0$ and $|I^+\!\sqcup\!I^-|\!\ge\!2$ or 
$\be_2\!=\!0$ and $|J|\!\ge\!1$, $\cN_{\be_1,\be_2;I^+,J,I^-}(\bh)\!=\!\eset$;
otherwise, the marked points on~$u^{\C}$ (in the first case) 
or on~$u^{\R}$ (in the second case) could vary while staying inside of 
the zero-dimensional $\cN_{\be_1,\be_2;I^+,J,I^-}(\bh)$.
The case $\be_1\!=\!0$ and $|I^+\!\sqcup\!I^-|\!=\!1$ 
reduces as usual to an invariant like the first term 
on the right-hand side of~\eref{KunnethSplit_e};
as described below, there is only one decomposition 
$\{1,\ldots,k\}\!=\!I^+\!\sqcup\!J\!\sqcup\!I^-$ with
$|I^+\!\sqcup\!I^-|\!=\!1$ relevant to
each of the four cases of Proposition~\ref{KunnethSplit_prp}.\\

\noindent
For any $\Ga\!\subset\!\R\ov\cM_{0,3}$, define 
\begin{equation*}\begin{split}
Z_{\Ga}&=\big\{\big(u,y_0,\ldots,y_k)\!\in\!
\R f_{012}^{-1}(\Ga)\!\times\!Y_0\!\times\!\ldots\!\times\!Y_k\!:
\ev_i(u)\!=\!h_i(y_i)~\forall\,i\!=\!0,\ldots,k\big\}\\
&= \big(\R f_{012}^{-1}(\Ga)\!\times\!Y_0\!\times\!\ldots\!\times\!Y_k\big)
\cap \R\ov\fM_{\bh}(\be) \,.
\end{split}\end{equation*}
For $j\!=\!1,2$ and generically chosen constraints~$h_i$,  
\BE{ZGajsplit_e}
Z_{\Ga_j}= \bigsqcup_{\begin{subarray}{c}\fd(\be_1)+\be_2=\be\\ 
                      \be_1,\be_2\in H_2(X),\,\be_2\neq0 \end{subarray}}\!
\bigsqcup_{\begin{subarray}{c}I^+\sqcup J\sqcup I^-=\{1,\ldots,k\}\\
  j\in I^+,\,3-j\in J\end{subarray}}
\hspace{-.4in}\cN_{\be_1,\be_2;I^+,J,I^-}(\bh);\EE
this decomposition corresponds to the first two sums in~\eref{KunnethSplit_e}
and the first term on the right-hand side of~\eref{KunnethSplit_e}.
It also holds with $(\Ga_j,j\!\in\!I^+)$ replaced by $(\Ga_{\bar{j}},j\!\in\!I^-)$.\\

\noindent
Let $L_{\Ga;\bh}\!\lra\!Z_{\Ga}$ denote the restriction~of
\BE{pbLdfn_e} \pi_1^*L= L\!\times\!Y_0\!\times\!\ldots\!\times\!Y_k \lra 
\cN_{\be}\!\times\!Y_0\!\times\!\ldots\!\times\!Y_k\,.\EE
As in complex GW-theory, a small modification of the gluing map~\eref{RPhidfn_e} gives rise 
to a gluing~map 
$$\Phi_{\Ga;\bh}\!: U_{\Ga;\bh}\lra   \R\ov\fM_{\bh}(\be),$$
where $U_{\Ga;\bh}\!\subset\!L_{\Ga;\bh}$ is a neighborhood of the zero section in~$L_{\Ga;\bh}$
(a finite collection of disks in this case).
Such a modification can be chosen to be of the form
$$\Phi_{\Ga;\bh}(u,\ups)=\Phi\big(\psi(u,\ups),\ups\big)
\qquad\forall~(u,\ups)\in U_{\Ga;\bh},$$
for some smooth function~$\psi$ on $U_{\Ga;\bh}$ sending $(u,0)$ to~$u$.
Thus, the induced map
$$\tnd(\R f_{012}\!\circ\!\Phi_{\Ga;\bh})\!: \pi_1^*\,\R f_{012}^*L_{\Ga}\lra L_{\Ga} $$
between the normal bundle of~$Z_{\Ga}$ in~$\R\ov\fM_{\bh}(\be)$
and of~$\Ga$ in~$\R\ov\cM_{0,3}$ is the identity.
Since $\Phi_{\Ga;\bh}$ is orientation-preserving by Lemma~\ref{Rgluing0_lmm},
it follows that  
every signed weighted element of~$Z_{\Ga}$ contributes~$+1$ to the number~\eref{RtiNdfn0_e}.
By the last two paragraphs, the signed weighted cardinality of~$Z_{\Ga}$ is given~by
the right-hand side of~\eref{KunnethSplit_e}.
\end{proof}

\section{Alternative proof of Theorem~\ref{main2_thm}}
\label{mainpf_sec}

\noindent
In this section, we give a proof of Theorem~\ref{main2_thm} 
(in effect of a combination of Corollary~\ref{DMrelR_crl} and Proposition~\ref{KunnethSplit_prp})
which bypasses 
the real Deligne-Mumford moduli space $\R\ov\cM_{0,3}$ of Section~\ref{DMrelR_sec}.
We instead pull back the standard relation on~$\ov\cM_{0,4}$ by 
the forgetful morphism
\BE{forgmap0_e}\begin{split}
f_{012\bar{0}}\!: \R\ov\fM_{k+1}(\be)&\lra\ov\cM_{0,4}\,,\\
\big[u,(z_0^+,z_0^-),\ldots,(z_k^+,z_k^-)\big]&\lra[z_0^+,z_1^+,z_2^+,z_0^-],
\end{split}\EE
preserving the marked points $z_0^+,z_1^+,z_2^+,z_0^-$ only (and stabilizing the domain
if necessary).\\

\noindent 
As in Section~\ref{Rmainpf_sec},
we either fix $c\!=\!\tau,\eta$ and assume that the conditions~\tauetacond{c}
and~\eref{bndcond_e} in Section~\ref{mainthm_sec} are satisfied
or assume that 
the conditions~\tauetacond{\tau} and \tauetacond{\eta}, but not necessarily~\eref{bndcond_e},
are satisfied.
In~both cases, we
continue with the abbreviations for moduli spaces of maps introduced 
in Section~\ref{Rmainpf_sec}
(before Corollary~\ref{DMrelR_crl} for the $\R$-spaces and 
before Proposition~\ref{KunnethSplit_prp} for the $\C$-spaces).
We can again assume that~\eref{mucond_e} holds and choose generic
representative $h_i\!:Y_i\!\lra\!X$ for the Poincare duals of 
$\mu_0\!\equiv\!\mu,\mu_1,\ldots,\mu_k$.\\

\noindent
Let $\Om_{0,4}\!\in\!H^2(\ov\cM_{0,4})$ be the Poincare dual of the point class
and
\BE{tiNdfn0_e}\wt{N}_{\be}^{\R}(\mu_0,\ldots,\mu_k)=
\int_{[\R\ov\fM_{k+1}(\be)]^{\vir}}f_{012\bar{0}}^*\Om_{0,4}\,
\ev_0^*\mu_0\,\ldots\,\ev_k^*\mu_k\,.\EE
For any $\la\!\in\!\cM_{0,4}$, define 
\begin{equation*}\begin{split}
Z_{\la}&=\big\{\big(u,y_0,\ldots,y_k)\!\in\!
f_{012\bar{0}}^{-1}(\la)\!\times\!Y_0\!\times\!\ldots\!\times\!Y_k\!:
\ev_i(u)\!=\!h_i(y_i)~\forall\,i\!=\!0,\ldots,k\big\}\\
&\subset\R\ov\fM_{k+1}(\be)\times Y_0\!\times\!\ldots\!\times\!Y_k\,.
\end{split}\end{equation*}
This subset is a compact oriented 0-dimensional suborbifold, i.e.~a finite set 
of weighted points, if $\la$ is generic.
The number~\eref{tiNdfn0_e} is the signed weighted cardinality $|Z_{\la}|^{\pm}$ 
of this set.\\

\noindent
We prove Theorem~\ref{main2_thm} by explicitly describing the elements of 
$Z_{[1,1]}$ and~$Z_{[1,0]}$, with notation as in Figure~\ref{m04_fig},
and determining their contribution to the number~\eref{tiNdfn0_e}.
The domain~$\Si_u$ of each element~$[u]$ of $Z_{[1,1]}$ and~$Z_{[1,0]}$ consists of 
at least two irreducible components.
If \eref{bndcond_e} holds, $\Si_u$ has an odd number of irreducible components;
the involution~$c_u$ associated with~$u$ restricts to~$c$ on one of 
the components and interchanges the others in pairs.
For dimensional reasons, the number of irreducible components of~$\Si_u$ cannot be greater
than~3 and thus must be either~2 or~3.
Each map~$u$ with its marked points is completely determined by its restriction~$u^{\R}$
to the component~$\Si_u^{\R}$ of~$\Si_u$ preserved by~$c_u$ (if the number of irreducible
components is odd) and its restriction~$u^{\C}$ to either of the other components.\\

\noindent
We depict all possibilities for the elements of $Z_{[1,1]}$ and~$Z_{[1,0]}$
in Figures~\ref{LHS0_fig} and~\ref{RHS0_fig}, respectively.
In each of the first three diagrams in these figures, 
the vertical line represents the irreducible component~$\Si_u^{\R}$ of~$\Si_u$
preserved by~$c_u$, while the two horizontal lines represent the components of~$\Si_u$
interchanged by~$c_u$;
in the last diagram in each figure, the two lines represent the components of~$\Si_u$
interchanged by~$c_u$.
The homology classes next to the lines specify the degrees of~$u$ on
the corresponding components.
The larger dots on the three lines indicate the locations of  the marked points
$z_0^+,z_1^+,z_2^+$; 
we label them by the constraints they map~to,
i.e.~$\mu,\mu_1,\mu_2$, in order to make the connection 
with the expression in Theorem~\ref{main2_thm} more apparent.
If a marked point~$z_i^+$ lies on the bottom component, 
its conjugate lies on the top component.
In such a case, we indicate the conjugate point by a small dot on 
the upper component and label it with~$\bar\mu_i$;
the restriction of~$u$ to the upper component maps this point to the image 
of~$\phi\!\circ\!h_i$.
By the definition of~$Z_{[1,1]}$, each diagram in Figure~\ref{LHS0_fig}
contains a node separating the marked points~$z_0^+,z_1^+$ 
(i.e.~the larger dots labeled by $\mu,\mu_1$) from 
the marked points~$z_2^+,z_0^-$ (i.e.~dots labeled by $\mu_2,\bar\mu$).
Similarly, each diagram in Figure~\ref{RHS0_fig} contains a node separating 
the marked points~$z_0^+,z_2^+$ from the marked points~$z_1^+,z_0^-$.
We arrange the diagrams in both cases so that the pair of marked points containing~$z_0^+$
lies above the other pair.
The remaining marked points, $z_3^{\pm},\ldots,z_k^{\pm}$, are distributed between 
the components in some way.
In the case of the first three diagrams in each figure, 
such a distribution is described by a partition of $\{1,\ldots,k\}$
into subsets $I^+,J,I^-$ of plus-decorated marked points on the top, middle, and bottom components,
respectively.\\

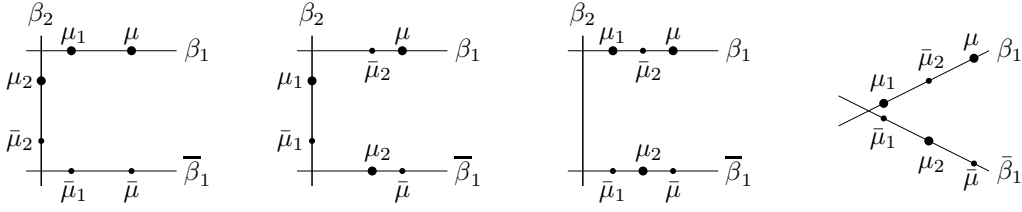
\begin{figure}
\begin{pspicture}(-.5,-.3)(10,2.8)
\psset{unit=.4cm}
\psline[linewidth=.05](3,5)(3,0)
\psline[linewidth=.02](2.5,4.5)(7.5,4.5)\psline[linewidth=.02](2.5,.5)(7.5,.5)
\pscircle*(4,4.5){.15}\pscircle*(6,4.5){.15}\pscircle*(3,3.5){.15}
\rput(4.1,5){\sm{$\mu_1$}}\rput(6,5){\sm{$\mu$}}\rput(2.3,3.5){\sm{$\mu_2$}}
\pscircle*(3,1.5){.1}\rput(2.3,1.5){\sm{$\bar\mu_2$}}
\rput(8.2,4.5){\sm{$\be_1$}}\rput(8.2,.5){\sm{$\ov\be_1$}}\rput(3,5.7){\sm{$\be_2$}}
\pscircle*(4,.5){.1}\pscircle*(6,.5){.1}
\rput(4.1,-.2){\sm{$\bar\mu_1$}}\rput(6,-.2){\sm{$\bar\mu$}}
\psline[linewidth=.05](12,5)(12,0)
\psline[linewidth=.02](11.5,4.5)(16.5,4.5)\psline[linewidth=.02](11.5,.5)(16.5,.5)
\pscircle*(15,4.5){.15}\pscircle*(12,3.5){.15}\pscircle*(14,.5){.15}
\rput(15,5){\sm{$\mu$}}
\rput(11.3,3.5){\sm{$\mu_1$}}\rput(14.2,1.1){\sm{$\mu_2$}}
\pscircle*(12,1.5){.1}\rput(11.3,1.5){\sm{$\bar\mu_1$}}
\rput(17.2,4.5){\sm{$\be_1$}}\rput(17.2,.5){\sm{$\ov\be_1$}}\rput(12,5.7){\sm{$\be_2$}}
\pscircle*(14,4.5){.1}\rput(14.2,3.8){\sm{$\bar\mu_2$}}
\pscircle*(15,.5){.1}\rput(15,-.2){\sm{$\bar\mu$}}
\psline[linewidth=.05](21,5)(21,0)
\psline[linewidth=.02](20.5,4.5)(25.5,4.5)\psline[linewidth=.02](20.5,.5)(25.5,.5)
\pscircle*(22,4.5){.15}\pscircle*(24,4.5){.15}\pscircle*(23,.5){.15}
\rput(22,5){\sm{$\mu_1$}}\rput(24,5){\sm{$\mu$}}
\rput(23.2,1.1){\sm{$\mu_2$}}
\rput(26.2,4.5){\sm{$\be_1$}}\rput(26.2,.5){\sm{$\ov\be_1$}}\rput(21,5.7){\sm{$\be_2$}}
\pscircle*(23,4.5){.1}\rput(23.2,3.8){\sm{$\bar\mu_2$}}
\rput(22,-.2){\sm{$\bar\mu_1$}}\rput(24,-.2){\sm{$\bar\mu$}}
\pscircle*(22,.5){.1}\pscircle*(24,.5){.1}
\psline[linewidth=.02](29.5,2)(34.5,4.5)\psline[linewidth=.02](29.5,3)(34.5,.5)
\pscircle*(31,2.75){.15}\pscircle*(32.5,3.5){.1}\pscircle*(34,4.25){.15}
\pscircle*(31,2.25){.1}\pscircle*(32.5,1.5){.15}\pscircle*(34,.75){.1}
\rput(33.9,4.8){\sm{$\mu$}}\rput(32.6,4.2){\sm{$\bar\mu_2$}}\rput(31,3.45){\sm{$\mu_1$}}
\rput(33.9,.2){\sm{$\bar\mu$}}\rput(32.6,.7){\sm{$\mu_2$}}\rput(31,1.55){\sm{$\bar\mu_1$}}
\rput(35.2,4.5){\sm{$\be_1$}}\rput(35.2,.5){\sm{$\bar\be_1$}}
\end{pspicture}
\caption{Domains of elements of $Z_{[1,1]}$}
\label{LHS0_fig}
\end{figure}

\begin{figure}
\begin{pspicture}(-.5,-.3)(10,2.8)
\psset{unit=.4cm}
\psline[linewidth=.05](3,5)(3,0)
\psline[linewidth=.02](2.5,4.5)(7.5,4.5)\psline[linewidth=.02](2.5,.5)(7.5,.5)
\pscircle*(4,4.5){.15}\pscircle*(6,4.5){.15}\pscircle*(3,3.5){.15}
\rput(4.1,5){\sm{$\mu_2$}}\rput(6,5){\sm{$\mu$}}\rput(2.3,3.5){\sm{$\mu_1$}}
\pscircle*(3,1.5){.1}\rput(2.3,1.5){\sm{$\bar\mu_1$}}
\rput(8.2,4.5){\sm{$\be_1$}}\rput(8.2,.5){\sm{$\ov\be_1$}}\rput(3,5.7){\sm{$\be_2$}}
\pscircle*(4,.5){.1}\pscircle*(6,.5){.1}
\rput(4.1,-.2){\sm{$\bar\mu_2$}}\rput(6,-.2){\sm{$\bar\mu$}}
\psline[linewidth=.05](12,5)(12,0)
\psline[linewidth=.02](11.5,4.5)(16.5,4.5)\psline[linewidth=.02](11.5,.5)(16.5,.5)
\pscircle*(15,4.5){.15}\pscircle*(12,3.5){.15}\pscircle*(14,.5){.15}
\rput(15,5){\sm{$\mu$}}
\rput(11.3,3.5){\sm{$\mu_2$}}\rput(14.2,1.1){\sm{$\mu_1$}}
\pscircle*(12,1.5){.1}\rput(11.3,1.5){\sm{$\bar\mu_2$}}
\rput(17.2,4.5){\sm{$\be_1$}}\rput(17.2,.5){\sm{$\ov\be_1$}}\rput(12,5.7){\sm{$\be_2$}}
\pscircle*(14,4.5){.1}\rput(14.2,3.8){\sm{$\bar\mu_1$}}
\pscircle*(15,.5){.1}\rput(15,-.2){\sm{$\bar\mu$}}
\psline[linewidth=.05](21,5)(21,0)
\psline[linewidth=.02](20.5,4.5)(25.5,4.5)\psline[linewidth=.02](20.5,.5)(25.5,.5)
\pscircle*(22,4.5){.15}\pscircle*(24,4.5){.15}\pscircle*(23,.5){.15}
\rput(22,5){\sm{$\mu_2$}}\rput(24,5){\sm{$\mu$}}\rput(23.2,1.1){\sm{$\mu_1$}}
\rput(26.2,4.5){\sm{$\be_1$}}\rput(26.2,.5){\sm{$\ov\be_1$}}\rput(21,5.7){\sm{$\be_2$}}
\pscircle*(23,4.5){.1}\rput(23.2,3.8){\sm{$\bar\mu_1$}}
\rput(22,-.2){\sm{$\bar\mu_2$}}\rput(24,-.2){\sm{$\bar\mu$}}
\pscircle*(22,.5){.1}\pscircle*(24,.5){.1}
\psline[linewidth=.02](29.5,2)(34.5,4.5)\psline[linewidth=.02](29.5,3)(34.5,.5)
\pscircle*(31,2.75){.15}\pscircle*(32.5,3.5){.1}\pscircle*(34,4.25){.15}
\pscircle*(31,2.25){.1}\pscircle*(32.5,1.5){.15}\pscircle*(34,.75){.1}
\rput(33.9,4.8){\sm{$\mu$}}\rput(32.6,4.2){\sm{$\bar\mu_1$}}\rput(31,3.45){\sm{$\mu_2$}}
\rput(33.9,.2){\sm{$\bar\mu$}}\rput(32.6,.7){\sm{$\mu_1$}}\rput(31,1.55){\sm{$\bar\mu_2$}}
\rput(35.2,4.5){\sm{$\be_1$}}\rput(35.2,.5){\sm{$\bar\be_1$}}
\end{pspicture}
\caption{Domains of elements of $Z_{[1,0]}$}
\label{RHS0_fig}
\end{figure}
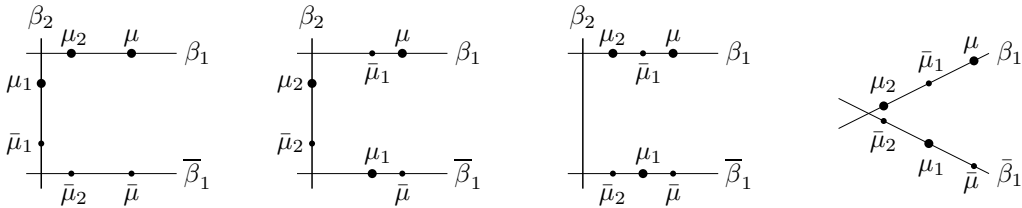

\noindent
Each element~$u$ of $Z_{[1,1]}$ and $Z_{[1,0]}$ 
described by the first three diagrams in Figures~\ref{LHS0_fig}
and~\ref{RHS0_fig}, respectively, is an element~of
the subspace
\BE{fMsp_e} \cN_{\be_1,\be_2;I^+,J,I^-}(\bh)\subset 
\cN_{\be_1,\be_2;I^+,J,I^-}\times Y_0\!\times\!\ldots\!\times\!Y_k\subset 
\R\ov\fM_{k+1}(\be)\times Y_0\!\times\!\ldots\!\times\!Y_k\EE
defined in \eref{cNhdfn_e} for some $\be_1,\be_2$ and $I^+,J,I^-$
with
$\be\!=\!\fd(\be_1)\!+\!\be_2$ and $\{1,\ldots,k\}\!=\!I^+\!\sqcup\!J\!\sqcup\!I^-$.
An element~$u$ of the first space in~\eref{fMsp_e}
has a well-defined nonzero weight~$w(u)$ with respect to the orientation
of $\cN_{\be_1,\be_2;I^+,J,I^-}$ described below~\eref{cNisom_e}.
The sum of these weights over all elements~$u$ represented by a fixed diagram
with fixed $(\be_1,\be_2)$ and $(I^+,J,I^-)$ is the signed weighted cardinality 
of $\cN_{\be_1,\be_2;I^+,J,I^-}(\bh)$ computed via the usual Kunneth decomposition;
see the first paragraph in the proof of Proposition~\ref{KunnethSplit_prp}.
As an isolated element of $Z_{[1,1]}$ or $Z_{[1,0]}$,
$u$ has a well-defined contribution~$\ve(u)w(u)$ to the number~\eref{tiNdfn0_e},
i.e.~the signed number of nearby elements of~$Z_{\la}$, 
with $\la\!\in\!\cM_{0,4}$ close to $[1,1]$ or~$[1,0]$.
By Lemma~\ref{sign0_lmm} below, $\ve(u)\!=\!1$ for all elements~$u$ 
represented by the first diagrams in Figures~\ref{LHS0_fig} and~\ref{RHS0_fig}, 
$\ve(u)=-1$ for the second diagrams  in these figures,
and $\ve(u)\!=\!0$ for the third diagrams.
Even if the contributions from the third diagrams were nonzero,
they would have been the same for $Z_{[1,1]}$ and $Z_{[1,0]}$ by symmetry
and so would have had no effect on the recursion of Theorem~\ref{main2_thm}.
The reason behind Lemma~\ref{sign0_lmm} is that the oriented normal bundle of
$\cN_{\be_1,\be_2;I^+,J,I^-}$ inside $\R\ov\fM_{k+1}(\be)$ is given by
the complex line bundle of smoothings of the top node in the first three diagrams,
which is conjugate to the complex line bundle of  
smoothings of the bottom node, while 
the complex tangent bundle of~$[1,1]$ or~$[1,0]$ in~$\ov\cM_{0,4}$
corresponds to the smoothings of the node  separating $\{z_0^+,z_1^+\}$ from 
$\{z_2^+,z_0^-\}$ in the case of~$[1,1]$ and  
$\{z_0^+,z_2^+\}$ from  $\{z_1^+,z_0^-\}$ in the case of~$[1,0]$.\\
 
\noindent
The remaining elements of $Z_{[1,1]}$ and $Z_{[1,0]}$,
i.e.~those described by the last diagrams in Figures~\ref{LHS0_fig}
and~\ref{RHS0_fig}, respectively, form one-dimensional subspaces
$Z_{[1,1]}'\!\subset\!Z_{[1,1]}$ and $Z_{[1,0]}'\!\subset\!Z_{[1,0]}$;
these diagrams appear only if \eref{bndcond_e} is not satisfied.
By Lemma~\ref{sign2_lmm} below, no topological
component of $Z_{[1,1]}'$ or $Z_{[1,0]}'$ contributes 
to the number~\eref{tiNdfn0_e}.

\begin{lmm}\label{sign0_lmm}
Suppose $u\!\in\!Z_{[1,1]}$ and the domain of $u$ contains an irreducible
component~$\Si_u^{\R}$ fixed by the involution~$c_u$.
\begin{enumerate}[label=(\arabic*),leftmargin=*]
\item If $\Si_u^{\R}$ contains the marked point~$z_2^+$, $\ve(u)\!=\!1$.
\item If $\Si_u^{\R}$ contains the marked point~$z_1^+$, $\ve(u)\!=\!-1$. 
\item If $\Si_u^{\R}$ contains neither of the marked points~$z_1^+,z_2^+$, $\ve(u)\!=\!0$. 
\end{enumerate}
The same statements with~1 and~2 interchanged hold for $u\!\in\!Z_{[1,0]}$.
\end{lmm}

\begin{proof} Let $L_{\bh}\lra Z_{[1,1]}\!-\!Z_{[1,1]}',Z_{[1,0]}\!-\!Z_{[1,0]}'$
be the restriction of the line bundle $\pi_1^*L$ defined in~\eref{pbLdfn_e}.
As in complex GW-theory, a small modification of the gluing map~\eref{RPhidfn_e} 
gives rise  to a gluing~map 
$$\Phi_{\bh}\!: U_{\bh}\lra   \R\ov\fM_{\bh}(\be),$$
where $U_{\bh}\!\subset\!L_{\bh}$ is a neighborhood of the zero section in~$L_{\bh}$,
which lifts any pre-specified family of smoothings of the domain.
Over the subsets $\cN_{\be_1,\be_2;I^+,J,I^-}(\bh)$ corresponding to 
the first two diagrams in Figures~\ref{LHS0_fig} and~\ref{RHS0_fig},
$\Phi_{\bh}$ is orientation-preserving by Lemma~\ref{Rgluing0_lmm}.
The differential 
\BE{fgldiff_e} \tnd\big\{f_{012\bar0}\!\circ\!\Phi_{\bh}\big\}\!: 
L\lra  \big\{f_{012\bar0}\!\circ\!\Phi_{\bh}\big\}^*T\cM_{0,4}\EE
is the composition of the differential for smoothing the nodes in 
$\ov\fM_{k+2}(\be)$,
$$\tnd(f_{012\bar0}\!\circ\!\Phi^{\C})\!: L\!\oplus\!L'\lra 
\big\{f_{012\bar0}\!\circ\!\Phi^{\C}\big\}^*T\cM_{0,4}\,,$$
where $L'$ is the analogue of $L$ for the second node, with the embedding
$$L\lra L\oplus L', \qquad \ups\lra \big(\ups,\tnd c(\ups)\big).$$
The restriction of the latter differential to the component, $L$ or~$L'$,
corresponding to the node separating off two of the marked points~$\{z_0^+,z_1^+,z_2^+,z_0^-\}$
is a $\C$-linear isomorphism, while the restriction to the other component is trivial.
Over the subsets $\cN_{\be_1,\be_2;I^+,J,I^-}(\bh)$ corresponding to 
the first diagrams in Figures~\ref{LHS0_fig} and~\ref{RHS0_fig},
the former component is~$L$ and \eref{fgldiff_e} is an orientation-preserving map.
Over the subsets $\cN_{\be_1,\be_2;I^+,J,I^-}(\bh)$ corresponding to 
the second diagrams in Figures~\ref{LHS0_fig} and~\ref{RHS0_fig},
the former component is~$L'$ and \eref{fgldiff_e} is an orientation-reversing map.
This establishes the first two statements of Lemma~\ref{sign0_lmm}.\\

\noindent
Near the spaces $\cN_{\be_1,\be_2;I^+,J,I^-}$ corresponding to 
the second-to-last diagrams in Figures~\ref{LHS0_fig} and~\ref{RHS0_fig}, 
the morphism 
$$f_{012\bar0}\!:\ov\fM_{k+2}(\be)\lra \ov\cM_{0,4}$$
is locally of the form 
$$L\oplus L'\lra \ov\cM_{0,4}, \qquad (\ups,\ups')\lra a\ups\ups',$$
for some $a$ dependent only on~$\cN_{\be_1,\be_2;I^+,J,I^-}$.
Thus, the restriction of~$f_{012\bar0}$ to $\R\ov\fM_{\bh}(\be)$ is locally of 
the~form
$$L\lra \ov\cM_{0,4}, \qquad \ups\lra a\ups\bar\ups\,.$$
The image of this maps is one-dimensional, 
which implies the third claim of Lemma~\ref{sign0_lmm}.
\end{proof}

\begin{lmm}\label{sign2_lmm}
The contribution of every topological component of $Z_{[1,1]}'$
and $Z_{[1,0]}'$ to the number~\eref{tiNdfn0_e} is~0.
\end{lmm}

\begin{proof}
If $(X,\om)$ is strongly semi-positive, 
each topological component~$C$ of $Z_{[1,1]}'$ and $Z_{[1,0]}'$ is a circle.
In general, $C$ is obtained by gluing several circles along some intervals
as specified by branching of the multi-section~$\fs$ used to regularize the moduli space.
Along~$C$, $\fs$ can be represented by several single-valued sections obtained
by gluing together local representatives as in~\cite[Section~3]{FO}.
Each such section determines disjoint circles in $Z_{[1,1]}'$ or~$Z_{[1,0]}'$.  
For the purposes of studying the nearby elements of~$Z_{\la}$ that lie
in the zero set of each of these sections, it is sufficient to assume that 
each topological component~$C$ of $Z_{[1,1]}'$ and $Z_{[1,0]}'$ is the circle~$S^1$.\\

\noindent
There is a gluing~map
\BE{Phidfn0_e}\Phi\!:C\!\times\!(-\de,\de)\lra 
\bigcup_{\la\in\ov\cM_{0,4}}\!\!\!\!Z_{\la}\,\EE
for $\de\!\in\!\R^+$ sufficiently small,
which restricts to the identity along~$C\!\times\!\{0\}$;
it is obtained via a $(\phi,c)$-equivariant version of a standard gluing
construction, such as in \cite[Section~3]{LT}, with $c\!=\!\tau,\eta$.
In particular, we can normalize the elements of~$C$
by setting the marked point $z_0^+\!=\!0$ and the node to~$\i$ on one of the components of the domain
and setting $z_2^+\!=\!1$, $z_0^-\!=\!\i$, and the node to~$0$ on the other component.
For each $t\!\in\!\R^*$ sufficiently small, we can define a marked pregluing map
$u_t\!:\P^1\!\lra\!X$ with the same values at the marked points as~$u$
and with the cross-ratio~$f_{012\bar{0}}$ given~by 
$$\la= f_{012\bar{0}}(u_t)=t\,z_1^+(u)\in\C^*\subset\cM_{0,4}$$
in some chart on $\ov\cM_{0,4}$.
This map can then be deformed to an element~$\ti{u}_t$ of~$Z_{\la}$, 
with the same $\la\!\in\!\cM_{0,4}$.
Since~$C$ consists of two-bubble maps (no additional bubbling),
the gluing construction can be carried out on the entire space~$C$ 
in this case.
\\

\noindent
Let $\bar\R^+\!=\!\R^{\ge0}$ and $\bar\R^-\!=\!\R^{\le0}$.
The restriction of $f_{012\bar{0}}\!\circ\!\Phi$ to 
$C\!\times\!((-\de,\de)\!\cap\!\bar\R^{\pm})$ 
is the composition~of the~maps
\begin{alignat*}{2}
C\!\times\!((-\de,\de)\!\cap\!\bar\R^{\pm})&\lra \big\{z\!\in\!\C\!:\,|z|\!<\!\de\big\}, 
&\qquad (e^{\fI\th},t)&\lra |t|e^{\fI\th}\,,\\
\big\{z\!\in\!\C\!:\,|z|\!<\!\de\big\}&\lra\C,&\qquad
re^{\fI\th}&\lra \pm rz_1^+\big(e^{\fI\th}\big).
\end{alignat*}
The two maps, for $\bar\R^+$ and~$\bar\R^-$, described by the first line above have
opposite local degrees, while the two maps described by the second map have 
the same local degrees.
Thus, the local degree of the~map
$$f_{012\bar{0}}\!\circ\!\Phi: C\!\times\!(-\de,\de)\lra \C, \qquad 
(u,t)\lra f_{012\bar{0}}\big(\Phi(u,t)\big)=t\,z_1^+(u),$$
is zero. 
This implies
the claim.
\end{proof}

\begin{proof}[{\bf\emph{Proof of Theorem~\ref{main2_thm}}}]
We compute the number~\eref{tiNdfn0_e} by adding up the contributions 
from the elements represented by the diagrams in Figure~\ref{LHS0_fig}.
We then compute it from the diagrams in Figure~\ref{RHS0_fig}
and compare the two expressions for the number~\eref{tiNdfn0_e}.\\

\noindent
By Lemmas~\ref{sign0_lmm} and~\ref{sign2_lmm}, only the first two 
diagrams in Figure~\ref{LHS0_fig} and~\ref{RHS0_fig} contribute.
By the Kunneth decomposition, as in the first part of the proof of 
Proposition~\ref{KunnethSplit_prp}, and by Proposition~\ref{twistGW_prp},
the signed cardinality of~$\cN_{\be_1,\be_2;I^+,J,I^-}(\bh)$ is given~by
\BE{diagcontr0_e}
\!\big|\cN_{\be_1,\be_2;I^+,J,I^-}(\bh)\big|^{\pm}
=\sum_{\begin{subarray}{c}1\le i\le\ell\\ \ga^i\in H^{2*}(X)^{\phi}_-\end{subarray}}
\!\!\!\!\!\!\!\!
\blr{\mu_0,\mu_{I^+\sqcup I^-},\ga_i}_{\be_1}^X\blr{\mu_J,\ga^i}_{\be_2}^{\R},\EE
where $\lr{\ldots}^{\R}$ denotes $\lr{\ldots}^{\phi,c}$ if \eref{bndcond_e} holds
and $\lr{\ldots}^{\phi}$ otherwise.
If $\be_1\!=\!0$ and the complex invariant in~\eref{diagcontr0_e} is nonzero, 
then $|I^+\sqcup I^-|\!=\!1$ for dimensional reasons.\\

\noindent
By Lemma~\ref{sign0_lmm}(1), the contribution to the number~\eref{tiNdfn0_e} from 
the first diagram in Figure~\ref{LHS0_fig} equals the sum of~\eref{diagcontr0_e}
over all admissible $(\be_1,\be_2)$ and $(I,J)$ with $1\!\in\!I$ and $2\!\in\!J$
and all partitions of $I\!-\!\{1\}$ into two subsets~$I^+$ and~$I^-$.
By Lemma~\ref{sign0_lmm}(2), the contribution to the number~\eref{tiNdfn0_e} from 
the second diagram in Figure~\ref{LHS0_fig} equals the negative of the sum of~\eref{diagcontr0_e}
over all admissible $(\be_1,\be_2)$ and $(I,J)$ with $2\!\in\!I$ and $1\!\in\!J$
and all partitions of $I\!-\!\{2\}$ into two subsets~$I^+$ and~$I^-$.
Thus, the number~\eref{tiNdfn0_e} equals
\begin{equation*}\begin{split}
\blr{\mu\mu_1,\mu_2,\mu_3,\ldots,\mu_k}_{\be}^{\R}
-\blr{\mu_1,\mu\mu_2,\mu_3,\ldots,\mu_k}_{\be}^{\R}
+\sum_{\begin{subarray}{c}\fd(\be_1)+\be_2=\be\\ \be_1,\be_2\in H_2(X)-\{0\} \end{subarray}}
\!\!\!\sum_{I\sqcup J=\{3,\ldots,k\}}\!\!\!
\sum_{\begin{subarray}{c}1\le i\le\ell\\ \ga^i\in H^{2*}(X)^{\phi}_-\end{subarray}}
\!\!\!\!\!\!2^{|I|}\Bigg(\quad&\\ 
\blr{\mu,\mu_1,\mu_I,\ga_i}_{\be_1}^X\!\blr{\mu_2,\mu_J,\ga^i}_{\be_2}^{\R}
-\blr{\mu,\mu_2,\mu_I,\ga_i}_{\be_1}^X\!\blr{\mu_1,\mu_J,\ga^i}_{\be_2}^{\R}
&\Bigg).
\end{split}\end{equation*}
Considering the first two diagrams in Figure~\ref{RHS0_fig},
we similarly find that the number~\eref{tiNdfn0_e} equals
\begin{equation*}\begin{split}
\blr{\mu_1,\mu\mu_2,\mu_3,\ldots,\mu_k}_{\be}^{\R}
-\blr{\mu\mu_1,\mu_2,\mu_3,\ldots,\mu_k}_{\be}^{\R}
+\sum_{\begin{subarray}{c}\fd(\be_1)+\be_2=\be\\ \be_1,\be_2\in H_2(X)-\{0\} \end{subarray}}
\!\!\!\sum_{I\sqcup J=\{3,\ldots,k\}}\!\!\!
\sum_{\begin{subarray}{c}1\le i\le\ell\\ \ga^i\in H^{2*}(X)^{\phi}_-\end{subarray}}
\!\!\!\!\!\!2^{|I|}\Bigg(\quad&\\
\blr{\mu,\mu_2,\mu_I,\ga_i}_{\be_1}^X\!\blr{\mu_1,\mu_J,\ga^i}_{\be_2}^{\R}
-\blr{\mu,\mu_1,\mu_I,\ga^i}_{\be_1}^X\!\blr{\mu_2,\mu_J,\ga^i}_{\be_2}^{\R}
&\Bigg).
\end{split}\end{equation*}
Setting the two expressions equal, we obtain the formula in Theorem~\ref{main2_thm}.
\end{proof}

\section{Miscellaneous odds and ends}
\label{quant_sec}

\noindent
We begin this section by deducing Corollaries~\ref{oddnums_crl} and~\ref{main_crl} 
from Corollary~\ref{main0_crl}.
We then deduce Corollaries~\ref{inveq_crl} and~\ref{inveqCI_crl} 
from Theorems~\ref{main2_thm} and~\ref{evenvan_thm} and
relate the formula of Theorem~\ref{main2_thm} to the quantum
product on the cohomology of the symplectic manifold~$(X,\om)$.
We conclude with tables of counts of real curves in~$\P^3$, $\P^5$, and~$\P^7$
and a discussion of their compatibility.

\begin{proof}[{\bf\emph{Proof of Corollary~\ref{oddnums_crl}}}]
(1) The claim holds for $d,k\!=\!1$, since there is a unique $\phi$-real line 
through any point in~$\P^{2n-1}$.
Modulo~2, the recursion of Corollary~\ref{main0_crl} becomes
\begin{equation*}\begin{split}
&\blr{c_1,c_2,c_3,\ldots,c_k}_d^{\phi}
\cong\blr{c_1\!+\!c_2\!-\!1,c_3,\ldots,c_k}_d^{\phi}
+\sum_{\begin{subarray}{c}2d_1+d_2=d\\ d_1,d_2\ge1 \end{subarray}}
\sum_{\begin{subarray}{c}2i+j=2n-1\\ i,j\ge1 \end{subarray}}\!\!\!
\Bigg(\\
&\hspace{1in}
\blr{c_1\!-\!1,c_2,2i}_{d_1}^{\P^{2n-1}}\!\blr{c_3,\ldots,c_k,j}_{d_2}^{\phi}
+d_1\blr{c_1\!-\!1,2i}_{d_1}^{\P^{2n-1}}\!\blr{c_2,\ldots,c_k,j}_{d_2}^{\phi}\Bigg).
\end{split}\end{equation*}
For dimensional reasons, 
$$\blr{c_1\!-\!1,c_2,2i}_{d_1}^{\P^{2n-1}}=0  \quad\forall~d_1\!\ge\!2,\qquad
\blr{c_1\!-\!1,2i}_{d_1}^{\P^{2n-1}}=0 \quad\forall~d_1\!\ge\!1.$$
Thus, the mod 2 recursion reduces~to
\begin{equation*}\begin{split}
&\blr{c_1,c_2,c_3,\ldots,c_k}_d^{\phi}
\cong\blr{c_1\!+\!c_2\!-\!1,c_3,\ldots,c_k}_d^{\phi}
+\!\!\!\!\sum_{\begin{subarray}{c}2i+j=2n-1\\ i,j\ge1 \end{subarray}}\!\!\!\!\!\!\!\!
\blr{c_1\!-\!1,c_2,2i}_1^{\P^{2n-1}}\!\blr{c_3,\ldots,c_k,j}_{d-2}^{\phi}.
\end{split}\end{equation*}
If $c_1\!+\!c_2\!-\!1\!\le\!2n\!-\!1$, 
$\lr{c_1\!+\!c_2\!-\!1,c_3,\ldots,c_k}_d^{\phi}$ is odd by induction on~$k$
for a fixed $d\!\ge\!1$ odd and
$\lr{c_1\!-\!1,c_2,2i}_1^{\P^{2n-1}}\!=\!0$ for dimensional reasons.
It follows that 
$$\blr{c_1,c_2,c_3,\ldots,c_k}_d^{\phi} \cong1\mod2$$
in this case.
If $c_1\!+\!c_2\!-\!1\!>\!2n\!-\!1$, then 
$$d\ge3, \qquad \lr{c_1\!+\!c_2\!-\!1,c_3,\ldots,c_k}_d^{\phi}=0,$$
and $\lr{c_1\!-\!1,c_2,2i}_1^{\P^{2n-1}}\!=\!0$ if 
$c_1\!+\!c_2\!+\!2i\!\neq\!4n$.
Since the linear span of general $\P^{2n-1-(c_1-1)}$ and $\P^{2n-1-c_2}$,
with 
$$1\le c_1,c_2 \le 2n\!-\!1  \qquad\hbox{and}\qquad c_1\!+\!c_2> 2n,$$ 
in~$\P^{2n-1}$ is a $\P^{4n-c_1-c_2}$, it intersects a general 
$\P^{2n-1-(4n-c_1-c_2)}$ in a single point.
This point lies on the unique line passing through linear subspaces
of~$\P^{2n-1}$ of codimensions $c_1,c_2,2i$ whenever  
$c_1\!+\!c_2\!+\!2i\!=\!4n$.
Thus,
$$\blr{c_1,c_2,c_3,\ldots,c_k}_d^{\phi} \cong
\blr{c_3,\ldots,c_k,2n\!-\!1-(4n\!-\!c_1\!-\!c_2)}_{d-2}^{\phi}
\mod2$$
in this case; the last number is odd by the induction on~$d$.\\

\noindent
(2) The second claim of Corollary~\ref{oddnums_crl} follows from the first and 
\cite[Theorem~1.8]{Teh}; 
the latter is contained in Corollary~\ref{inveqCI_crl} and Theorem~\ref{evenvan_thm}.
\end{proof}

\begin{proof}[{\bf\emph{Proof of Corollary~\ref{main_crl}}}]
The first statement follows immediately from Corollary~\ref{main0_crl} 
and implies~that 
$$N_d^{\R}\cong_4\begin{cases}1,&\hbox{if}~d\in\Z^+\!-\!2\Z;\\
0,&\hbox{if}~d\in2\Z^+\,.
\end{cases}$$
We use simultaneous induction on the degree $d$ to show that 
$$N_d^{\C}\cong_4\begin{cases}1,&\hbox{if}~d\in\Z^+\!-\!2\Z;\\
0,&\hbox{if}~d\in2\Z^+\,;
\end{cases}
\quad\hbox{and}\qquad
\wt{N}_d^{\C}\cong_4\begin{cases}1,&\hbox{if}~d\in\Z^+\!-\!2\Z~\hbox{or}~d\!=\!2;\\
2,&\hbox{if}~d\!=\!4;\\
0,&\hbox{if}~d\in2\Z^+\!-\!\{2,4\}\,;
\end{cases}$$
from the base case $N_1^{\C}\!=\!1$ (the number of lines through 2 points in~$\P^3$).
By \cite[Theorem~10.4]{RT},
\BE{RTrec_e}\begin{split}
N_d^{\C}&=\sum_{\begin{subarray}{c}d_1+d_2=d\\ d_1,d_2\ge1\end{subarray}}
\!\!\Bigg(\!d_2^2\binom{2d\!-\!3}{2d_1\!-\!2}
  -d_1d_2\binom{2d\!-\!3}{2d_1\!-\!1}\!\!\!\Bigg)\wt{N}_{d_1}^{\C}N_{d_2}^{\C}\,,\\
\wt{N}_d^{\C}&=dN_d^{\C}+
\sum_{\begin{subarray}{c}d_1+d_2=d\\ d_1,d_2\ge1\end{subarray}}\!\!
\Bigg(\!d_1d_2^2\binom{2d\!-\!2}{2d_1\!-\!1}
  -d_2^3\binom{2d\!-\!2}{2d_1\!-\!2}\!\!\!\Bigg)\wt{N}_{d_1}^{\C}N_{d_2}^{\C}\,.
\end{split}\EE
By \eref{RTrec_e}, $\ti{N}_1^{\C},N_3^{\C}\!=\!1$ and $\wt{N}_3^{\C}\!=\!5$.
Modulo~4, the summands  in~\eref{RTrec_e} with~$d_2$ even  vanish
(by Corollary~\ref{oddnums_crl}, $N_{d_2}^{\C}\!\in\!2\Z$ if $d_2\!\in\!2\Z$).
Thus, by the inductive assumption only the summands with $d_1\!=\!2,4$ 
may be nonzero in either sum in~\eref{RTrec_e} with $d\!\ge\!5$ odd.
These two summands contribute $d\!-\!2$ and $d\!-\!3$, respectively,
i.e.~1 together, to the first sum.
They contribute $d\!-\!1$ and $0$, respectively,
i.e.~again~1 together with the term $dN_d^{\C}$, to the second sum.\\

\noindent
For $d\!\in\!2\Z^+$, \eref{RTrec_e} and the inductive assumptions give
\BE{RTrec_e2}\begin{split}
N_d^{\C}&\cong_4\!\!
\sum_{\begin{subarray}{c}d_1+d_2=d\\ d_1,d_2\ge1\tn{~odd}\end{subarray}}
\!\!\!\!\!\Bigg(\!\!\!\binom{2d\!-\!3}{2d_1\!-\!2}
  -(d\!-\!1)\binom{2d\!-\!3}{2d_1\!-\!1}\!\!\!\Bigg)
= (2\!-\!d)\!\!\sum_{\begin{subarray}{c}d_1+d_2=d\\ d_1,d_2\ge1\tn{~odd}\end{subarray}}
\!\!\!\!\!\binom{2d\!-\!3}{2d_1\!-\!1} \,,\\
\wt{N}_d^{\C}&\cong_4\!\!
\sum_{\begin{subarray}{c}d_1+d_2=d\\ d_1,d_2\ge1\tn{~odd}\end{subarray}}\!\!
\!\!\!\Bigg(\!d_1\binom{2d\!-\!2}{2d_1\!-\!1}
  -d_2\binom{2d\!-\!2}{2d_1\!-\!2}\!\!\!\Bigg)
  =\frac12\sum_{\begin{subarray}{c}d_1+d_2=d\\ d_1,d_2\ge1\tn{~odd}\end{subarray}}\!\!
\!\!\!\binom{2d\!-\!2}{2d_1\!-\!1}\,.
\end{split}\EE
By symmetry, the last expression on the first line above equals
$$\frac{2\!-\!d}{2}
\!\!\sum_{\begin{subarray}{c}d_1+d_2=d\\ d_1,d_2\ge1\tn{~odd}\end{subarray}}
\!\!\!\!\!\Bigg(\!\!\binom{2d\!-\!3}{2d_1\!-\!1}\!+\!
\binom{2d\!-\!3}{2d_2\!-\!1}\!\!\Bigg)
=\frac{2\!-\!d}{2}
\!\!\sum_{\begin{subarray}{c}d_1+d_2=d\\ d_1,d_2\ge1\tn{~odd}\end{subarray}}
\!\!\!\!\!\binom{2d\!-\!2}{2d_1\!-\!1}\,.$$
Each of the last binomial coefficients is even.
If in addition $d\!\in\!4\Z$, these coefficients come in pairs:
the one for $d_1$ and $d\!-\!d_1$ are the same.
This shows that $N_d^{\C}\!\in\!4\Z$ if $d\!\in\!2\Z$.\\

\noindent
By~\eref{RTrec_e2}, $\wt{N}_2^{\C}\!\cong_4\!1$ and $\wt{N}_4^{\C}\!\cong_4\!2$.
Suppose $d\!=\!2(d'\!+\!1)$ with $d'\!\ge\!2$, $d_1\!=\!2d_1'\!+\!1$, and 
$d_2\!=\!2d_2'\!+\!1$ (so that $d_1'\!+\!d_2'\!=\!d'$).
By Kummer's Theorem, 
the highest power of~$2$ that divides half of the last binomial coefficient 
in~\eref{RTrec_e2} is the number~$c_2(d_1',d_2')$ of carries in the addition 
of $d_1'$ and~$d_2'$ modulo~2.
The pairs $(d_1',d_2')$ for which $c_2(d_1',d_2')\!=\!0$ are obtained from~$d'$ 
by distributing the 1's in the binary representation of~$d'$ between~$d_1'$ and~$d_2'$.
Thus, the number of such pairs~$(d_1',d_2')$ is $2^{\#}$, where $\#$ is the number
of 1's in the binary representation of~$d'$.
Since $d_1'\!\neq\!d_2'$ for such pairs,
the contribution from $(d_1',d_2')$ and $(d_2',d_1')$ to
the last expression in~\eref{RTrec_e2}, including the half factor, is~2 modulo~4.
Thus, the contribution from all such pairs to the last expression in~\eref{RTrec_e2}
is~$2^{\#}$.
The contribution from any other pair $(d_1',d_2')$ is divisible by~2, 
since $c_2(d_1',d_2')\!\ge\!1$, and such pairs come in pairs giving 
the same contribution to~\eref{RTrec_e2}, unless $d_1'\!=\!d_2'$.
If $d_1'\!=\!d_2'$ and thus $d'\!\in\!2\Z^+$, $c_2(d_1',d_2')\!=\!1$ if and only 
if $\#\!=\!1$.
Thus, if $d'\!\in\!2\Z^+$, the total contribution to~\eref{RTrec_e2} from the terms with 
$c_2(d_1',d_2')\!=\!0$ and the term with $d_1'\!=\!d_2'$ is~0 modulo~4.
If $d'\!\not\in\!2\Z^+$, $\#\!\ge\!2$, since $d'\!\ge\!2$, and so
this contribution is still~0 modulo~4. 
This shows that $\wt{N}_d^{\C}\!\in\!4\Z$ if $d\!\in\!2\Z$ and $d\!\ge\!6$.
\end{proof}

\begin{proof}[{\bf\emph{Proof of Corollary~\ref{inveq_crl}}}]
We can assume that $X$ is connected and thus $H^0(X)^{\phi}_-\!=\!\{0\}$.
By Gromov's compactness theorem, we can rescale~$\om$ so that 
$$\inf\!\big\{\om(\be)\!: \be\!\in\!H_{\eff}(X)_{\phi}\big\}=1.$$
We prove the claim by induction on the~number 
$$\lr{\be}_k\equiv \om(\be)\!+\!k\in(1,\i)
\qquad\forall\,k\!\in\!\Z^+,\,\be\!\in\!H_{\eff}(X)_{\phi}\,.$$
Let $\{\ga_i\}_{i\le\ell}$ and $\{\ga^i\}_{i\le\ell}$ be as in Theorem~\ref{main2_thm}.
By the divisor relation,
\BE{inveq_e5b}\begin{split}
\blr{\mu_1,\ldots,\mu_k}_{\!\be}^{\!\phi,c}&=\lr{\mu_2,\be}
\blr{\mu_1,\mu_3,\ldots,\mu_k}_{\!\be}^{\!\phi,c}, \\
\blr{\mu_1,\ldots,\mu_k}_{\!\be}^{\!\phi}&=\lr{\mu_2,\be}
\blr{\mu_1,\mu_3,\ldots,\mu_k}_{\!\be}^{\!\phi}
\end{split}\qquad\forall\,\mu_2\!\in\!H^2(X),\EE
whenever these invariants are defined, $\be\!\neq\!0$, and $k\!\ge\!2$.\\

\noindent
By the linearity of real genus~0 GW-invariants and 
Theorem~\ref{evenvan_thm}\ref{van0_it}, 
it is sufficient to construct the maps~$P_{\be;\be'}$ on the direct sum 
of $(H^{2*}(X)^{\phi}_-)^{\otimes k}$ so that these maps satisfy~\eref{inveq_e0}
with $H^{2*}(X)$ replaced by $H^{2*}(X)^{\phi}_-$.
For \hbox{$\be,\be'\!\in\!H_{\eff}(X)_{\phi}$}, define
\begin{gather*}
P_{\be';\be}\!: H^{2*}(X)^{\phi}_-\lra H^{2*}(X)^{\phi}_-,\qquad
P_{\be';\be}(\mu)=\begin{cases}\mu,&\hbox{if}~\be'\!=\!\be;\\
0,&\hbox{if}\,\be'\!\neq\!\be;\end{cases}\\
P_{\be';\be}\!=\!0\!:
\bigoplus_{k=1}^{\i}\big(H^{2*}(X)^{\phi}_- \big)^{\otimes k} \lra H^{2*}(X)^{\phi}_-
\qquad\hbox{if}\quad \be\!-\!\be'\not\in H_{\eff}(X)_{\phi}\!\cup\!\{0\}\,. 
\end{gather*}
These linear maps satisfy the $k\!=\!1$ case of~\eref{inveq_e0} under 
the assumptions in~\ref{inveq_it1}
and of~\eref{inveq_e0}  with $\lr{\cdot}^{\phi,c}$ replaced by~$\lr{\cdot}^{\phi}$
under the assumptions in~\ref{inveq_it2}.\\

\noindent
Suppose $M\!\in\!\Z^+$ and for every pair $(k,\be)$ in $\Z^+\!\times\!H_{\eff}(X)_{\phi}$
with \hbox{$\lr{\be}_k\!<\!M$}  there exist linear maps
\BE{inveq_e9}P_{\be';\be}\!: \big(H^{2*}(X)^{\phi}_-\big)^{\otimes k}\lra H^{2*}(X)^{\phi}_-
\qquad\hbox{with}\quad \be'\!\in\!H_{\eff}(X)_{\phi}\EE
that satisfy~\eref{inveq_e0} under the assumptions in~\ref{inveq_it1}
and~\eref{inveq_e0}  with $\lr{\cdot}^{\phi,c}$ replaced by~$\lr{\cdot}^{\phi}$
under the assumptions in~\ref{inveq_it2}.
Let $(k,\be)$ be a pair in $\Z^+\!\times\!H_{\eff}(X)_{\phi}$ such that
$$k\!>\!1  \qquad\hbox{and}\qquad    
M\le\lr{\be}_k< M\!+\!1.$$
Choose a basis~$\cB_k$~for $(H^{2*}(X)^{\phi}_-)^{\otimes k}$
consisting of products of homogeneous elements 
(each factor~$\mu_i$ lies in $H^c(X)^{\phi}_-$ for  some $c\!\in\!2\Z^+$).\\

\noindent
Let $\mu_1\!\otimes\!\ldots\!\otimes\!\mu_k\!\in\!\cB_k$.
By the divisibility assumption, there exist 
$\mu\!\in\!H^{2*}(X)_+^{\phi}$ and $\mu_2'\!\in\!H^2(X)_-^{\phi}$ such that 
$\mu_2\!=\!\mu\mu_2'$.
For each $\be'\!\in\!H_{\eff}(X)_{\phi}$ such that 
$\be\!-\!\be'\!\in\!H_{\eff}(X)_{\phi}\!\cup\!\{0\}$, define
\begin{equation*}\begin{split}
P_{\be';\be}(\mu_1,\ldots,\mu_k)
=\lr{\mu_2',\be}P_{\be';\be}\big(\mu\mu_1,\mu_3,\ldots,\mu_k\big)+
\sum_{\begin{subarray}{c}\fd(\be_1)+\be_2=\be\\ 
\fd(\be_1),\be_2\in H_{\eff}(X)_{\phi}\end{subarray}}
\!\!\!\sum_{I\sqcup J=\{3,\ldots,k\}}\!\!\!
\sum_{\begin{subarray}{c}1\le i\le\ell\\ \ga^i\in H^{2*}(X)^{\phi}_-\end{subarray}}
\!\!\!\!\!\!\!\!\!\!\!2^{|I|}\!\Bigg(\quad&\\
\blr{\mu,\mu_1,\mu_I,\ga_i}_{\be_1}^X
\!P_{\be';\be_2}\!\big(\mu_2',\mu_J,\ga^i\big)
-\blr{\mu,\mu_2',\mu_I,\ga_i}_{\be_1}^X
\!P_{\be';\be_2}\big(\mu_1,\mu_J,\ga^i\big)&\Bigg).
\end{split}\end{equation*}
The values of $P_{\be';\be}$ and $P_{\be';\be_2}$ above are well-defined because
$$\lr{\be}_{k-1}=\lr{\be}_k\!-\!1<M, \quad
\lr{\be_2}_{|J|+2}=\lr{\be}_k-\om\big(\fd(\be_1)\big)-|I|
\le\lr{\be}_k\!-\!1 <M\,.$$

\vspace{.2in}

\noindent
By~\eref{inveq_e5b}, the equation 
in Theorem~\ref{main2_thm} with $\mu_2$ replaced by~$\mu_2'$ is equivalent~to
\begin{equation*}\begin{split}
\blr{\mu_1,\mu_2,\mu_3,\ldots,\mu_k}_{\be}^{\phi,c}
=\lr{\mu_2',\be}\blr{\mu\mu_1,\mu_3,\ldots,\mu_k}_{\be}^{\phi,c}
+\sum_{\begin{subarray}{c}\fd(\be_1)+\be_2=\be\\ \fd(\be_1),\be_2\in H_{\eff}(X)_{\phi} \end{subarray}}
\!\!\!\sum_{I\sqcup J=\{3,\ldots,k\}}\!\!\!
\sum_{\begin{subarray}{c}1\le i\le\ell\\ \ga^i\in H^{2*}(X)^{\phi}_-\end{subarray}}
\!\!\!\!\!\!\!\!\!\!\!2^{|I|}\!\Bigg(\quad&\\
\blr{\mu,\mu_1,\mu_I,\ga_i}_{\be_1}^X
\!\blr{\mu_2',\mu_J,\ga^i}_{\be_2}^{\phi,c}
-\blr{\mu,\mu_2',\mu_I,\ga_i}_{\be_1}^X
\!\blr{\mu_1,\mu_J,\ga^i}_{\be_2}^{\phi,c}&\Bigg)
\end{split}\end{equation*}
under the assumptions in~\ref{inveq_it1}
and  with $\lr{\cdot}^{\phi,c}$ replaced by~$\lr{\cdot}^{\phi}$
under the assumptions in~\ref{inveq_it2}.
Along with the assumption on~\eref{inveq_e9}, this implies that the elements
$P_{\be';\be}(\mu_1,\ldots,\mu_k)$ of $H^{2*}(X)^{\phi}_-$ satisfy~\eref{inveq_e0}
under the assumptions in~\ref{inveq_it1}
and~\eref{inveq_e0}  with $\lr{\cdot}^{\phi,c}$ replaced by~$\lr{\cdot}^{\phi}$
under the assumptions in~\ref{inveq_it2}.
This completes the inductive step of the~proof.
\end{proof}

\begin{proof}[{\bf\emph{Proof of Corollary~\ref{inveqCI_crl}}}]
By the discussion above the statement of this corollary and 
immediately after~\tauetacond{\tau} and~\tauetacond{\eta}
in Section~\ref{mainthm_sec}, the real genus~0 GW-invariants
$\lr{\ldots}_d^{\phi}$ and $\lr{\ldots}_d^{\phi,c}$ of~$(X,\phi)$ are defined
in all cases considered in the statement of Corollary~\ref{inveqCI_crl}.
By the sentence below~\eref{dimcong_e}, we can assume that the complex dimension
of~$X$ is odd.
By the Lefschetz Theorem on Hyperplane Sections \cite[p156]{GH} and Poincare Duality,
$H^{2*}(X)$ is then generated by~$H^2$ over~$\Q$ and 
$H^{4*}(X)\!=\!H^{4*}(X)^{\phi}_+$.
In light of Theorem~\ref{evenvan_thm}, this implies that 
the real genus~0 degree~$d$ GW-invariants of $(X,\om_n|_X,\phi)$
with any insertion $\mu_j\!\in\!H^{4*}(X)$ vanish.\\

\noindent
If $\phi_{\P^{n-1}}\!=\!\eta_n$ and $c\!=\!\tau$, the moduli spaces
in~\eref{fMbeX_e} are empty for any $J\!\in\!\cJ_{\om}^{\phi}$ 
because $X^{\phi}\!=\!\eset$.
The same is the case if ($\tau_n\eta$) holds because 
the involution~$\tau_n$ lifts to an involution~$\wt\tau_n$ on
the line bundle $\cO_{\P^{n-1}}(1)$,
a real degree~$d$ map from $(\P^1,\eta)$ to $(\P^{n-1},\tau_n)$ pulls back 
$(\cO_{\P^{n-1}}(1),\wt\tau_n)$ to a degree~$d$ line bundle over~$\P^1$
with an involution lifting~$\eta$,
and only even-degree line bundles over~$\P^1$ admit such lifts. 
This establishes the vanishing claim if either $\phi_{\P^{n-1}}\!=\!\eta_n$ and $c\!=\!\tau$
or ($\tau_n\eta$) holds.
The assumption that the degrees of~$\mu_j$ are even is not necessary in these~cases.\\

\noindent
By the real version of Quantum Lefschetz Hyperplane Theorem (as in \cite[Proposition~7.7]{Ge2}), 
the real genus~0 degree~$d$ GW-invariants of $(X,\om_n|_X,\phi)$ with insertions
$\mu_j\!=\!H^{c_j}$ for some $c_j\!\in\!\Z^{\ge0}$
are equal to the real genus~0 GW-invariants
of $(\P^{n-1},\phi_{\P^{n-1}})$ twisted by the Euler class of a vector bundle.
If either $a_i\!\in\!2\Z$ for some~$i$ or $d\!\in\!2\Z$ and $\ell\!\in\!\Z^+$,
then this bundle contains a subbundle of odd rank and
the invariants of $(X,\om_n|_X,\phi)$ vanish.\\

\noindent
Suppose $d\!\in\!2\Z$ and $\ell\!=\!0$, i.e.~$X_{n;\a}\!=\!\P^{n-1}$.
If $[\P^1]$ is the generator of $H_2(\P^{n-1})$, then
$$c_{\min}^{\phi}\big(d[\P^1]\big)=2n>(\dim_{\R}\!X)/2+1=n.$$
Thus, the real genus~0 degree~$d$ GW-invariants of $(\P^{n-1},\om_n,\phi_{\P^{n-1}})$
with $d\!\in\!2\Z$ vanish by the last statement of Corollary~\ref{inveq_crl2}.
This concludes the proof of Corollary~\ref{inveqCI_crl}\ref{CIvan_it}.\\

\noindent
It remains to establish Corollary~\ref{inveqCI_crl}\ref{CIred_it}.
We assume~that 
$$\dim_{\C}\!X=n\!-\!1\!-\!\ell\ge0.$$
Along with the assumption on~$|\a|$, this implies that $n\!>\!|\a|$.
By Corollary~\ref{inveqCI_crl}\ref{CIvan_it} and the reasoning above, 
we can also assume that $a_i\!\not\in\!2\Z$ for every~$i$, $d\!\not\in\!2\Z$,
and $X$ is odd-dimensional.
The last assumption implies that $H^{2*}(X)$ is generated by $H^2(X)^{\phi}_-$
as an algebra over~$\Q$, \hbox{$H_2(X)\!=\!H_2(X)_{\phi}$} is one-dimensional,
and $(X,\om_n|_X,\phi)$ is a real Fano symplectic manifold.
Along with the middle assumption, it implies~that 
\begin{equation*}\begin{split}
c_{\min}^{\phi}\big(d[\P^1]\big)&=n\!-\!|\a|=\lr{\a}_n-\big((\dim_{\R}\!X)/2\!-\!1\big),\\
c_+^{\phi}\big(d[\P^1]\big)&=3\big(n\!-\!|\a|\big)>n\!-\!\ell
=(\dim_{\R}\!X)/2+1.
\end{split}\end{equation*}
Corollary~\ref{inveqCI_crl}\ref{CIred_it} thus follows from 
the first statement of Corollary~\ref{inveq_crl2} and
the linearity of real genus~0 GW-invariants.
\end{proof}

\noindent
Analogously to the situation in complex GW-theory, 
Theorem~\ref{main2_thm} is related to the quantum cohomology of~$(X,\om)$.
Let~$(X,\om,\phi)$ be a real symplectic manifold. 
Suppose that either
\begin{enumerate}[label=(C\arabic*),leftmargin=*]

\item\label{Rinv_it} the conditions
\tauetacond{\tau} and \tauetacond{\eta} in Section~\ref{mainthm_sec} hold or 

\item\label{cinv_it} $c\!\in\!\{\tau,\eta\}$ is fixed and \tauetacond{c} holds.

\end{enumerate}
In the first case, let 
$$H_2(X)_{\phi}^{\star}=H_2(X)_{\phi}\!-\!\{0\}\,.$$
In the second case, let $H_2(X)_{\phi}^{\star}$ be as above if $X^{\phi}\!=\!\eset$ 
and $H_2(X)_{\phi}\!-\!\Im(\fd)$ if $X^{\phi}\!\neq\!\eset$.
For each $\be\!\in\!H_2(X)_{\phi}^{\star}$,
denote by $\lr{\ldots}_{\be}^{\phi}$ the real invariant~\eref{phinums_e2}
in the case~\ref{Rinv_it} and the real invariant~\eref{phinums_e} 
in the case~\ref{cinv_it}.\\

\noindent
Choose bases $\{\ga_i\}_{i\le\ell}$ and $\{\ga^i\}_{i\le\ell}$ for $H^*(X)$ so that 
$$\PD_{X^2}(\De_X)=\sum_{i=1}^{\ell}\ga_i\!\times\!\ga^i\in H^*(X^2),$$
as before.
Let $q$ denote
the formal variable in the Novikov ring~$\La$ on $H_2(X;\Z)$ and set
$$\wt\La=\La[q^{1/2}], \quad
\wt{Q}H^*(X)= H^*(X)\otimes \wt\La, \quad
\wt{Q}H^*(X)_{\pm}^{\phi}= H^*(X)_{\pm}^{\phi}\otimes \wt\La= 
QH(X)^{\phi}_{\pm}[q^{1/2}]\,;$$
see \cite[Section~11.1]{MS}.
We define a homomorphism of modules over~$\ti\La$~by
$$\fR_{\phi}\!: \wt{Q}H^*(X)\lra \wt{Q}H^*(X)^{\phi}_{(-1)^{n+1}}, \quad
\fR_{\phi}\mu=\sum_{\be\in H_2(X)_{\phi}^{\star}}
\sum_{i=1}^{\ell}
\blr{\mu,\ga_i}_{\be}^{\phi}\ga^iq^{\be/2}
\quad\forall~\mu\!\in\!H^*(X)\,,$$
where $2n\!=\!\dim X$.
By Theorems~\ref{evenvan_thm} and~\ref{main2_thm}, 
$$\fR_{\phi}\mu=0\quad\forall~\mu\in \wt{Q}H^*(X)^{\phi}_+
\qquad\hbox{and}\qquad
\fR_{\phi}\mu_1*\mu_2=\mu_1*\fR_{\phi}\mu_2\quad\forall~\mu_1,\mu_2\in \wt{Q}H^*(X)^{\phi}_-\,,$$
respectively, where $*$ is the quantum product.
If in addition $\lr{c_1(X),\be}\!\in\!2\Z$ for all $\be\!\in\!H_2(X)$ that can
be represented by $J$-holomorphic spheres for a generic $J\!\in\!\cJ_{\om}^{\phi}$, then 
$$\fR_{\phi}\mu_-*\mu_+=\fR_{\phi}\big(\mu_-*\mu_+\big)
\qquad\forall~\mu_-\in\wt{Q}H^*(X)^{\phi}_-\,, \mu_+\in\wt{Q}H^*(X)^{\phi}_+\,;$$
this can be seen by an argument similar to the proof of Proposition~\ref{twistGW_prp}.\\

\begin{table}
\begin{center}
\begin{tabular}{||c|c||}
\hline\hline
$d$& $N_d^{\R}$\\
\hline
1& 1\\
\hline
3& 1\\
\hline
5& 5\\
\hline
7& 85\\
\hline
9& 1993\\
\hline
11& 136457\\
\hline
13& 3991693\\
\hline
15& 1580831965\\
\hline
17& -129358296175\\
\hline
19& 106335656443537\\
\hline
21& -39705915765949931\\
\hline
23& 27364388694945255653 \\
\hline
25&  -19263282511829476981415\\
\hline
27& 17458116427845844069499545\\
\hline
29&  -18101279473337469331178336611\\
\hline
31& 22138019795038729862257691515501\\
\hline\hline
\end{tabular}
\end{center}
\caption{The number $N_d^{\R}$ of degree~$d$ real rational curves 
through $d$ non-real points in~$\P^3$.}
\label{P3nums_tbl}
\end{table}

\noindent
We conclude with some counts of real curves in~$\P^3$, $\P^5$, and~$\P^7$;
see Tables~\ref{P3nums_tbl} and~\ref{Pnnums_tbl}.
These numbers are consistent with basic algebro-geometric 
considerations \cite[p177]{GH}.
\begin{enumerate}[label=(\arabic*),leftmargin=*]
\item Every degree~1 curve lies in a~$\P^1$, 
every non-real point~$p$ in $\P^{2n-1}$ determines a real $\P^1\!\subset\!\P^{2n-1}$,
and a real line passing through~$p$ lies in this~$\P^1$.
Thus, $N_1^{\R}$, $\lr{5^13^0}_1^{\tau_5}$, and $\lr{7^15^03^0}_1^{\tau_7}$
should equal~1, at least in the absolute value.

\item Every degree~3 curve lies in a~$\P^3$, 
every two general non-real points~$p_1$ and~$p_2$ in $\P^{2n-1}$ determine 
a real $\P^3\!\subset\!\P^{2n-1}$, for $n\!\ge\!2$, and
a real degree~3 curve passing through~$p_1$ and~$p_2$ lies in this~$\P^3$.
Thus, a real degree~3 curve in~$\P^5$ passing through two general points~$p_1$ and~$p_2$
and a general plane~$\pi$ lies in the real~$\P^3$ determined by these 
two points and passes through the point $\pi\!\cap\!\P^3$; so the number 
$\lr{5^23^1}_3^{\tau_5}$ should equal~$N_3^{\R}$, at least in the absolute value.
By the same reasoning, the number $\lr{7^25^03^1}_3^{\tau_7}$
should also equal~$N_3^{\R}$.

\item Every degree~5 curve lies in a~$\P^5$,
every three non-real points~$p_1$, $p_2$, and~$p_3$ in $\P^7$ determine 
a real~$\P^5$,  and
a real degree~5 curve passing through~$p_1$, $p_2$, and~$p_3$ lies in this~$\P^5$.
Thus, the numbers $\lr{7^35^13^0}_5^{\tau_7}$ and $\lr{7^35^03^2}_5^{\tau_7}$
should equal 
$\lr{5^43^0}_5^{\tau_5}$ and $\lr{5^33^2}_5^{\tau_5}$, respectively.
\end{enumerate}

\begin{table}
\begin{center}
\begin{tabular}{||c|c|c||}
\hline\hline
$d$& cond& $\lr{5^a3^b}_d^{\tau_5}$\\
\hline
1& $5^1 3^0$& 1\\
\hline
1& $5^0 3^2$& 1\\
\hline
3& $5^2 3^1$& -1\\
\hline
3& $5^1 3^3$& -3\\
\hline
3& $5^0 3^5$& -5\\
\hline
5& $5^4 3^0$& 1\\
\hline
5& $5^3 3^2$& 1\\
\hline
5& $5^2 3^4$& -7\\
\hline
5& $5^1 3^6$& 93\\
\hline
5& $5^0 3^8$& 12417\\
\hline
7& $5^5 3^1$& -23\\
\hline
7& $5^4 3^3$& -213\\
\hline
7& $5^3 3^5$& -2679\\
\hline
7& $5^2 3^7$& -23001\\
\hline
7& $5^1 3^9$& 874089\\
\hline
7& $5^0 3^{11}$& 90271011\\
\hline
9& $5^7 3^0$& 21\\
\hline
9& $5^6 3^2$& -503\\
\hline
9& $5^5 3^4$& -16399\\
\hline
9& $5^4 3^6$& -394863\\
\hline
9& $5^3 3^8$& -6924579\\
\hline
9& $5^2 3^{10}$& 69060873\\
\hline
9& $5^1 3^{12}$& 19824606009\\
\hline
9& $5^0 3^{14}$& 1811570349393\\
\hline\hline
\end{tabular}
\hspace{1in}
\begin{tabular}{||c|c|c||}
\hline\hline
$d$& cond& $\lr{7^a5^b3^c}_d^{\tau_7}$\\
\hline
1& $7^1 5^0 3^0$& 1\\
\hline
1& $7^0 5^1 3^1$& 1\\
\hline
1& $7^0 5^0 3^3$& 1\\
\hline
3& $7^2 5^0 3^1$& -1\\
\hline
3& $7^1 5^2 3^0$& -1\\
\hline
3& $7^1 5^1 3^2$& -3\\
\hline
3& $7^1 5^0 3^4$& -5\\
\hline
3& $7^0 5^3 3^1$& -3\\
\hline
3& $7^0 5^2 3^3$& -1\\
\hline
3& $7^0 5^1 3^5$& 89\\
\hline
3& $7^0 5^0 3^7$& 1155\\
\hline
5& $7^3 5^1 3^0$& 1\\
\hline
5& $7^3 5^0 3^2$& 1\\
\hline
5& $7^2 5^2 3^1$& -3\\
\hline
5& $7^2 5^1 3^3$& -27\\
\hline
5& $7^2 5^0 3^5$& -175\\
\hline
5& $7^1 5^4 3^0$& -11\\
\hline
5& $7^1 5^3 3^2$& -71\\
\hline
5& $7^1 5^2 3^4$& -239\\
\hline
5& $7^1 5^1 3^6$& 2181\\
\hline
5& $7^1 5^0 3^8$& 75405\\
\hline
5& $7^0 5^5 3^1$& -55\\
\hline
5& $7^0 5^4 3^3$& 349\\
\hline
5& $7^0 5^3 3^5$& 20589\\
\hline
5& $7^0 5^2 3^7$& 438481\\
\hline
5& $7^0 5^1 3^9$& 7937169\\
\hline
5& $7^0 5^0 3^{11}$& 139758309\\
\hline\hline
\end{tabular}
\end{center}
\caption{The numbers $\lr{5^a3^b}_d^{\tau_5}$ and $\lr{7^a5^b3^c}_d^{\tau_7}$
of degree~$d$ real rational curves  through $a$~non-real points and 
$b$~non-real planes in~$\P^5$
and through $a$ non-real points, $b$ non-real planes, and 
$c$ non-real linear $\P^4$'s in~$\P^7$, respectively.}
\label{Pnnums_tbl}
\end{table}

\clearpage

\vbox{
\noindent
{\it Department of Mathematics, Princeton University, Princeton, NJ 08544}\\
{\it Current address: Institut de Math\'ematiques de Jussieu - Paris Rive Gauche,
Universit\'e Pierre et Marie Curie,  4~Place Jussieu,
75252 Paris Cedex 5, France\\
penka.georgieva@imj-prg.fr}\\

\noindent
{\it Department of Mathematics, Stony Brook University, Stony Brook, NY 11794\\
azinger@math.stonybrook.edu}}

\end{document}